\documentclass[11pt]{amsart}

\usepackage{amsthm}
\usepackage{amsfonts}
\usepackage{amsmath}
\usepackage{amssymb}
\usepackage{mathrsfs}
\usepackage{enumitem}
\usepackage{fancyhdr}
\usepackage{tikz}
\usepackage{tikz-cd}
\usepackage{hyperref}
\usepackage{lmodern}

\usetikzlibrary{arrows,decorations.markings}

\usepackage[numbers]{natbib}

\usepackage{mathtools}

\setlist[description]{leftmargin=\parindent,labelindent=\parindent}

\usepackage{graphicx}
\graphicspath{ {images/} }

\usepackage[colorinlistoftodos]{todonotes} 

\newtheorem{innercustomthm}{Assumption}
\newenvironment{customthm}[1]
  {\renewcommand\theinnercustomthm{#1}\innercustomthm}
  {\endinnercustomthm}

\newtheorem{theorem}{Theorem}
\newtheorem{lemma}[theorem]{Lemma}
\newtheorem{proposition}[theorem]{Proposition}
\newtheorem{corollary}[theorem]{Corollary}
\newtheorem{conjecture}[theorem]{Conjecture}
\numberwithin{theorem}{section}
\numberwithin{equation}{section}

\theoremstyle{definition}
\newtheorem{definition}[theorem]{Definition}

\newtheorem{remark}[theorem]{Remark}

\newtheorem{assumption}[theorem]{Assumption}

\title{On intrinsic homological mirror symmetry for toric degenerations}
\author{Shaoyun Bai \and Sebastian Haney}

\DeclareMathOperator\Hom{Hom}

\DeclareMathOperator\Ext{Ext}
\DeclareMathOperator\End{End}

\DeclareMathOperator\id{id}

\DeclareMathOperator\Arg{Arg}

\DeclareMathOperator\Log{Log}

\DeclareMathOperator\pt{pt}
\DeclareMathOperator\Coh{Coh}

\DeclareMathOperator\ev{ev}

\DeclareMathOperator\ind{ind}
\DeclareMathOperator\sing{sing}
\DeclareMathOperator\Proj{Proj}
\DeclareMathOperator\Fuk{Fuk}
\DeclareMathOperator\Perf{Perf}
\DeclareMathOperator\Spec{Spec}
\DeclareMathOperator\Sk{\mathsf{Sk}}
\DeclareMathOperator\SR{\mathsf{SR}}
\DeclareMathOperator\Diff{Diff}

\DeclareMathOperator\Mod{Mod}
\DeclareMathOperator\PSS{\mathsf{PSS}}
\DeclareMathOperator\SSP{\mathsf{SSP}}
\DeclareMathOperator\res{\mathsf{res}}
\DeclareMathOperator\Ver{\mathsf{Vert}}
\DeclareMathOperator\Edge{\mathsf{Edge}}
\DeclareMathOperator\Leaf{\mathsf{Leaf}}
\DeclareMathOperator\cont{\mathsf{con}}
\DeclareMathOperator\ord{\mathsf{ord}}
\DeclareMathOperator\plog{plog}

\DeclareMathOperator\Bad{\mathsf{Bad}}

\allowdisplaybreaks
\begin{document}
\begin{abstract}
This paper studies the Floer-theoretic aspects of homological mirror symmetry inspired by proposals of Perutz and Siebert and the Gross--Siebert intrinsic mirror symmetry program. Given a maximally unipotent degeneration of smooth projective Calabi--Yau manifolds over the punctured disk, we construct a ring using the fixed point Floer cohomology groups of the iterates of the monodromy of the degeneration equipped with the pair of pants product.  Under the assumption that this ring is commutative, we can consider a candidate mirror family defined by the relative Proj construction. Further assuming that a smooth fiber $X_t$ contains a so-called \textit{tropical Lagrangian section}, we construct a fully faithful embedding from the derived category of perfect complexes on our candidate mirror family into the Fukaya category of $X_t$. We verify both of these assumptions for certain Batyrev--Borisov toric degenerations, as well as some toric degenerations of Calabi--Yau threefolds coming from the Gross--Siebert reconstruction algorithm. These two geometric hypotheses are both phrased to support the general study of mirror symmetry for maximally unipotent degenerations of Calabi--Yau manifolds, largely reducing the symplectic inputs for proving homological mirror symmetry to the problem of constructing tropical Lagrangian sections.
\end{abstract}
\maketitle
\tableofcontents
\section{Introduction}
Kontsevich's homological mirror symmetry conjecture, as stated in~\cite{KontsevichHMS}, says that for certain symplectic Calabi--Yau manifolds $(X,\omega)$, meaning that $c_1(X) = 0$, for which there is a(n unspecified) \textit{mirror} complex algebraic variety $\check{X}$, one should expect that the derived Fukaya category of $X$ is equivalent to the derived category of coherent sheaves on the mirror. The original cases in which mirror symmetry phenomena were observed in string theory involved Calabi--Yau manifolds which arise as smooth fibers of \textit{maximally unipotent degenerations}~\cite{GreenePlesser, COGP}. Such Calabi--Yau manifolds are also said to live \textit{near a large complex structure limit point}~\cite[Definition 16.25]{GrossJoyceHuybrechts}, and indeed this was classically considered a precondition for the existence of a mirror. Based on string-theoretic heuristics, mirror symmetry was expected to manifest as an involutive duality between a Calabi--Yau manifold near the large complex structure limit, viewed as an algebraic variety (the B-model), and a mirror Calabi--Yau manifold equipped with a so-called large K{\"a}hler class, viewed as a symplectic manifold (the A-model.)

Some of the first broad mathematical constructions of mirror pairs were the constructions of Calabi--Yau hypersurfaces and complete intersections in toric Fano varieties by Batyrev~\cite{Batyrev} and Batyrev--Borisov~\cite{BatyrevBorisov}, respectively. The Strominger--Yau--Zaslow (SYZ) conjecture proposes a far-reaching generalization of this picture, positing that mirror pairs of Calabi--Yau manifolds should carry dual special Lagrangian torus fibrations over smooth manifolds with dual (singular) integral affine structures~\cite{SYZ}. To circumvent some considerable analytic challenges presented by the SYZ conjecture, which have been partially addressed for non-Archimedean versions of the conjecture in~\cite{KS06, LiSYZ} \textit{inter alia}, Gross and Siebert proposed an algebro-geometric interpretation of SYZ mirror symmetry, leading to a very general construction of candidate mirror pairs~\cite{GSrealaffine} as \textit{toric degenerations}, initiating what is now called the \textit{Gross--Siebert program}.

Partly inspired by a conjectural picture of homological mirror symmetry for toric degenerations in the context of their program~\cite{GHS}, Gross and Siebert went on to formulate an essentially completely general \textit{intrinsic mirror} construction in~\cite{GSintrinsicmirrors}, which associates a candidate \textit{intrinsic mirror family} to any maximally unipotent degeneration of Calabi--Yau manifolds, without recourse to toric geometry or Lagrangian torus fibrations. A key technical innovation required to pose the intrinsic mirror construction was the introduction of punctured Gromov--Witten invariants in the work of Abramovich--Chen--Gross--Siebert~\cite{ACGS2}. Gross and Siebert~\cite{GSintrinsicmirrors} associate to a log smooth degeneration $\mathscr{X}\to\mathbb{D}$ of smooth projective Calabi--Yau manifolds over the disk a $\Bbbk[\![T]\!]$-algebra\footnote{This is obtained by applying a change of coefficients to the ring constructed in~\cite{GSintrinsicmirrors}, which is initially defined over a larger base.}, where $\Bbbk$ is a field of characteristic zero, called the \textit{intrinsic mirror ring} that we denote $R_{\mathsf{GS}}$, and form the candidate mirror space
\begin{align}\label{GSintrinsicmirrorfamily}
    \check{\mathcal{X}}_{\mathsf{GS}}\coloneqq\underline{\Proj}_{\Bbbk[\![T]\!]}R_{\mathsf{GS}}
\end{align}
by a relative Proj construction. The structure coefficients of this ring are given by punctured Gromov--Witten invariants, and thus the intrinsic mirror family is defined entirely by A-model considerations.

Concurrently, a great deal of progress has been made towards proving homological mirror symmetry for special cases of the Batyrev--Borisov construction~\cite{SeiQuartic, SheridanHypersurfaces, SheridanSmith, GHHPS} and for K3 surfaces~\cite{HackingKeating}. This latter development uses heuristics from the Gross--Siebert program. So far, these approaches all either use the ambient toric geometry in a nontrivial way, or they rely on the existence of a global (singular) Lagrangian torus fibration. In particular, it is not clear that these techniques will be available, in their most potent incarnations, to study of homological mirror symmetry for intrinsic mirror pairs, or even for higher-dimensional toric degenerations.

In several talks over a number of years, Bernd Siebert has outlined an elegant strategy, being developed with Tim Perutz, for proving homological mirror symmetry between smooth fibers of maximally unipotent degenerations (viewed on the A-side), and intrinsic mirror families thought of as projective varieties over a Novikov field. One of the main goals of this program is to provide a Floer-theoretic description of the intrinsic mirror ring, as described in the following conjecture.
\begin{conjecture}[Perutz--Siebert]\label{GSringisomorphism}
    Let $R_{\mathsf{GS}}$ denote the intrinsic mirror ring associated to a polarized maximally unipotent degeneration of smooth Calabi--Yau manifolds $\mathcal{X}^{\circ}\to\mathbb{D}^{\circ}$ over the punctured disk $\mathbb{D}^{\circ}$, which has a semistable extension $\mathscr{X}\to\mathbb{D}$ over the disk. If $X_t$ is a smooth fiber of $\mathcal{X}^{\circ}$ and $\phi\colon X_t\to X_t$ denotes the symplectic monodromy of this degeneration, then there is an isomorphism
    \begin{align*}
        R_{\mathsf{GS}}\cong R_{\phi}\coloneqq\bigoplus_{d=0}^{\infty} HF^0(X_t,\phi^d)
    \end{align*}
    where the direct sum on the right-hand side is of fixed point Floer cohomology groups, and $R_{\mathsf{GS}}$ is associated to $\mathscr{X}\to\mathbb{D}$.\footnote{The definition of $R_{\phi}$ is reviewed more carefully \S{\ref{mainsection}}, \S{\ref{fixedpointfloersection}}, and \S{\ref{enhancedfixedpointfloersection}}. In the last of these sections we explain how to view $R_{\phi}$ is a module over the ring of formal power series.}
\end{conjecture}
The method proposed by Perutz and Siebert for proving this conjecture would involve constructing a symplectic version of the punctured Gromov--Witten invariants, which is the subject of ongoing research.

In this paper, we initiate a study of the Floer-theoretic aspects of Perutz and Siebert's program. Our aim is to separate the salient problems in Lagrangian Floer theory and fixed point Floer cohomology from the hard, and important, problem of constructing punctured invariants for symplectic manifolds, allowing for the further study of homological mirror symmetry by relatively elementary methods. One of our main results is the following.
\begin{theorem}\label{toric-degeneration-hms-thm}
    Let $\mathcal{X}\to\mathbb{D}$ be a toric degeneration whose total space has strongly semi-simple singularities (see~\cite[Definition 2.10]{Arguzreal}), and whose central fiber has \textit{positive real gluing data} (see Theorem~\ref{lagrangianpositivereallocus}.) If $\mathcal{X}$ also admits a suitable resolution of singularities (see Assumption~\ref{toricdegenerationresolutionassumption}), then $R_{\phi}$ is commutative and finitely-generated, and there is an $A_{\infty}$-functor
    \begin{align}\label{mirrorinclusion0}
        \Perf(\underline{\Proj}_{\Bbbk[\![T]\!]}R_{\phi})\otimes_{\Bbbk[\![T]\!]}\Lambda_{\Bbbk}\to\Fuk(X_t)
    \end{align}
    between functors over the universal Novikov field $\Lambda_{\Bbbk}$, where the right hand side denotes the (relative) Fukaya category of a smooth fiber of the degeneration. This functor is a quasi-equivalence onto its image.
\end{theorem}
This can be interpreted as proving a half of homological symmetry. It is not \textit{a priori} clear that the mirror family $\underline{\Proj}_{\Bbbk[\![T]\!]}R_{\phi}$ is smooth, but some evidence that it is smooth in the case of simple toric degenerations, based on Conjecture~\ref{GSringisomorphism}, can be found in e.g.~\cite{GScanonicalwall, Goncharov}. For this reason, we cannot yet apply the standard generation criteria~\cite{AbouzaidGeneration, PerutzSheridanCore, GanGeneration, Sanda, SheridanOC, AFOOO}, but we expect that this gap can eventually be bridged, at least for simple toric degenerations.
\begin{corollary}\label{bb3corollary}
    There is an inclusion of $A_{\infty}$-categories as in~\eqref{mirrorinclusion0} for toric degenerations of the following types (on the A-side)
    \begin{itemize}
        \item[(i)] Maximal partial crepant projective (MPCP) Batyrev--Borisov degenerations constructed by Gross~\cite{GrossBB} (with strongly semi-simple singularities~\cite[Definition 2.10]{Arguzreal}), and;
        \item[(ii)] toric degenerations over the disk of relative dimension $3$ satisfying Assumption 6.1 of ~\cite{Goncharov}.
    \end{itemize}
\end{corollary}

\begin{remark}
The resolutions of singularities we need were already considered in the literature: by Yamamoto~\cite{Yamamoto} in the Batyrev--Borisov case and by Goncharov~\cite{Goncharov} in the threefold case. These results suggest that our methods can be applied to a much wider class of toric degenerations by carefully constructing resolutions of singularities for their total spaces, and we are very hopeful that this can be feasibly carried out. The notion of \textit{strong semi-simple singularities} is introduced in~\cite{Arguzreal}, and is a strengthening of the familiar notion of simplicity~\cite[Definition 1.60]{GS1} for the intersection complex of a toric degeneration. This constraint is presumably removable (see Remark~\ref{stronglysemisimpleremark}.) If one restricts to degenerations with three-dimensional fibers, then simple toric degenerations are automatically strongly semi-simple, as pointed out in~\cite{Arguzreal}.
\end{remark}

In the cases considered by Ganatra--Hanlon--Hicks--Pomerleano--Sheridan~\cite{GHHPS}, we can show that the mirror spaces we construct recover the Batyrev mirror space, providing further evidence that Gross and Siebert's intrinsic mirror family, as well as its analogue in fixed point Floer cohomology, define the expected mirror space.
\begin{theorem}\label{comparingmirrors}
    Let $\mathcal{X}_{\Delta}\to\mathbb{D}$ be a Batyrev degeneration of relative dimension at least $3$ of hypersurfaces in the toric variety associated to the reflexive polytope $\Delta$. Suppose that the fan $\Sigma^*$ dual to the polar dual $\Delta^*$ of $\Delta$ is smooth, that the MPCS~\cite[Definition 1.3]{GHHPS} and connectedness~\cite[Definition 1.6]{GHHPS} conditions hold, and that the ground field $\Bbbk$ has characteristic as specified in~\cite[Definition 1.4]{GHHPS}. Then, after base-changing to the Novikov \textbf{field}, we have that
    \begin{align*}
        \mathcal{X}_{\Delta^*}\cong\underline{\Proj}_{\Bbbk[\![T]\!]}R_{\phi}
    \end{align*}
    where $\mathcal{X}_{\Delta^*}$ is the Batyrev mirror space specified in~\cite[\S{1.3} and Conjecture A]{GHHPS}.
\end{theorem}
We remark that the proof of this theorem relies on the main result of~\cite{GHHPS}, so we do not claim an independent proof of homological mirror symmetry in these cases. One can, however, take this as evidence for the equivalence of the mirror constructions in~\cite{GSrealaffine} and~\cite{GSintrinsicmirrors} in some situations, along the lines of~\cite{GScanonicalwall, Goncharov}.

The next section, which is both an extended introduction and the conceptual heart of the paper, details a version of the Perutz--Siebert proposal with our modifications. Along the way, we will also summarize the technical contributions in this paper. We have opted to phrase our argument as the construction of a mirror functor for an arbitrary smooth fiber of a polarized maximally unipotent degeneration of Calabi--Yau manifolds, contingent on two geometric hypotheses, Assumptions~\ref{commutativeassumption} and~\ref{lagrangianassumption}. The first of these assumptions asserts that $R_{\phi}$ is commutative, and the second posits the existence of a certain Lagrangian submanifold out of which one can build mirrors to line bundles in the Fukaya category.

We hope that both of these working assumptions are stated in a way that supports further investigation, and we expect that they can be verified very generally in relatively short order. We are currently working to verify Assumption~\ref{commutativeassumption} for any maximally unipotent degeneration in joint work in-progress with Roman Krutowski. We expect that this would amount to a proof of Conjecture~\ref{GSringisomorphism} with the intrinsic mirror ring replaced by a version based on Gromov--Witten theory for symplectic orbifolds, inspired by the work of Tseng and You~\cite{TsengYou}. On the other hand, we are also working to verify Assumption~\ref{lagrangianassumption}, and to investigate the smoothness of our mirror family, for maximally unipotent degenerations of Calabi--Yau threefolds: the original case of mirror symmetry. 

\subsection*{Acknowledgments}
The first-named author was supported by the NSF standard grant DMS-2404843 and CAREER grant DMS-2540393. He thanks Daniel Pomerleano for encouragements and illuminating discussions. The second-named author was supported by the NSF MSPRF award DMS-2502860. He thanks Denis Auroux, Sheel Ganatra, and Roman Krutowski for helpful discussions. Both authors thank Paul Hacking for a discussion about~\cite{HackingKeating}. We also thank Roman Krutowski for comments on a version of this paper.

\section{Constructing a mirror functor}\label{mainsection}
Let $\mathbb{D}$ denote a disk in $\mathbb{C}$ centered at the origin, and let $\mathbb{D}^{\circ}$ denote the punctured disk $\mathbb{D}\setminus\lbrace 0\rbrace$. Consider a relatively minimal, maximally unipotent family of smooth closed Calabi--Yau $n$-folds $\pi^{\circ}\colon\mathcal{X}^{\circ}\to\mathbb{D}^{\circ}$. We assume that the fibers $X_t\coloneqq\pi^{-1}(t)$ are Calabi--Yau manifolds in the strict sense, meaning that they are simply connected, that $c_1(X_t) = 0$, and that $h^{i,0}(X_t) = 0$ for all $0<i<n$. We further assume that $\mathcal{X}^{\circ}$ is \textit{polarized}, i.e. equipped with a relatively ample line bundle
\begin{align}\label{polarizationassumption}
    \mathcal{E}\to\mathcal{X}^{\circ} \,.
\end{align}

In this section, we will outline a proof of homological mirror symmetry for a smooth fiber of such a family, modulo generation of the Fukaya category, equipped with a K\"{a}hler form determined by $\mathcal{E}$, over a ground field $\Bbbk$, by way of an extended introduction. Our proof assumes the existence of a suitable Lagrangian $\Bbbk$-homology sphere in a smooth fiber $X_t$, and thus applies to prove homological mirror symmetry for $X_t$ over ground fields of characteristic away from a set of primes $\Bad(\mathcal{X}^{\circ})$ defined in Definition~\ref{tropicallagrangiandefinition} which depends only on $\mathcal{X}^{\circ}$. We call such a Lagrangian submanifold $X_t$ a \textit{tropical Lagrangian section}. For toric degenerations defined over $\mathbb{R}$ (whose central fibers satisfy a certain cohomological constraint and whose total spaces admit sufficiently well-behaved resolutions) such a tropical Lagrangian section is given by the \textit{positive real locus} (see Theorem~\ref{lagrangianpositivereallocus}). This encompasses the case of toric degenerations of Batyrev--Borisov complete intersections.

We construct an \textit{intrinsic} mirror family using the fixed-point Floer cohomology groups of the iterates of the monodromy of $\mathcal{X}^{\circ}$. A version of the proof strategy outlined below was first envisioned by Perutz and Siebert.\footnote{Notably, Perutz and Siebert suggested studying the ring $R_{\phi}$ in this setting, and they anticipated the appeal to~\cite{Polishchuk}.} Our results apply without the assumption that our mirror family is smooth, although this means we cannot appeal to standard generation criteria~\cite{AbouzaidGeneration, PerutzSheridanCore, GanGeneration, Sanda, SheridanOC, AFOOO}. We also work directly with a ring defined using fixed-point Floer cohomology, at the expense of introducing Assumption~\ref{commutativeassumption} below.

The restrictions $\mathcal{E}_t$ of $\mathcal{E}$ to the fibers $X_t$ determine a locally constant family $\kappa$ of K{\"a}hler classes $\kappa_t\in H^2(X_t;\mathbb{R})$~\cite[\S{3.3.1}]{VoisinBook}. There is a smoothly varying family of K\"{a}hler forms $\omega_t\in\Omega^{1,1}(X_t)$, for $t\in\mathbb{D}^{\circ}$, representing the K\"{a}hler classes $\kappa_t$~\cite[\S{9.3.3}]{VoisinBook}. By Moser's theorem, the resulting family of symplectic manifolds is locally trivial. Thus, there is a monodromy representation
\begin{align}\label{symplecticmonodromyrep}
\pi_1(\mathbb{D}^{\circ},t)\to\pi_0\operatorname{Symp}(X_t,\omega_t) \,.
\end{align}
If $\mathcal{X}^{\circ}\to\mathbb{D}^{\circ}$ is equipped with the Ehresmann connection determined by the symplectic forms $\omega_t$, then these symplectomorphisms are obtained by parallel transport~\cite[Lemma 7.2]{SeiThesis}.
\begin{definition}
Let $\phi\colon X_t\to X_t$ denote a representative of the class in $\pi_0\operatorname{Symp}(X_t,\omega_t)$ corresponding to the image of a standard generator of $\pi_1(\mathbb{D}^{\circ},t)$ in~\eqref{symplecticmonodromyrep}.
\end{definition}

Consider the degree $0$ fixed-point Floer cohomology groups $HF^0(X_t,\phi^d)$ associated to the symplectomorphisms $\phi^d\colon X_t\to X_t$. We will review the construction of fixed-point Floer cohomology in the Calabi--Yau setting in \S{\ref{fixedpointfloersection}}. These groups are \textit{a priori} modules over a version of the Novikov field~\cite[\S{4}]{HoferSalamonNovikov}. Following~\cite{SeiThesis}, we will usually take this to be the \textit{universal Novikov field}
\begin{align}\label{universalnovikovfield}
\Lambda_{\Bbbk}\coloneqq\left\lbrace\sum_{i=0}^{\infty}a_i T^{\lambda_i}\mid a_i\in\Bbbk,\,\lambda_i\in\mathbb{R}\text{ and, }\lim\lambda_i = \infty\right\rbrace
\end{align}
with coefficients in $\Bbbk$. We write $\Lambda$ instead of $\Lambda_{\Bbbk}$ when the ground field is clear.

We will show that $HF^0(X_t,\phi^d)$ is naturally defined -- as an invariant of some \textit{snc model} of $\mathcal{X}^{\circ}$ -- with coefficients in the ring of Laurent series $\Bbbk[\![T]\!]$ (see \S{\ref{semistablereductionsection}} and \S{\ref{enhancedfixedpointfloersection}} below.) There is a pair of pants product
\begin{align}\label{pairofpantsproduct}
\star\colon HF^0(X_t,\phi^{d_0})\widehat{\otimes}_{\Bbbk[\![T]\!]} HF^0(X_t,\phi^{d_1})\to HF^0(X_t,\phi^{d_0+d_1})
\end{align}
for any $d_0,d_1\in\mathbb{Z}_{\geq0}$.\footnote{We will only consider the degree $0$ fixed point Floer cohomology for nonnegative powers of $\phi$, though it is a standard fact that the product is defined in all degrees and for all powers of $\phi$ more generally.} The grading on Floer cohomology is determined by a trivialization of the canonical bundle on $\mathcal{X}^{\circ}$ (which we give more details about below in Remark~\ref{CYformremark}.)
\begin{definition}
Define the $\Bbbk[\![T]\!]$-algebra $R_{\phi}$ by
\begin{align}\label{mirrorring}
R_{\phi}\coloneqq \bigoplus_{d\geq0}HF^0(X_t,\phi^d)
\end{align}
with multiplication given by~\eqref{pairofpantsproduct}. (Here the fixed-point Floer cohomology groups refer to the versions described in \S{\ref{enhancedfixedpointfloersection}}.)
\end{definition}
We would like to define the mirror to $\mathcal{X}^{\circ}$ by applying the relative Proj construction to $R_{\phi}$ over $\Bbbk[\![T]\!]$, but it is not immediately clear that this makes sense. As a starting point for our discussion, we will assume in this paper that:
\begin{customthm}{C}\label{commutativeassumption}
$R_{\phi}$ is a commutative ring.
\end{customthm}
We will show below that $R_{\phi}$ is also finitely-generated over $\Bbbk[\![T]\!]$ (Theorem~\ref{stanleyreisnerisom} and Remark~\ref{centralfiberrmk}.)
\begin{remark}
The discussion in~\cite[\S{5.4}]{PascaleffMonoidal} shows that Assumption~\ref{commutativeassumption} is necessary. In the important special case where $\mathcal{X}^{\circ}$ is (the complement of the central fiber in) a toric degeneration defined over $\mathbb{R}$, we will use the real involution to check that Assumption~\ref{commutativeassumption} holds. This is sufficient for studying, say, toric degenerations of Batyrev--Borisov complete intersections. In forthcoming work with Roman Krutowski, we intend to show that $R_{\phi}$ is commutative for any maximally unipotent degeneration $\mathcal{X}^{\circ}\to\mathbb{D}^{\circ}$ by comparing it to a version of the intrinsic mirror ring of~\cite{GSintrinsicmirrors} defined using orbifold Gromov--Witten invariants following~\cite{TsengYou} in the symplectic category.
\end{remark}
\begin{definition}
Define the \textit{(intrinsic) mirror family}
\begin{align}\label{mirrorfamily}
\check{\mathcal{X}}\coloneqq\underline{\mathbf{Proj}}_{\Bbbk[\![T]\!]}R_{\phi}
\end{align}
over $\operatorname{Spec}\Bbbk[\![T]\!]$ (where $R_\phi$ is thought of as a quasi-coherent $\mathcal{O}_{\operatorname{Spec}\Bbbk[\![T]\!]}$-module.)
\end{definition}
\begin{remark}
    Defining $R_{\phi}$ over $\Bbbk[\![T]\!]$ means that the mirror family $\mathcal{\check{X}}$ has a $\Bbbk[\![T]\!]$-model, and hence comes with a central fiber $\check{X}_0$. Working with the special fiber instead of $\check{\mathcal{X}}$ is helpful in some places, though we emphasize that the correct statement of homological mirror symmetry should be an equivalence of categories of the Novikov field.
\end{remark}

To state homological mirror symmetry, we must also decide on a version of the Fukaya category to use. For a degeneration of elliptic curves or K3 surfaces, the full Fukaya category $\Fuk(X_t)$ of a smooth fiber $X_t$ can be constructed using classical techniques (see~\cite[(8c)]{SeiQuartic} and the references therein for the case of K3 surfaces.) In higher dimensions, we choose to use the (algebro-geometric) relative Fukaya category of~\cite[\S{1.2}]{PerutzSheridan}, mainly because it is the easiest version of the Fukaya category of a closed Calabi--Yau manifold to construct rigorously, and because it cleanly allows for the treatment of infinitely many Lagrangian submanifolds as objects. The construction of the relative Fukaya category depends on a choice of system of divisors $\mathbf{E}_t$ in $X_t$ to be specified below (see Lemma~\ref{divisorconstruction1}.) We can view the relative Fukaya category $\Fuk(X_t,\mathbf{E}_t)$ as a $\Lambda_{\Bbbk}$-linear (uncurved) $A_{\infty}$-category~\cite[\S{1.4}]{PerutzSheridan}. The following restates part of of Kontsevich's homological mirror symemtry conjecture~\cite{KontsevichHMS} using our mirror.
\begin{theorem}\label{mirrorinjectionthm}
Let $\mathcal{X}^{\circ}\to\mathbb{D}^{\circ}$ be a maximally unipotent degeneration of strict Calabi--Yau $n$-folds as above, and suppose that $\Bbbk$ is a field whose characteristic is not contained in the set of bad primes $\Bad(\mathcal{X}^{\circ})$ (Definition~\ref{tropicallagrangiandefinition}.) Suppose that Assumotions~\ref{commutativeassumption} and~\ref{lagrangianassumption} below are satisfied. Then there is a fully faithful embedding
\begin{align}
\Perf_{dg}(\check{\mathcal{X}})\otimes_{\Bbbk[\![T]\!]}\Lambda_{\Bbbk}\to\Perf\Fuk(X_t,\mathbf{E}_t)
\end{align}
of $\Lambda_{\Bbbk}$-linear $A_{\infty}$-categories, for any $t\in\mathbb{D}^{\circ}$ close to the origin, between the category of perfect modules over the relative Fukaya category $\Fuk(X_t,\mathbf{E}_t)$ and (a dg-enhancement of) the derived category of perfect complexes on $\check{\mathcal{X}}$. When $n = 1,2$, we can replace the relative Fukaya category with the full Fukaya category $\Fuk(X_t)$.
\end{theorem}
\begin{corollary}
    When $\check{\mathcal{X}}^{\circ}$ is smooth, the fully faithful inclusion of Theorem~\eqref{mirrorinjectionthm} induces a quasi-equivalence
    \[D^b\Coh(\check{\mathcal{X}}) = \Perf_{dg}(\check{\mathcal{X}})\simeq\Fuk(X_t,\mathbf{E}_t) \,.\]
    In other words, one has homological mirror symmetry between $X_t$ and $\check{\mathcal{X}}$.
\end{corollary}
\begin{proof}
    This is immediate from the smoothness assumption and~\cite[Theorem A]{GanGeneration}. The open-closed map for the relative Fukaya category, on which the generation criteria depend, was recently constructed in~\cite{SheridanOC}.
\end{proof}

\begin{remark}[Smoothness]
Instead of stating homological mirror symmetry using $D^b\Coh$ on the B-side, as is standard in this setting, we use $\Perf$ since we have not shown that our mirror family is smooth. Comparing with the example of~\cite[Appendix C]{Pomerleano}, one should not expect that $\mathcal{\check{X}}$ will always be smooth, even after base-changing to $\Lambda_{\Bbbk}$. In situations where it is not smooth, it is reasonable to speculate that the mirror to $X_t$ might be a (non-commutative) crepant resolution of our mirror space, if one already believes that the Fukaya category of $X_t$ should be homologically smooth.

We expect, however, that it will very often be the case that $\check{\mathcal{X}}$ is smooth, in which case we will have that $\Perf(\check{\mathcal{X}}) = D^b\Coh(\check{\mathcal{X}})$. It is a folklore conjecture~\cite{GSinvitation} that the class of (smooth) toric degenerations is the largest class of maximally unipotent Calabi--Yau degenerations which is closed under taking mirrors, which informs the prediction that $\mathcal{\check{X}}$ will be smooth in this case in view of the conjectural compatibility between the intrinsic mirror construction and the toric degeneration mirror construction. In general, the mirror construction of~\cite{GSintrinsicmirrors}, as well as our own mirror construction, are not expected to be involutive. Note that the mirror degenerations~\eqref{mirrorfamily} are of a special form: they are defined over $\mathbb{Z}$, and their central fibers are unions of projective spaces (by Remark~\ref{centralfiberrmk}.) 
\end{remark}

\begin{remark}[Models for Fukaya categories]
It is natural to ask whether we can replace the relative Fukaya category in Theorem~\ref{mirrorinjectionthm} with the full Fukaya category. A proposal for defining the Fukaya category of a general compact symplectic manifold, albeit with only a \textit{finite} collection of Lagrangian submanifolds as objects, was put forward in~\cite{AFOOO}. In principle, any compact embedded Lagrangian submanifold of $X_t$ should give rise to a perfect module over $\Fuk(X_t,\mathbf{E}_t)$, assuming we are given a suitable construction of the Fukaya category. In particular, it should be possible for us to replace $\Fuk(X_t,\mathbf{E}_t)$ with any reasonable version of the full Fukaya category, up to a base change of the mirror. In fact, given a such a reasonable version of the full Fukaya category, our argument could be phrased without reference to the relative Fukaya category at all, since it does not use versality in the style of~\cite{SheridanVersalityCDM}. We will not, however, attempt to settle these foundational issues in the present work.
\end{remark}
Although we are not able to show that the functor of Theorem~\ref{mirrorinjectionthm} is an equivalence, and indeed the results of~\cite{Pomerleano} indicate that this is probably too much to hope for, we can construct a left inverse for this functor. This will be carried out in Appendix~\ref{onesidedappendix}. From this discussion, we obtain the following corollary, which is a natural analogue of~\cite[Theorem 1.1(c)]{Pomerleano}.
\begin{proposition}\label{finitelygeneratedprop}
    Let $K\subset X_t\setminus\mathbf{E}_t$ be any Lagrangian submanifold. Then the $R_{\phi}$-module
    \begin{align*}
    \bigoplus_{d=0}^{\infty} HF^0(\phi^d(L),K)
    \end{align*}
    (determined by the Yoneda module construction~\cite[(1l)]{SeiBook}) is finitely-generated.
\end{proposition}
Again, this result should apply in any reasonable model for the full Fukaya category, meaning that the restriction to $K\subset X_t\setminus\mathbf{E}_t$ is presumably removable. Amusingly, we can prove this result without considering holomorphic curves at all, instead relying on a result from the theory of stable model categories due to Schwede and Shipley~\cite[Theorem 3.1.1]{SS03} (Theorem~\ref{schwedeshiplythm} below.) One would otherwise expect a proof of this result to require a detailed analysis of the multiplication map on $\Fuk(X_t,\mathbf{E}_t)$ (cf.~\cite[Theorem 1.1(c)]{Pomerleano}.)

The rest of this section outlines a proof of Theorem~\ref{mirrorinjectionthm} under Assumption~\ref{commutativeassumption}, and an additional assumption on the existence of certain Lagrangian submanifolds (Assumption~\ref{lagrangianassumption}), which we will introduce below.

\subsection{Semistable reduction}\label{semistablereductionsection}
Although the symplectic monodromy~\eqref{symplecticmonodromyrep} is defined for the \textit{smooth} family $\mathcal{X}^{\circ}$, we can appeal to the \textit{semistable reduction theorem} to obtain detailed local models for (some power of) $\phi$ by considering an \textit{snc} model $\mathscr{X}\to\mathbb{D}$ for the degeneration. This will be our main tool for computing the various Floer-theoretic invariants associated with the pair $(X_t,\phi)$.

The semistable reduction theorem was first proved for $1$-dimensional bases in the algebraic and complex analytic categories in~\cite{KKMS}, and for general bases in the complex analytic category in~\cite{EH}. Recall from~\cite[p. 7668]{EH} that a map
\begin{align*}
p\colon X\to B
\end{align*}
of smooth complex analytic varieties, where $\dim X = n+1$ and $\dim B = 1$, is \textit{semistable} if there exist reduced divisors $\Delta_{X}\subset X$ and $\Delta_B\subset B$ such that $p$ is locally -- in the complex analytic topology -- of the form
\begin{align}\label{semistablelocaldescription}
p(x_1,\ldots,x_{n+1}) &=  t \,;\\
t &= \prod_{j = 1}^{k} x_j \nonumber
\end{align}
for some $k\leq n$, where $x_1,\ldots,x_{n+1}$ and $t$ are local coordinates in which $\Delta_Y = \lbrace x_1\cdots x_{\ell} = 0\rbrace$ for some $\ell\geq k$ and $\Delta_B = \lbrace t = 0\rbrace$.

Returning to our more specialized situation, by~\cite[Prop. 1.8(1)]{GScanonicalwall} we can choose, possibly after a base change branched at the origin, a dlt relatively minimal model $\pi'\colon(\mathcal{X}',X_0')\to\mathbb{D}$, where $X_0'$ is the fiber of $\mathcal{X}'$ over $0\in\mathbb{D}$ and $\mathcal{X}'\setminus X_0'$ is (the base-change of) $\mathcal{X}^{\circ}$. A \textit{semistable reduction} $\tilde{\pi}\colon\mathscr{X}\to\widetilde{\mathbb{D}}$ of $\pi'\colon\mathcal{X}'\to\mathbb{D}$ is a semistable map (with respect to the the central fiber $\mathbf{D} = \mathscr{X}_0=\Delta_{\mathscr{X}}\subset\widetilde{X}$ and the origin $\lbrace 0\rbrace=\Delta_{\widetilde{\mathbb{D}}}\subset\widetilde{\mathbb{D}}$) for which the following diagram commutes
\begin{equation}\label{semistablereductiondiagram}
  \begin{tikzcd}
    \mathscr{X}\arrow{d}{\tilde{\pi}} \arrow{r} & \mathcal{X}'\arrow{d}{\pi'} \\
    \widetilde{\mathbb{D}}\arrow{r} & \mathbb{D}
  \end{tikzcd}
\end{equation}
and such that
\begin{itemize}
\item[$\bullet$] $\widetilde{\mathbb{D}}\to\mathbb{D}$ is a covering map branched along the origin;
\item[$\bullet$] the central fiber $\mathbf{D}$ of $\mathscr{X}$ is a reduced simple normal crossings divisor;
\item[$\bullet$] $\mathscr{X}\to \mathcal{X}\times_{\mathbb{D}}\widetilde{\mathbb{D}}$ is an isomorphism when restricted to $\widetilde{\mathbb{D}}\setminus\lbrace 0\rbrace$, and (consequently);
\item[$\bullet$] $\tilde{\pi}$ is semistable.
\end{itemize}

It will be convenient for us to use the following slightly sharper version of semistable reduction.
\begin{proposition}[{\cite[Proposition 1.8(2)]{GScanonicalwall}}]\label{gsresolutionmodel}
Let $\mathcal{X}^{\circ}\to\mathbb{D}^{\circ}$ be a (maximally unipotent) family of Calabi--Yau manifolds as above, and let $\mathcal{X}'\to\mathbb{D}$ be a dlt minimal model with central fiber $X_0'$. Then there is a resolution of singularities $p\colon(\mathscr{X},\mathbf{D})\to(\mathcal{X}',X_0')$, where $\mathbf{D} = \mathscr{X}_0$ is the fiber over $0$, and an expression
\begin{align}\label{KXexpression}
K_{\mathscr{X}}+\mathbf{D}\equiv\sum_{i\in R} a_i D_i
\end{align}
where $\mathbf{D}$ is a reduced simple normal crossings divisor, which contains a $0$-stratum, and the sum is over the set $R$ of irreducible components $D_i$ of $\mathbf{D}$. The coefficients $a_i$ are always nonnegative, and are nonzero precisely when $D_i$ is exceptional for $p$. In particular $\mathscr{X}\to\mathbb{D}$ is semistable, and it is also projective. The components $D_i$ with $a_i = 0$ are called the \textbf{good} components of the central fiber, and the set of good components is denoted $S_0$.
\end{proposition}
\begin{remark}
We remark that in~\cite[Proposition 1.8]{GScanonicalwall} Gross and Siebert prove the existence of resolutions satisfying stronger conditions than we have stated above. The nonnegativity of the coefficients $a_i$ follows from the discussion following Definition 1.5 of~\cite{GScanonicalwall}. The semistability and projectivity of $\mathscr{X}\to\mathbb{D}$ are implicit in Gross and Siebert's statement of this result; they are part of the general requirements of the so-called relative case in~\cite[\S{1.1}]{GScanonicalwall}. The existence of a $0$-stratum in $\mathbf{D}$ follows from maximal unipotency, which Gross and Siebert use to verify Assumption 1.1(1) of~\cite{GScanonicalwall}.
\end{remark}
\begin{remark}\label{CYformremark}
Since $\mathbf{D}$ supports the canonical class, there is a unique (up to scaling) meromorphic $n$-form whose divisor of zeros and poles is $-\mathbf{D}$. This restricts to a nowhere-vanishing relative Calabi--Yau form on $\mathcal{X}^{\circ}$, which we use to obtain the grading of fixed point Floer cohomology implicit in~\eqref{mirrorring}.
\end{remark}

Using the resolution of Proposition~\ref{gsresolutionmodel}, we can describe the \textit{essential skeleton} $\Sk(\mathcal{X}^{\circ})$ of $\mathcal{X}^{\circ}$. Recall that for any snc model $\mathscr{X}$ of $\mathcal{X}^{\circ}$, there is an associated \textit{skeleton} $\Sk(\mathscr{X})$ which is a finite simplicial space canonically homeomorphic to the dual complex of the central fiber $\mathbf{D} = \mathscr{X}_0$~\cite[\S{3}]{MustataNicaise}. By applying further blowups to the total space, we can guarantee that all strata of the central fiber are irreducible, hence connected, so that this is a simplicial complex in the ordinary sense. The essential skeleton of $\mathcal{X}^{\circ}$ is a certain simplicial subspace of $\Sk(\mathscr{X})$ first defined by Kontsevich--Soibelman~\cite[\S{6.6}]{KS06}, under the assumption that $\mathcal{X}^{\circ}$ has trivial canonical sheaf, and generalized by Mustata--Nicaise~\cite{MustataNicaise}. By~\cite[4.6.2]{MustataNicaise}, $\Sk(\mathcal{X}^{\circ})$ is independent of the choice of $\mathscr{X}$. On the other hand, Nicaise--Xu show that $\Sk(\mathcal{X}^{\circ})$ is equal to the skeleton of any minimal dlt model of $\mathcal{X}^{\circ}$~\cite[Theorem 3.3.3]{NicaiseXu}. Combining these observations with Proposition~\ref{gsresolutionmodel}, we obtain the following.
\begin{corollary}
The essential skeleton $\Sk(\mathcal{X}^{\circ})$ is a simplicial subspace of $\Sk(\mathscr{X})$ which corresponds to the dual complex of good components, under the canonical homeomorphism from $\Sk(\mathscr{X})$ to the dual complex of $\mathbf{D}$.
\end{corollary}

\subsection{Floer cohomology rings}
We can form the Stanley--Reisner ring of the simplicial complex $\Sk(\mathcal{X}_0)$, which we denote by $\SR(\mathcal{X}_0)$. Using a version of Ganatra--Pomerleano's log-PSS map~\cite{GP1, GP2}, we will show that $R_{\phi}$ recovers $\SR(\mathcal{X}_0)$ at low energy.
\begin{theorem}\label{stanleyreisnerisom}
There is a natural isomorphism
\begin{align}
\SR(\mathcal{X}_0)\cong R_{\phi}\otimes_{\Bbbk[\![T]\!]}\Bbbk \,.
\end{align}
\end{theorem}
Under Assumption~\ref{commutativeassumption}, which guarantees that~\eqref{mirrorfamily} is well-defined, this says that the special fiber $\check{X}_0$ of the mirror family $\check{\mathcal{X}}$ is $\Proj\SR(\mathcal{X}^{\circ})$.
\begin{remark}\label{centralfiberrmk}
Since the Stanley--Reisner ring $\SR(\mathcal{X}^{\circ})$ is finitely-generated, it follows that $R_{\phi}$ is finitely-generated.

Moreover, this description allows us to describe the irreducible components of $\check{X}_0$ combinatorially: the irreducible components of $\check{X}_0$ correspond to the maximal cells of $\Sk{\mathcal{X}^{\circ}}$, all of which are simplices, and the Stanley--Reisner ring of an $n$-simplex is a copy of $\mathbb{P}^n$. Thus $\check{X}_0$ is a union of projective spaces along their toric boundaries (cf.~\cite[Theorem 1.7 and Remark 1.9]{MillerSturmfels}.)
\end{remark}

The log-PSS map of~\cite{GP1,GP2} is defined for log Calabi--Yau pairs $(X,D)$, using the results of Farajzadeh-Tehrani--McLean--Zinger~\cite{MTZ} to construct a \textit{symplectic regularization} for the symplectic structure near $\mathbf{D}$. Roughly, this means that we replace the symplectic form on $X$ with an isotopic one that has a certain standard form in carefully chosen tubular neighborhoods of the intersection strata of $\mathbf{D}$. This is itself based on arguments of Seidel~\cite{SeiBiased} and McLean~\cite{McLeanGrowth} which allow one to symplectically isotope $D$ so that its components intersect each other orthogonally with respect to the symplectic form on $X$. Another antecedent to this circle of ideas is Seidel's proof of homological mirror symmetry for the quartic K3 surface in~\cite[(7c)]{SeiQuartic}, where a similar, but more local, regularization argument is used to find convenient local models for the symplectic monodromy of the Dwork pencil.

To help explain how these ideas will manifest in our situation, we introduce some terminology from~\cite{BoucksomJonsson}. Set $D_I\coloneqq \bigcap_{i\in I}D_i$ for any subset $I\subset S_0$ of the set of irreducible components of the central fiber $\mathbf{D}$ of $\mathscr{X}$. Assuming without loss of generality that each $D_I$ is connected, write $D_I^{\circ} \coloneqq D_I\setminus\bigcup_{i\in S_0\setminus I}D_i$. Let $U\subset\mathscr{X}$ be an open subset intersecting $\mathbf{D}$ nontrivially and let $x = (x_1,\ldots,x_{n+1})$ be local coordinates on $U$. Following~\cite[\S{2.2}]{BoucksomJonsson}, we say that the pair $(U,x)$ is \textit{adapted} to $\mathbf{D}$ if the following conditions hold.

\begin{itemize}
\item[(i)] If ${D}_i$ for $i\in I\subset S_0$ are the irreducible components of $\mathbf{D}$ intersecting $U$, where then their intersection with $U$, if it is nonempty, is given by
\begin{align*}
U\cap \bigcap_{i\in I}D_i  = U\cap D_I^{\circ} \,.
\end{align*}

\item[(ii)] Locally, $\lbrace x_j = 0\rbrace = \mathbf{D}_{i_j}$ with $|x_j|<1$, where $I = \lbrace i_1,\ldots,i_k\rbrace$
\end{itemize}
In this situation, we call $D_I$ the \textit{stratum} of $U$. This stratum corresponds to a face of the dual complex of $\mathbf{D}$. Conversely, let $U_I$ denote a tubular neighborhood of $D_I$ in $\mathscr{X}$, which is a union of adapted charts with stratum $D_I$. We call $U_Y$ an \textit{adapted tubular neighborhood} of $D_I$. In general, given two adapted charts $(U,x)$ and $(U',x')$ with the same stratum, we have that $x_j = u_j x_j'$ on $U\cap U'$, where $u_i\neq0$ is a holomorphic function for all $i = 1,\ldots,k$, possibly after a change of indices. It is easy to see that $\mathscr{X}\to\mathbb{D}$ is locally of the form $x_1\cdots x_k$ in such an adapted chart.

Let $\Omega$ denote the symplectic form on $\mathscr{X}$. Any $\Omega$-regularization of $\mathbf{D}$, in the sense of Definition~\ref{Omega-regularization-definition} below, consists of tubular neighborhoods $U_I$ of $D_I$ which be written as a union of adapted charts for the strata $D_I$ of $\mathbf{D}$. The results of~\cite{MTZ}, recalled in Theorem~\ref{MTZthm}, guarantee the existence of a suitable regularization for $\mathbf{D}$. We would also hope for this regularization, with the modified symplectic form $\Omega_1$, to be suitably compatible with the symplectic monodromy~\eqref{symplecticmonodromyrep}.

More basically, if $\pi\colon\mathscr{X}\to\mathbb{D}$ denotes a family as in Proposition~\ref{gsresolutionmodel} as before, one would hope to make sense of symplectic parallel transport along curves in the base of this family with respect to the \textit{regularized} symplectic form $\Omega_1$. This does not make sense \textit{a priori} (cf.~\cite[Lemma 7.2 and Remark 7.3]{SeiThesis}), but after slightly perturbing the projection map to obtain a symplectic fibration $\pi_{reg}\colon(\mathscr{X},\Omega_1)\to\mathbb{D}$ following~\cite{McLeanLogCanonical}, we can construct a parallel transport automorphism $\phi_{reg}$ of one of the fibers of this perturbed family similar to~\eqref{symplecticmonodromyrep}. Most importantly, unwinding the definitions shows that the new projection map $\pi_{reg}$ is still locally described by a product of defining equations for a nearby stratum of $\mathbf{D}$, per~\eqref{compatiblewithregularizationdefinitingequation}. This procedure for modifying the fibration is reviewed in \S{\ref{symplectic-regularization-section}}. 
  
Although a generic fiber of $\pi_{reg}$ will be symplectically deformation equivalent to a fiber of $\pi$, it seems difficult to show that $\phi_{reg}$ and $\phi$ are isotopic, as pointed out in a similar context in~\cite{BobadillaPelka}. Indeed, since we only show that the mapping tori of $\phi$ and $\phi_{reg}$ are deformation equivalent, we can at best hope that they induce conjugate auto-equivalences of the Fukaya category.\footnote{Mirror symmetry heuristics suggest -- assuming the mirror is smooth -- that $\phi$ (viewed as an autoequivalence of the Fukaya category) is mirror to tensoring with a line bundle if and only if $\phi_{reg}$ is mirror to tensoring with a line bundle~\cite[\S{3}]{BondalOrlov}. Strictly speaking, our arguments below only prove that this mirror symmetry interpretation is true for $\phi_{reg}$.} The deformation equivalence of mapping tori means that it suffices, for the purpose of computing $R_{\phi}$, to study the fixed point Floer cohomology of powers of $\phi_{reg}$. When studying Lagrangian Floer theory below, we will prefer to work with $\phi_{reg}$ instead of $\phi$, since the former will admit nice local models in the neighborhoods of the $\Omega$-regularization. For that reason, we will replace $\phi$ with $\phi_{reg}$ and $\Omega$ with its regularization for the rest of this section, and in all of our discussions of Lagrangian Floer theory (cf. Assumption~\ref{completedfamilyassumption}) without further comment. In practice, this substitution is unproblematic.

The theory of regularizations developed in~\cite{MTZ, McLeanLogCanonical} tells us that -- near the large complex structure limit -- the action of the (regularized) monodromy $\phi\colon X_t\to X_t$ is locally described by `translations of SYZ fibers,' following~\cite{SeiQuartic}. More precisely, if we consider $\mathbb{C}^{n+1}$ with the map $\pi\colon\mathbb{C}^{n+1}\to\mathbb{C}$ given by~\eqref{semistablelocaldescription}, then the unique horizontal lift of the vector field $-it\partial_t$ is
\begin{align}\label{horzlift}
-\frac{i}{\frac{1}{|x_1|^2}+\cdots+\frac{1}{|x_k|^2}}\left(\frac{1}{\bar{x}_1},\ldots,\frac{1}{\bar{x}_k},0,\ldots,0\right)
\end{align}
For a point $t\in\mathbb{D}^{\circ}$, the fiber of $\pi$ is $(\mathbb{C}^*)^{k-1}\times\mathbb{C}^{n-k+1}\cong\lbrace z_1\cdots z_k = t\rbrace\subset\mathbb{C}^{n+1}$. There is a natural Lagrangian fibration $(\mathbb{C}^*)^{k-1}\times\mathbb{C}^{n-k+1}\to\mathbb{R}^n$ whose fibers are copies of $T^{k-1}\times\mathbb{R}^{n-k+1}$ given by
\begin{align}\label{SYZfibrationnearlowerstrata}
\lbrace\log|x_1| = r_1 \,,\ldots, \log|x_k| = r_k \text{ and }\mathrm{Re}(x_{k+1}) = r_{k+1}\,,\ldots,\,\mathrm{Re}(x_{n+1}) = r_{n+1}\rbrace
\end{align}
Examining~\eqref{horzlift} shows that parallel transport along a circle in the base $\mathbb{C}$ acts on $\pi^{-1}(t)$ by a translation in the $T^k$-direction. In each regularized neighborhood, $\phi$ will be approximately of this form. We remark that near the zero strata of $\mathbf{D}$, we can use Darboux's theorem to isotope the symplectic form to the standard one, with respect to the coordinates $(x_1,\ldots,x_{n+1})$, and so in these charts we can take the monodromy to be literally given by the flow of~\eqref{horzlift}.\footnote{This simple instance of regularization will suffice for all arguments involving Lagrangian Floer theory carried out in this section.}

The isomorphism of Theorem~\ref{stanleyreisnerisom} is proved at the level of $\Bbbk$-modules in Lemma~\ref{PSS-domain-codomain-abstract-iso}, using the observation that the orbits of $\phi\colon X_t\to X_t$ form Morse--Bott families diffeomorphic to tori. Of course, there will also be families of orbits near the strata of the central fiber that are not intersections of good components, but by our choice of resolution in Proposition~\ref{gsresolutionmodel}, these orbits will not contribute to the degree $0$ fixed point Floer cohomology groups for grading reasons (cf.~\cite[Conjecture 1.28]{GSintrinsicmirrors}.)

We continue under the assumption that $X_t$ contains a Lagrangian submanifold that is compatible with these Morse--Bott families in the following sense. Associated to any snc model $\mathscr{X}\to\mathbb{D}$ and any point $t\in\mathbb{D}$ sufficiently close to the origin there is a \textit{troplicalization} map
\begin{align}\label{adaptedLogmap}
    \Log_t\colon X_t\to\Sk(\mathscr{X})
\end{align}
defined using adapted tubular neighborhoods in~\cite[Proposition 2.1]{BoucksomJonsson}. In an adapted chart, the tropicalization map is locally of the form
\begin{align*}
\Log_{\mathcal{X}}(x_0,\ldots,x_n)\coloneqq\frac{1}{\log|t|^{-1}}(\log|x_0|,\ldots,\log|x_{k-1}|)
\end{align*}
where $k$ is the dimension of the associated stratum. The codomain of the tropicalization map restricted to this chart is the face $\sigma_I$ of $\Sk(\mathscr{X})$ corresponding to the associated stratum $D_I$ of $\mathbf{D}$. The next definition roughly singles out Lagrangian sections of the tropicalization map.
\begin{definition}\label{tropicallagrangiandefinition}
Fix an $\Omega$-regularizaton $\mathfrak{R}$ for $(\mathscr{X},\mathbf{D})$ as in Definition~\ref{Omega-regularization-definition}, and let $t\in\mathbb{D}^{\circ}$ be a point close to the origin. We say that a Lagrangian rational homology sphere $L\subset X_t$ is a \textit{tropical Lagrangian section} if the following hold.
\begin{itemize}
    \item[(i)] Near any point in the intersection $L\cap U_I$, where $U_I$ is the tubular neighborhood of any \textit{good} stratum $D_I$ in the regularization $\mathfrak{R}$, for any $I\subset S_0$, there is an adapted chart in which $L$ restricts to a product
    \[ L\cong L_{(\mathbb{C}^*)^{k-1}}\times L_{\mathbb{C}^{n-k+1}}\subset(\mathbb{C}^*)^{k-1}\times\mathbb{C}^{n-k+1}\]
    of Lagrangian sections with respect to the natural Lagrangian fibrations, so in particular $L$ is locally a section of the fibration~\eqref{SYZfibrationnearlowerstrata}.

    \item[(ii)] The log map~\eqref{adaptedLogmap} restricted to $L$ is a homotopy equivalence onto the essential skeleton $\Sk(\mathcal{X}^{\circ})$, and, moreover, in each of the adapted charts of item (i) it locally looks like the composition
    \[L_{(\mathbb{C}^*)^{k-1}}\times L_{\mathbb{C}^{n-k+1}}\to L_{(\mathbb{C}^*)^{k-1}} \xrightarrow{\Log_t}\sigma_I \]
    where the first map is projection onto the first factor and the second map identifies $L_{(\mathbb{C}^*)^{k-1}}$ with a portion of the cell $\sigma_I$ of $\Sk(\mathcal{X}^{\circ})$.
\end{itemize}
Let $\Bad(\mathcal{X}^{\circ})$ denote the set of primes $p$ such that $H^*(L;\mathbb{Z}) = H^*(\Sk(\mathcal{X}^{\circ});\mathbb{Z})$ has nontrivial $p$-torsion.
\end{definition}
Note that condition (ii) of this definition implies that the set of bad primes $\Bad(\mathcal{X}^{\circ})$ depends on $\mathcal{X}^{\circ}$. The imposition that $L$ is a rational homology sphere is also reduntant by~\cite[Prop. 31 and Par. 32]{KollarXu}.) We only discuss homological mirror symmetry over fields $\Bbbk$ for which $\operatorname{char}\Bbbk\not\in\Bad(\mathcal{X}^{\circ})$. This is because non-vanishing middle-dimensional cohomology in $H^*(L;\Bbbk)$ would prevent us from appealing to the main results of~\cite{Polishchuk} below.

\begin{customthm}{L}\label{lagrangianassumption}
$X_t$ contains a tropical Lagrangian section $L$ for some $t\in\mathbb{D}^{\circ}$.
\end{customthm}
\begin{remark}\label{divisorremark}
We will show in Lemma~\ref{divisorconstruction1} that given any tropical Lagrangian section $L$, the Lagrangian submanifolds $\phi^d(L)$ for all $d\in\mathbb{Z}$ are objects of the relative Fukaya category with respect to a certain system of divisors $\mathbf{E}_t$ in $X_t$ in the sense of~\cite[Definition 1.1]{PerutzSheridan}. We will prove that a subdomain of $X_t\setminus \mathbf{E}_t$ is fixed by the action of $\phi$. Since $L$ is a $\Bbbk$-homology sphere, it is unobstructed over $\Lambda_{\Bbbk}$ by~\cite[Corollary 3.8.18]{FOOOI} and its Floer cohomology is isomorphic to its ordinary cohomology (at the level of vector spaces) by~\cite[Theorem D]{FOOOI}.
\end{remark}

Under Assumptions~\ref{commutativeassumption} and~\ref{lagrangianassumption}, we will construct an $A_{\infty}$-functor\footnote{Throughout, we will replace both $A_{\infty}$-categories with $A_{\infty}$-minimal models.}
\begin{align}\label{mirrorinjection}
\Perf_{dg}(\check{\mathcal{X}})\to\Perf\Fuk(X_t,\mathbf{E}_t)
\end{align}
which sends $\mathcal{O}_{\check{\mathcal{X}}}$ to the tropical Lagrangian section $L$, and which sends the powers $\mathcal{L}^{\otimes d}$ of a relatively ample line bundle $\mathcal{L}$ on $\check{\mathcal{X}}$ to the monodromy iterates $\phi^d(L)$ for $d\in\mathbb{Z}$. More precisely, the line bundle $\mathcal{L}$ is the one arising from the relative Proj construction.

Recall from~\cite[Theorem 4]{Orlov} that the powers of $\mathcal{L}$ split-generated $\Perf(\check{\mathcal{X}})$, meaning that it suffices to define this functor on the full $A_{\infty}$-subcategory of $\Perf_{dg}(\check{\mathcal{X}})$ with objects $\mathcal{L}^{\otimes d}$. To show that such a functor is well-defined, we must first show that the Floer cohomology groups between the Lagrangians $\phi^d(L)$ coincide with the Ext groups between the perfect complexes $\mathcal{L}^{\otimes d}$. This is essentially immediate from the definition of a tropical Lagrangian section and the local form~\eqref{horzlift} of the monodromy vector field after semistable reduction.

In more detail, we will construct a (length zero) \textit{twisted closed-open map}
\begin{align}\label{twistedclosedopenfirstmention}
\mathcal{CO}_{\phi^d}\colon CF^*(X_t,\phi^d)\to CF^*(L,\phi^d(L)) \,.
\end{align}
When $d = 0$, this is just the ordinary length zero closed-open map. These are assembled into a total map
\begin{align}
\mathcal{CO}_+^0\coloneqq\bigoplus_{d=0}^{\infty}\mathcal{CO}^0_{\phi^d}\colon\bigoplus_{d=0}^{\infty}CF^0(X_t,\phi^d)\to\bigoplus_{d=0}^{\infty}CF^0(L,\phi^d(L))
\end{align}
in degree $0$.

Using the fact that $L$ is a tropical Lagrangian section, we will prove the following by considering a low-energy limit of the energy spectral sequence.
\begin{proposition}\label{closedopenisomorphism}
For any tropical Lagrangian section $L$, the map $[\mathcal{CO}_+^0]$ is an isomorphism of Floer cohomology groups of degree $0$. In particular, $HF^0(L,\phi^d(L))$ is isomorphic to $H^0(\check{X}_0,\mathcal{L}_0^{\otimes d})$.

Moreover, the Floer cohomology groups $HF^*(L,\phi^d(L))$ vanish unless $* = 0$ and $d\geq0$ or $* = n$ and $d\leq 0$.
\end{proposition}
Additionally, we have that $HF^0(L,\phi^d(L))\cong HF^n(\phi^d(L),L)^\vee\cong HF^n(L,\phi^{-d}(L))^{\vee}$ for all $d\in\mathbb{Z}$, using the $\phi$-action and weak proper Calabi--Yau structure on $\Fuk(X_t,\mathbf{E}_t)$. The vanishing of all middle-dimensional cohomology groups, over $\Lambda_{\Bbbk}$, follows from the condition that $L$ is a $\Bbbk$-homology $n$-sphere and  an elementary degree computation.

\subsection{Matching $A_{\infty}$-structures}
Consider the full subcategory of $\Fuk(X_t,D_t)$ whose objects are $\lbrace\phi^d(L)\rbrace_{d\in\mathbb{Z}}$, and whose $A_{\infty}$-operations are denoted by $\lbrace\mathfrak{m}_k\rbrace_{k\geq1}$. Since $\phi^d(L)\neq L$ for all $d\in\mathbb{Z}$, as one can see from~\eqref{horzlift}, we can choose perturbation data in the construction of $\Fuk(X_t,D_t)$ so that this category carries an action of the group $\mathbb{Z}\cong\langle\phi\rangle\subset\pi_0\mathrm{Symp(X_t,\omega_t)}$ by strict $A_{\infty}$-autoequivalences (cf.~\cite[(10b)]{SeiBook} and Remark~\ref{divisorremark}.) Using this, we can construct an $A_{\infty}$-algebra
\begin{align*}
\mathcal{A}_{\phi}\coloneqq\bigoplus_{d\geq0} HF^0(L,\phi^d(L))\oplus\bigoplus_{d\leq0} HF^n(L,\phi^d(L))
\end{align*}
where the $A_{\infty}$-operations are given by
\begin{align*}
\mathfrak{m}_k^{\mathcal{A}}(a_1,\ldots,a_k) = \mathfrak{m}_k(\phi^{\ell_2+\cdots+\ell_k}(a_1),\ldots,\phi^{\ell_k}(a_{k-1}),a_k)
\end{align*}
for a sequence of elements $a_{i}\in HF^*(L,\phi^{\ell_i}(L))$.

Clearly the $A_{\infty}$-structure on $\mathcal{A}_{\phi}$ completely determines the $A_{\infty}$-structure on the full subcategory of $\Fuk(X_t,\mathbf{E}_t)$ split-generated by $\lbrace\phi^d(L)\rbrace_{d\in\mathbb{Z}}$, and vice-versa. Proposition~\ref{closedopenisomorphism} shows that $A_{\phi}$ is isomorphic, as an associative algebra, to
\begin{align}\label{coordinate-algebra}
\bigoplus_{d\geq0} H^0(\check{\mathcal{X}},\mathcal{L}^{\otimes d})\oplus\bigoplus_{d\leq0} H^n(\check{\mathcal{X}},\mathcal{L}^{\otimes d}) \,.
\end{align}
To construct the functor~\eqref{mirrorinjection}, we also need to show that the middle-dimensional cohomology groups $H^*(\check{\mathcal{X}},\mathcal{L}^{\otimes d})$ vanish. Since the rank of sheaf cohomology is upper semi-continuous, it suffices to consider this question for the line bundles $\mathcal{L}_0^{\otimes d}$ on the central fiber $\check{X}_0 = \Proj\SR(\mathcal{X}^{\circ})$ of the mirror family, which are obtained by restricting $\mathcal{L}^{\otimes d}$. For $d = 0$, we have that $H^i(\check{X}_0,\mathcal{O}_{\check{X_0}})$ is isomorphic to the ordinary cohomology of $\Sk(\mathcal{X}^{\circ})$ with coefficients in $\Bbbk$ by~\cite[Lemma 3.63]{KollarBook}\footnote{Although this result is only written over $\mathbb{C}$ in~\cite{KollarBook}, it is clear that the proof given there works over any field.}, meaning that it vanishes when $0<i<n$. By Serre's vanishing theorem~\cite[III.5.3(ii)]{Hartshorne}, there is some $n_0>0$ such that the cohomology groups $H^i(\check{X}_0,\mathcal{L}^{\otimes d})$ vanish for all $d\geq n_0$ and $i>0$. By replacing $\phi$ with $\phi^{n_0}$, we can take $n_0 = 1$.

The next lemma allows us to apply Serre duality for Cohen--Macaulay varieties to conclude that the cohomology groups $H^*(\check{X}_0,\mathcal{L}_0^{\otimes d})$ vanish unless $* = 0$ and $d\geq0$ or $* = n$ and $d\leq 0$.
\begin{lemma}\label{slclemma}
The variety $\check{X}_0$ is Cohen--Macaulay with semi log-canonical singularities and trivial canonical bundle.
\end{lemma}
The analogous result in the (absolute) log Calabi--Yau setting appears in Oldfield's thesis~\cite{Oldfield}. For the reader's convenience, we will prove Lemma~\ref{slclemma} in Appendix~\ref{slcappendix} with combinatorial methods. Strictly speaking, we will never use the fact that $\check{X}_0$ has slc singularities, but it follows almost immediately from the proofs of the other two parts of Lemma~\ref{slclemma}. Note, however, that it is known that the conclusion of the Kodaira vanishing theorem holds for such varieties over $\mathbb{C}$ by~\cite[Corollary 6.6]{kodairaslc}, meaning that when $\Bbbk = \mathbb{C}$ we do not have to pass to a subsequence of tropical Lagrangian sections.

With the above understood, we can appeal to the work of Polishchuk~\cite{Polishchuk} to show that the assignment~\eqref{mirrorinjection} respects $A_{\infty}$-structures. By Theorem 1.1 of \textit{op. cit.}, there is a unique nontrivial isomorphism class of $A_{\infty}$-structure on the $\Bbbk$-algebra
\begin{align}\label{centralcoordinatealgebramirror}
\bigoplus_{d\geq0} H^0(\check{\mathcal{X}},\mathcal{L}^{\otimes d})\oplus\bigoplus_{d\leq0} H^n(\mathcal{\check{X}},\mathcal{L}^{\otimes d})
\end{align}
up to scaling.\footnote{Note that Polishchuk's results do not require that $\check{\mathcal{X}}$ be smooth; only that the relevant middle-dimensional cohomology groups vanish.} 

In more detail, there is a canonical class of $A_{\infty}$-structures on the derived category $D^+(\mathcal{C})$ of bounded below complexes of any abelian category with enough injectives. Since it will be useful for organizing our computations of $A_{\infty}$-structures, we describe this canonical class following~\cite[\S{3.2}]{Polishchuk}. Suppose that we have an exact sequence
\begin{align*}
0\to\mathcal{F}_1\xrightarrow{\alpha_1}\mathcal{F}_2\to\cdots\xrightarrow{\alpha_{n+1}}\mathcal{F}_{n+1}\to 0
\end{align*}
in $\mathcal{C}$, where $n\geq1$. This sequence determines an extension class $\beta\in\Ext^{n}(\mathcal{F}_{n+1},\mathcal{F}_1)$, under the standard correspondence between extensions and classes in $\Ext^n$-groups (e.g.~\cite[Lemma 13.27.5]{StacksProject}). If $\Ext^{j-i-1}(\mathcal{F}_j,\mathcal{F}_i) = 0$ whenever $0\leq i<j\leq n$, then by~\cite[Lemma 3.4]{Polishchuk}, it follows that
\begin{align*}
\mathfrak{m}_{n+2}(\alpha_1,\ldots,\alpha_{n+1},\beta) = \pm\id_{\mathcal{O}} \,.
\end{align*}
for any $A_{\infty}$-structure $\lbrace\mathfrak{m}_k\rbrace_{k\in\mathbb{Z}_{\geq0}}$ from the canonical class.

Polishchuk~\cite[Theorem 1.1(1)]{Polishchuk} shows that the natural $A_{\infty}$-structure on $\Perf_{dg}(\check{\mathcal{X}})$ is induced from the canonical class, in the special case $\mathcal{C} = \Perf(\check{\mathcal{X}})$, as follows. Replacing $\mathcal{L}$ with a positive power if necessary, we can choose sections $s_1,\ldots,s_{n+1}\in H^0(\mathcal{L})$ with no common zeros. This means that they give rise to a Koszul resolution
\begin{align}\label{koszulsequence}
0\to\mathcal{O}\xrightarrow{\alpha_1}\mathcal{O}^{\oplus(n+1)}\otimes_{\mathcal{O}}\mathcal{L}\xrightarrow{\alpha_2}\mathcal{O}^{\oplus\binom{n+1}{2}}\otimes_{\mathcal{O}}\mathcal{L}^{\otimes 2}\rightarrow\cdots\xrightarrow{\alpha_{n+1}}\mathcal{L}^{\otimes(n+1)}\rightarrow 0
\end{align}
of the structure sheaf $\mathcal{O} = \mathcal{O}_{\check{\mathcal{X}}}$, where all morphisms are determined by the sections $s_i$ (e.g. $\alpha_1 = s_1\oplus\cdots\oplus s_{n+1}$.) If $\beta\in\Ext^{n}(\mathcal{L}^{\otimes(n+1)},\mathcal{O})$ represents the Yoneda extension class, \cite[Lemma 3.4]{Polishchuk} shows that
\begin{align*}
\mathfrak{m}_{n+2}(\alpha_1,\ldots,\alpha_{n+1},\beta) = \pm\id_{\mathcal{O}} \,.
\end{align*}
We can find a corresponding nontrivial value of $\mathfrak{m}^{\mathcal{A}}_{n+2}$ on $\mathcal{A}_{\phi}$ using Morse-theoretic methods. We interpret this nontrivial $A_{\infty}$-operation as a count of (low energy) Morse flow trees. This, together with~\cite[Theorem 1.1(1)]{Polishchuk}, shows that $\mathcal{A}_{\phi}$ is $A_{\infty}$-isomorphic to~\eqref{centralcoordinatealgebramirror}, and in turn that there is an $A_{\infty}$-functor as in~\eqref{mirrorinjection}.

Let $x_I$ denote a $0$-stratum of the central fiber $\mathbf{D}$ and choose an adapted tubular neighborhoods $U_I$ centered at $x_I$. In this chart, the map $\mathscr{X}\to\mathbb{D}$ looks like
\[ (x_1,\ldots,x_{n+1})\mapsto x_1\cdots x_{n+1}\]
as we have discussed, and, by~\cite[Lemma 1.7]{SeiLES}, we can assume that the symplectic form on $\mathscr{X}$ looks like the standard K{\"a}hler form with respect to these coordinates. This allows us to, locally in $X_t\cap U_I$, identify $\phi$ with the return map associated to~\eqref{horzlift}, which is given by fiberwise translations of the standard Lagrangian $T^n$-fibration on $(\mathbb{C}^*)^n\to\mathbb{R}^n$ given by the log map
\begin{align*}
\Log_t\colon(\mathbb{C}^*)^n&\to\mathbb{R}^n \\
(z_1,\ldots,z_n)&\mapsto(\log_t|z_1|,\ldots,\log_t|z_n|)
\end{align*}
(cf.~\eqref{SYZfibrationnearlowerstrata}.) We can also identify $\phi^d(L)\cap U_{I}$ with a section this fibration by Definition~\ref{tropicallagrangiandefinition}. For definiteness, we can  
\begin{align}\label{localgraphs}
(y_1,\ldots,y_n)\mapsto-\sum_{i=1}^n dN\cdot y_i\frac{\partial}{\partial y_i}
\end{align}
where $N>0$ is an integer.

If we instead consider an irreducible component $D_i$ of $\mathbf{D}$, and let $U_i$ denote the associated reguarlized neighborhood, then we no longer have a Lagrangian torus fibration on $U_i\cap X_t$, but the Definition~\ref{tropicallagrangiandefinition}(ii) tells us that $\phi^d(L)\cap U_i$ is a Lagrangian $n$-disk in $X_t\cap U_i$. Since the family $\mathscr{X}\to\mathbb{D}$ is smooth when restricted to $U_i$, the monodromy with respect to the regularized symplectic form is isotopic to the identity in this chart. This allows us to Hamiltonian isotope (finitely many of) the Lagrangians $\phi^d(L)$ with $d>0$ so that they each intersect $L$ in exactly one point. We can also assume that this finite collection of Lagrangians, when restricted to $U_i$, all lie in a Weinstein neighborhood $W_i\subset U_i\cap X_t$ of $L$. Summarizing, we have shown that
\begin{lemma}\label{locallagrangianintersections}
Suppose that $x_I$ is a $0$-stratum of $\mathbf{D}$ contained in the intersection of $n+1$ (distinct) $n$-strata denoted $D_i$ for $i = 1,\ldots,n+1$. Then there is an open subdomain $W_I$ of $X_t$ containing the open subsets $X_t\cap U_I$ and $X_t\cap W_{i}$, where $W_i$ is a Weinstein neighborhood of the Lagrangian disk $L\cap U_i$ as above. Moreover, we can apply Hamiltonian isotopies to the Lagrangian submanifolds $\phi^d(L)$, for $d = 0,\ldots,n+1$, so that they are locally given by the graphs of~\eqref{localgraphs}, in each of the subdomains $W_I$, for all such $0$-strata, simultaneously.
\end{lemma}
We remind the reader that $\mathbf{D}$ contains a $0$-stratum since $\mathcal{X}^{\circ}$ is a maximally unipotent degeneration by Proposition~\ref{gsresolutionmodel}.
\begin{remark}\label{momentimageremark}
One can think of this subdomain $W_I$ as a subset of $(\mathbb{C}^*)^n$ as follows. Starting with the standard $n$-simplex in $\mathbb{R}^n$, one first takes Darboux balls $U_1,\ldots,U_{n+1}\subset(\mathbb{C}^*)^n$ which project to open disks around the vertices of the $n$-simplex. For definiteness, we can take the vertices of the simplex to be the origin and the unit standard basis vectors. We can then choose a contractible open subset of the simplex which is an inward of the top stratum of the simplex, and let $U\subset(\mathbb{C}^*)^n$ denote its preimage under $\Log_t$. Without loss of generality, we assume that the union
\[ W\coloneqq U\cup U_1\cup\cdots\cup U_n\]
is connected. It is then clear that we can identify $W_I$ with $W$ in such a way that $U$ is taken to $U_I$ and each $U_i$ is taken to $W_i$, where we can replace $U$ with a smaller open subset $U'$ (for which $\Log_t(U')$ has boundary further away from the middle strata of the simplex) if necessary.

We remark that the simplex in this discussion should be thought of as a simplex in the essential skeleton $\Sk(\mathcal{X}^{\circ})$. To clear up a potential confusion, our arguments below are not being applied to the complement of all middle-dimensional strata in $\Sk(\mathcal{X}^{\circ})$ at once, but only to one of the neighborhoods $W_I$ at a time.
\end{remark}
Viewing $W_I\cong W$ as an open subdomain of $(\mathbb{C}^*)^n$, we can apply the results of~\cite{AbouzaidTropical}. We claim that the Lagrangian disks $\phi^d(L)\cap W_I\subset W$ are open subsets of the Lagrangian disks constructed in~\cite{AbouzaidCoordinateRing}. The Lagrangian disks in $(\mathbb{C}^*)^n$ containing $\phi^d(L)$ in question, denoted $T_d$ for in this discussion, are small perturbations of portions of the graphs of~\eqref{localgraphs} thought of as Lagrangians submanifolds with boundary in $(\mathbb{C}^*)^n$~\cite[\S{5}]{AbouzaidCoordinateRing}. We can assume without loss of generality that $L\cap W_I$ (and $T_0$) can be identified with a subset of the positive real locus in $(\mathbb{C}^*)^n$.

The arguments of~\cite{AbouzaidTropical}, which are similar to~\cite{FOMorse, KSMorse}, show that the $A_{\infty}$-operations on (the full subcategory generated by these Lagrangian sections of) the Fukaya category are determined by counting Morse flow trees in the base $\mathbb{R}^n$. With this fact, we can compute the higher $A_{\infty}$-operation in our setting using the intrinsic description of the canonical class of $A_{\infty}$-structures.
\begin{lemma}\label{local-a-infinity-computation}
The $A_{\infty}$-structure on $\mathcal{A}_{\phi}$ inherited from the Fukaya category is nontrivial, i.e. it is $A_{\infty}$-isomorphic to the natural $A_{\infty}$-structure on~\eqref{coordinate-algebra}.
\end{lemma}
\begin{proof}
Using Lemma~\ref{locallagrangianintersections}, as explained above, we can identify certain open subsets of the Lagrangians $\phi^d(L)$ with corresponding open subsets of the graphs of~\eqref{localgraphs} inside $(\mathbb{C}^*)^n$, up to Hamiltonian isotopy. As before, these Lagrangian submanifolds are denoted $T_0,\ldots,T_{n+1}\subset T^{2n}$. These graphs are mirror to powers of an ample line bundle under the mirror functor of~\cite[Theorem 1.2]{AbouzaidTropical}. The fact that the derived category of $\mathbb{P}^n$ is in the canonical class~\cite[Lemma 3.4 and Theorem 1.1(1)]{Polishchuk}, implies the nonvanishing of a certain $A_{\infty}$-operation with $n+2$-inputs involving the Lagrangians $T_d$. Explicitly, one has that
\begin{align*}
\mathfrak{m}_{n+2}(\alpha_1,\ldots,\alpha_{n+1},\beta) = \pm\id_{T_0} \,
\end{align*}
where
\begin{align*}
\alpha_i\in HF^0(T_i^{\binom{n+1}{i-1}},T_{i+1}^{\binom{n+1}{i}})
\end{align*}
(with the exponent denoting a direct sum of objects in the Fukaya category) and
\begin{align*}
\beta\in HF^n(T_{n+1},T_0) \,.
\end{align*}
The proof of~\cite[Theorem 1.2]{AbouzaidTropical} shows that the higher $A_{\infty}$-operations can all be interpreted as counts of Morse flow trees. These lift to holomorphic $(n+3)$-gons in $\mathbb{C}^n$ with Lagrangian labels chosen from the Lagrangians $T_i$ which contribute to this $A_{\infty}$-operation.

We claim that after applying Hamiltonian isotopies to $T_0,\ldots,T_{n+1}$, the relevant $(n+3)$-gons will be regular and will lie in the image of the subdomain $W_I$. To see this, note that locally in $W_I$, we can identify $W_I\cap T_i$ (which we recall can be identified with a subdomain of $\mathbb{C}^n = T^*\mathbb{R}^n)$ with the graph of an exact $1$-form $df_i$ in $T^*\mathbb{R}^n$, where we can also assume that $T_0$ lifts to an open subset of the zero-section of $T^*\mathbb{R}^n$. The primitives $f_i$ can, without loss of generality, be assumed to have index $0$ critical points at the points in $T_0$ corresponding to the vertices of the simplex used to describe $W_I$ in Remark~\ref{momentimageremark}. We can now take a limit $J_t$ as $t\to 0$ of almost complex structures, as explained in~\cite{FOMorse}, which makes the $J_t$-holomorphic curves almost linear in the cotangent fibers, to obtain the correspondence between Morse flow trees and holomorphic polygons in this setting. Having made these identifications, the generators $\alpha_i$ will correspond to vertices of the simplex of Remark~\ref{momentimageremark}, and the generator $\beta$ will correspond to the barycenter.

The result of Lemma~\ref{locallagrangianintersections} implies that the holomorphic $(n+3)$-gons constructed this way give rise to (regular) holomorphic $(n+3)$-gons in $X_t$ by standard arguments~\cite{FOMorse}. Set
\begin{align}\label{total-input-generators}
\alpha_i\in HF^0(\phi^i(L)^{\binom{n+1}{i-1}},\phi^{i+1}(L)^{\binom{n+1}{i}})
\end{align}
and
\begin{align}\label{total-output-generators}
\beta\in HF^n(\phi^{n+1}(L),L)
\end{align}
to be the sums of the corresponding Floer generators corresponding to all $0$-strata of $\mathbf{D}$. It remains to show that
\begin{align}\label{actualainfinityoperation}
\mathfrak{m}_{n+2}(\alpha_1,\ldots,\alpha_{n+1},\beta) = \pm\id_L\in HF^0(L,L)
\end{align}
in $X_t\setminus\mathbf{E}_t$. First note that the only disks corresponding to~\eqref{actualainfinityoperation} are the ones constructed above, by the exactness of $\phi^d(L)\subset X_t\setminus\mathbf{E}_t$ and area considerations.

Finally, recalling that the chain complex $CF^*(L,L)$ is defined in~\cite{PerutzSheridan} using a Hamiltonian perturbation of $L$, we can take this perturbation to be supported in a Weinstein neighborhood of $L$, where it is given by the graph of a $C^2$-small Morse function on $L$. Specifically, the dual complex to $\Sk(\mathcal{X}^{\circ})$ determines a Morse function on $L$ whose $0$-cells correspond to the $n$-cells of $\Sk(\mathcal{X}^{\circ})$, and the unit in $HF^0(L)$ is represented by the sum of these generators. Then, in view of~\cite{FOMorse}, the $A_\infty$ structural coefficients again reduce to counts of Morse flow trees in $X_t\setminus\mathbf{E}_t$, therefore the computation in the proof of~\cite[Theorem 1.2]{AbouzaidCoordinateRing} implies the desired statement. We have shown that the $A_{\infty}$-product on the left-hand side of~\eqref{actualainfinityoperation} is nonzero, and by~\cite[Remarks. 1.]{Polishchuk} this completes the proof.
\end{proof}
\begin{remark}
    Although the identification of holomorphic disks to determine the $A_{\infty}$-operations took place locally near a single $n$-simplex of $\Sk(\mathcal{X}^{\circ})$, we produced an $A_{\infty}$-operation with inputs~\eqref{total-input-generators} and output~\eqref{total-output-generators} defined as a sum over all vertices and over all barycenters of $n$-cells, respectively. This is consistent with the fact that, in the mirror, the central fiber of~\eqref{mirrorfamily} is a union of projective spaces. Standard facts in toric geometry imply that, if we view the $n$-simplex as the Newton polytope of $\mathbb{P}^n_{\Bbbk}$, the vertices and barycenters will correspond to sections of ample line bundles. Summing over all of these sections appropriately gives us a Koszul resolution of the form~\eqref{koszulsequence}, again because the derived category of $\mathbb{P}^n_{\Bbbk}$ is in Polishchuk's canonical class.

    If we consider an ample line bundle defined over the entire central fiber $\check{X}_0$, then a section of this bundle will of course restrict to a section of an ample line bundle on each irreducible component. This is the reason we have to describe the mirror Koszul resolution more globally in~\eqref{total-input-generators} and~\eqref{total-output-generators} than our local computation might, at first glance, seem to require.
\end{remark}
\begin{proof}[Proof of Theorem~\ref{mirrorinjectionthm}]
    The algebra isomorphism
    \[ \bigoplus_{d\geq0} H^0(\check{\mathcal{X}},\mathcal{L}^{\otimes d})\oplus\bigoplus_{d\leq0} H^n(\check{\mathcal{X}},\mathcal{L}^{\otimes d}) \cong \bigoplus_{d\geq0} HF^0(L,\phi^d(L))\oplus\bigoplus_{d\leq0} H^n(L,\phi^d(L))\]
    over $\Lambda_{\Bbbk}$, combined with Lemma~\ref{local-a-infinity-computation} completes the proof.
\end{proof}

In what remains of this paper, we will provide the details pertaining to the construction and basic properties of $\check{\mathcal{X}}$ which were omitted from the discussion so far, and we will show that Assumptions~\ref{commutativeassumption} and~\ref{lagrangianassumption} all hold for certain toric degenerations.

\subsection{Applicability and comparisons with other approaches}
Unfortunately, we cannot yet show that Theorem~\ref{mirrorinjectionthm} applies to all simple toric degenerations. This is because we rely on the existence of a nice resolution of singularities in order to compute the essential skeleton of a toric degeneration (see Assumption~\ref{toricdegenerationresolutionassumption} below.) The existence of this resolution is, at least in principle, only necessary to show to verify Assumption~\ref{lagrangianassumption}. The construction of Lagrangian submanifolds in a smooth fiber of a toric degeneration $\mathcal{X}\to\mathbb{D}$ that we rely on comes from of Arg{\"u}z~\cite{Arguzreal}. Under certain conditions on the gluing data, of a toric degeneration, the results of \textit{op. cit.} show that the real locus of a real toric degeneration contains a connected component homeomorphic to the intersection complex~\cite[Example 1.13]{GSrealaffine} of the central fiber of $\mathcal{X}$. Assumption~\ref{toricdegenerationresolutionassumption} allows us to identify the essential skeleton of $\mathcal{X}^{\circ} = \mathcal{X}\mid_{\mathbb{D}^{\circ}}$ with the intersection complex, which we need to check Definition~\ref{tropicallagrangiandefinition}(ii).

Let us also note that our Theorem~\ref{toric-degeneration-hms-thm} has nontrivial overlap with known cases of homological mirror symmetry for closed Calabi--Yau manifolds. It is instructive to compare this result with analogous results for the relative Fukaya category in these settings. In~\cite{SeiQuartic, SheridanHypersurfaces, SheridanSmith, GHHPS}, relative Fukaya categories are defined using a certain normal crossings divisor, the complement of which are (very) affine varieties with extra symmetries. These symmetries make it possible to control the deformation space of the Fukaya category of the divisor complement, allowing one to apply versality results (cf.~\cite{SheridanVersality}.) In contrast to these works, we use an essentially arbitrary smooth divisor to define the relative Fukaya category, mainly as an auxiliary tool to prove transversality results.

For the purpose of illustration, we consider the following slightly non-rigorous example. If $N_t$ denotes a normal crossings divisor, as used in one of the works cited above, and we take $D_t$ to be a smoothing of this divisor, then in general $X_t\setminus N_t$ is a Liouville subdomain of $X_t\setminus D_t$. More precisely, one obtains $X_t\setminus D_t$ by performing Weinstein handle attachments on $X_t\setminus N_t$. For appropriate choices of $N_t$ and $D_t$, the tropical Lagrangian $L$ should lie in the complement of both divisors, and thus they define objects of both relative Fukaya categories $\Fuk(X_t,D_t)$ and $\Fuk(X_t,N_t)$ (assuming for the purpose of this informal discussion that the former relative Fukaya category is well-defined.)

If we tensor both categories with the Novikov field $\Lambda_{\Bbbk}$, then local nature of our Floer-theoretic computations show that the quasi-isomorphism type of $\mathcal{A}_{\phi}$ is unaffected by this change. On the other hand, one can consider both \textit{relative} Fukaya categories $\Fuk(X_t,D_t)$ and $\Fuk(X_t,N_t)$ as Fukaya categories over the ring of Laurent series $\Bbbk[\![T]\!]$ by~\cite[\S{1.4}]{PerutzSheridan}. By changing coefficients to $\Bbbk[\![T]\!]$ this way, we can produce two different versions of $\mathcal{A}_{\phi}$ defined over $\Bbbk$, and they will be non-isomorphic. Algebraically, this is a symptom of the elementary fact that two non-isomorphic $\Bbbk[\![T]\!]$-algebras can be isomorphic over $\Bbbk(\!(T)\!)$. Geometrically, this corresponds to looking at two families with the same general fiber but different special fibers. The local computation given in the proof of Lemma~\ref{local-a-infinity-computation} shows that both versions of $\mathcal{A}_{\phi}$ over $\Bbbk$ belong to the canonical class, as expected by~\cite[Lemma 3.4]{Polishchuk}, since both $A_{\infty}$-algebras correspond to projective $\Bbbk$-varieties.

We should not expect that the Fukaya category (of compact Lagrangians) $\Fuk(X_t\setminus D_t)$ will be split-generated by $\lbrace\phi^d(L)\rbrace_{d\in\mathbb{Z}}$, as the process of attaching Weinstein handles may create new compact Lagrangians in $X_t\setminus D_t$ that are not contained in $X_t\setminus N_t$. Ultimately, these new Lagrangian submanifolds should become Floer-theoretically trivial in $X_t$ once the divisor is re-inserted. Using the well-known philosophy that different choices of divisor should correspond to different large K{\"a}hler limits, we do expect that for some (unspecified) `correct' choice of system of divisors $D_t'\subset X_t$, the Fukaya category of compact Lagrangians in $\Fuk(X_t\setminus D_t')$ will be split-generated by the tropical Lagrangians. By a correct system of divisors, we mean the one that corresponds to the large volume limit correspoding to $\Proj(R_{\phi}\otimes\Bbbk)$ under mirror symmetry. Since we do not have any method for identifying such a system of divisors in the generality studied in this section, we only state (partial) versions of homological mirror symmetry over the Novikov field.

To summarize, these considerations suggest that removing the (essentially arbitrarily chosen) system of divisors $\mathbf{E}_t$ from $X_t$ corresponds to approaching a different large complex structure limit point in the complex moduli space of the mirror than the versions in~\cite{SeiQuartic, SheridanHypersurfaces, SheridanSmith, GHHPS} by deleting the systems of divisors in their respective settings. In the case of toric degenerations, the special form of the central fibers of~\eqref{mirrorfamily} indicate that this is most likely the case. A similar phenomenon is present in~\cite{HackingKeating}, as Hacking has explained to us, where the authors use smooth divisors to set up the relative Fukaya category. We thank Denis Auroux and Paul Hacking for discussions about this point.

\subsection{Organization}
What remains of this paper is divided into three parts. The first part sets up general Floer-theoretic machinery that applies to any maximally unipotent degeneration. In \S{\ref{fixedpointfloersection}}, we review fixed point Floer cohomology following Seidel's thesis~\cite{SeiThesis}, and we state some technically convenient variations on the definition. The next section \S{\ref{adaptedtubularneighborhoodssection}} reviews the theory of regularizations developed in~\cite{MTZ, McLeanGrowth, McLeanLogCanonical, TZsmooth}, and contains our basic dynamical computations for fixed point Floer cohomology, e.g., Lemma~\ref{orbitsinstandardneighborhoods}. We develop the appropriate analogues of the (low energy) log-PSS maps of Ganatra--Pomerleano~\cite{GP1, GP2} and Pomerleano~\cite{Pomerleano} in \S{\ref{logPSSmaps-section}}, leading up to a proof of Theorem~\ref{stanleyreisnerisom}. Setting up the moduli spaces we need for this argument relies on the work of Farajzadeh-Tehrani in~\cite{Tehrani, Tehrani2}. In \S{\ref{relfuksection}}, we recall Perutz and Sheridan's construction of the relative Fukaya category~\cite{PerutzSheridan}, and adapt the construction of the twisted closed-open map given by Jeffs--Yao--Zhao in~\cite{JYZ}.
This section also proves Proposition~\ref{closedopenisomorphism}.

In the next part, we apply the general machinery detailed above to toric degenerations. Section~\ref{toricdegeneration-review-section} mostly summarizes results of Gross and Siebert~\cite{GS1, GS2, GSrealaffine}, Ruddat and Siebert~\cite{RuddatSiebert}, and Arg{\"u}z~\cite{Arguzreal} in a convenient form for us. That said, our description of the positive real locus in Theorem~\ref{lagrangianpositivereallocus} has not explicitly been given in the literature as far as we know. Finally, we finish the proof of Theorem~\ref{toric-degeneration-hms-thm} in \S{\ref{toric-degeneration-hms-section}}. The third part of the paper consists of two appendices, the first of which provides the proof of Proposition~\ref{finitelygeneratedprop} and also constructs a one-sided inverse mirror functor, while the second appendix contains the proof of Lemma~\ref{slclemma}, which describes the singularities of the central fiber of the mirror family, following~\cite{Oldfield}.
 
\part{Floer theory}

\section{Review of fixed point Floer cohomology}\label{fixedpointfloersection}
In this section, we will review the construction of fixed point Floer cohomology, and describe the perturbations of the symplectic form on a semistable degeneration $\mathscr{X}\to\mathbb{D}$ as obtained from, say, Proposition~\ref{gsresolutionmodel}. Though our construction of fixed point Fleor cohomology applies more generally, we will focus our attention on the setting of \S{\ref{mainsection}} for simplicity. When the particular choice of fiber is immaterial, we will let $X$, as opposed to $X_t$ for $t\in\mathbb{D}^{\circ}$, denote some smooth fiber of  a degeneration $\mathcal{X}^{\circ}\to\mathbb{D}^{\circ}$.

The fixed point Floer cohomology groups of a symplectomorphism $\phi\colon X\to X$ of a (closed) symplectic manifold $(X,\omega)$ were first introduced in\cite{DostoglouSalamon}, generalizing Floer's original work~\cite{Floer} on Hamiltonian diffeomorphisms, which can be interpreted as the special case $\phi = \id$. A different construction of fixed point Floer cohomology in terms of mapping tori was given by Seidel in~\cite{SeiThesis}. We will follow Seidel's approach, since it appears more natural for the purpose of studying~\eqref{mirrorfamily}. Interpreting the differentials and products on fixed point Floer cochain complexes as sections of symplectic fiber bundles allows us to more easily realize them as curves in the total space of an snc model $\widetilde{\mathcal{X}}\to\mathbb{D}$ for $\mathcal{X}^{\circ}\to\mathbb{D}^{\circ}$.

Our treatment of fixed point Floer cohomology closely follows the relevant parts of~\cite{SeiThesis}. Although that work is written specifically for symplectic $4$-manifolds, it is clear that the methods generalize to symplectic Calabi--Yau manifolds of any dimension when combined with standard arguments, so we will omit most technical details from our summary.

Let $\phi\colon X\to X$ denote the symplectic monodromy~\eqref{symplecticmonodromyrep} of the family $\mathcal{X}^{\circ}\to\mathbb{D}^{\circ}$. We can form the mapping torus of $\phi$, denoted $T_{\phi}$. This is a symplectic fiber bundle over $S^1$, in the sense of~\cite[Definition 7.1]{SeiThesis}, which carries a closed $2$-form $\Theta_{\phi}$ descending from $\omega$ on $\mathbb{R}\times X$. Clearly, $T_{\phi}$ is isomorphic to $\mathcal{X}^{\circ} |_{\partial \mathbb{D}}$, as a symplectic fiber bundle.

Given a function $H\in C^{\infty}(T_{\phi},\mathbb{R})$, we can perturb the form $\Theta_{\phi}$ by setting
\begin{align}\label{hamiltonianperturbationsmappingtori}
\Theta_{\phi,H}\coloneqq\Theta_{\phi}-d(Hdt)
\end{align}
where $dt$ is the pullback of the standard $1$-form on $S^1$. Let $(\phi^t_H)_{t\in\mathbb{R}}$ denote the flow of the pullback of $H$ to $\mathbb{R}\times X$. Then there is a natural identification
\begin{align}\label{perturbedmappingtorusidentification}
(T_\phi,\Theta_{\phi,H})\cong (T_{\phi\circ\phi^H_1},\Theta_{\phi\circ\phi^H_1})
\end{align}
under which $\Theta_{\phi,H}$ and $\Theta_\phi$ coincide on each fiber of the bundle projection $\pi\colon T_{\phi}\to S^1$.

Consider the horizontal subbundle $TT_{\phi}^h\subset TT_{\phi}$ of the tangent bundle defined to be the complement, with respect to $\Theta_{\phi}$, of the vertical subbundle $TT_{\phi}^v\coloneqq\ker(D\pi)$. Let $\mathcal{H}(T_{\phi},\Theta_{\phi})$ denote the space of horizontal sections of $T_{\phi}$, i.e. the sections $\nu\colon S^1\to T_{\phi}$ of $\pi$ for which the image of $D\nu(v)$ is contained in $TT_\phi^h$. After precomposing $\phi$ with a generic Hamiltonian perturbation, all of the fixed points of $\phi$ are nondegenerate, and hence that $(T_{\phi},\Theta_{\phi})$ is nondegenerate in the sense of~\cite[p. 39]{SeiThesis}. Note that sections of $T_{\phi}$ correspond to maps $u\colon\mathbb{R}\to X$ for which $u(t) = \phi(u(t+1))$, and that horizontal sections correspond to constant maps to $X$ satisfying this condition. A horizontal section is nondegenerate if and only if the corresponding point of $\phi$ is a nondegenerate fixed point. This implies that the set $\mathcal{H}(T_{\phi},\Theta_{\phi})$ is finite when $T_{\phi}$ is nondegenerate. When discussing fixed point Floer cohomology, we will assume that $\phi$ is nondegenerate unless explicitly stated otherwise. We typically omit Hamiltonian perturbations of symplectic mapping tori from our notation.

The fixed point Floer cochain group is defined to be
\begin{align*}
CF^k(X,\phi)\coloneqq\bigoplus_{\substack{\nu\in\mathcal{H}(T_{\phi},\Theta_{\phi})\\ \mu(\nu) = k}}o_{\nu}
\end{align*}
where $\mu(\nu)$ denotes the Maslov number of $\nu$ (cf.~\cite[p. 589]{DostoglouSalamon}) and $o_{\nu}$ is the orientation line associated to $\nu$, which is a rank $1$ free $\Lambda$-module~\cite{FloerHoferSalamon}, where $\Lambda$ is the univeral Novikov field over $\Bbbk$~\eqref{universalnovikovfield}. We are using cohomological grading conventions, in contrast with~\cite{SeiThesis}.

The differential on the fixed point Floer cochain complex will count pseudoholomorphic sections of a fiber bundle 
\[ (E_{\phi},\Omega_{\phi}) = \mathbb{R}\times(T_{\phi},\Theta_{\phi})\to C\]
over the infinite cylinder $C \coloneqq \mathbb{R}\times S^1$, where the \textit{fiberwise} symplectic form $\Omega_{\phi}$ is obtained by pulling back $\Theta_{\phi}$ under the projection $E_{\phi}\to T_{\phi}$. The almost complex structures on $E_{\phi}$ which we consider are of the type defined in~\cite[Definition 8.6]{SeiThesis}.
\begin{definition}\label{fibrationsacss}
Let $\mathcal{J}(T_{\phi},\Theta_{\phi})$ denote the space of almost complex structures $J$ on $E_{\phi}$ such that:
\begin{itemize}
\item[(i)] The projection map $E_{\phi}\to C$ is $(J,j)$-holomorphic, where $j$ is the standard almost complex structure on $C$.
\item[(ii)] $J$ is $\mathbb{R}$-invariant.
\item[(iii)] $\Omega_{\phi}(\cdot,J\cdot)$ is a symmetric bilinear form on the tangent bundle $TE_{\phi}$ which is positive-definite on the vertical bundle $TE_{\phi}^v$.
\end{itemize}
\end{definition}
Any almost complex structure $J\in\mathcal{J}(T_{\phi},\Theta_{\phi})$ preserves the splitting 
\[TE_{\phi} = TE^v_{\phi}\oplus TE^h_{\phi}\]
of the tangent bundle into its vertical and horizontal components, by condition (i). Said differently, such an almost complex structure can be written as a direct sum of (tame) almost complex structures
\begin{align}\label{fibrationacssplitting}
    J\coloneqq J^{v}\oplus J^h
\end{align}
on the vertical and horizontal bundles. The horizontal almost complex structure $J^h$ is just the horizontal lift of the standard complex structure on the cylinder, and by (ii) the vertical almost complex structure $J^v$ is an $S^1$-family of tame almost complex structures on $X$ satisfying the $\phi$-twisted periodic condition.

Let $\mathcal{M}(E_{\phi},J)$ denote the space of $J$-holomorphic sections $C\to T_{\phi}$, for $J\in\mathcal{J}(T_{\phi},\Theta_{\phi})$. Given such a $J$, suppose that $\sigma$ is a $J$-holomorphic section of $T_{\phi}$ with positive limit $\nu^+\in\mathcal{H}(T_{\phi},\Theta_{\phi})$ and negative limit $\nu^{-}\in\mathcal{H}(T_{\phi},\Theta_{\phi})$, as in~\cite[(8.2)]{SeiThesis}.  Let
\begin{align*}
\lambda(\sigma) = \int_C\sigma^*\Omega_{\phi}
\end{align*}
denote the \textit{energy}. The pullback $(\nu^{\pm})^*\Theta_{\phi}$ vanishes, since it is a $2$-form on $S^1$, so $\lambda(\sigma)$ is finite by the decay condition in~\cite[(8.2)]{SeiThesis}.

Fix a pair of horizontal sections $\nu^{\pm}\in\mathcal{H}(T_{\phi},\Theta_{\phi})$ with index difference satisfying $\mu(\nu^{+})-\mu(\nu^{-}) = k$, and let $\pi_2(\nu^{-},\nu^{+})$ denote the set of homotopy classes of sections $C\to E_{\phi}$ with positive limit $\nu^{+}$ and negative limit $\nu^{-}$. We can consider the subspaces
\begin{align}\label{cylindersectionmodulispaces}
\widetilde{\mathcal{M}}(E_{\phi},J;B)
\end{align}
of $\mathcal{M}(E_{\phi},J)$ consisting of all sections in the class $B \in \pi_2(\nu^{-},\nu^{+})$. Each of these subspaces carry a natural $\mathbb{R}$-action by translations in the domains $C$. By~\cite{SalamonZehnder}, the virtual dimension of~\eqref{cylindersectionmodulispaces} is given by
\begin{align*}
\ind(\sigma) = \mu(\nu^{+})-\mu(\nu^{-})\,.
\end{align*}
Although~\cite{SeiThesis} indexes the moduli spaces of cylinders by their index and energy, it is clear that this data is determined by $B$. Let $\lambda(B)$ denote the energy of any $u\colon C\to E_{\phi}$ such that $[u] = B\in\pi_2(\nu^{-},\nu^{+})$, and define the index $\ind(B)$ similarly. Let $\mathcal{M}(E_{\phi},J;B)$ denote the quotient $\widetilde{\mathcal{M}}(E_{\phi},J;B)/\mathbb{R}$. The relevant version of Gromov compactness for sections of $T_{\phi}$ appears in~\cite[Theorem 12.7]{SeiThesis}. This is essentially just the Gromov compactness theorem for Floer cylinders applied to $E_{\phi}$ equipped with the symplectic form
\begin{align}\label{symplecticformfromfiberwise}
    \Omega_{\phi}'\coloneqq\Omega_{\phi}+c(ds\wedge dt)
\end{align}
for some constant $c>0$, and where $(s,t)\in\mathbb{R}\times S^1 = C$. For a given $J\in\mathcal{J}(T_{\phi},\Theta_{\phi})$, we can choose $c$ so that $\Omega_{\phi}'$ tames $J$. Although~\cite{SeiThesis} is only written for symplectic $4$-manifolds, it is clear that the compactness result holds in any dimension. Note that since $X$ is Calabi--Yau, all sphere bubbles have Chern number $0$, so the conditions of~\cite[Theorem 12.7]{SeiThesis} are satisfied. In this setting, all sphere bubbles are contained in a single fiber of $T_{\phi}$~\cite[Definition 11.9]{SeiThesis}. Combined with a gluing theorem~\cite[Theorem 12.5]{SeiThesis}, this means that the spaces $\mathcal{M}(E_{\phi},J;B)$ have natural Gromov compactifications by broken $J$-holomorphic sections~\cite[Definition 12.1]{SeiThesis}, with sphere bubbles attached.

We achieve transversality for these spaces of sections by combining the techniques used to construct Hamiltonian Floer cohomology for Calabi--Yau manifolds in~\cite{HoferSalamonNovikov} with the transversality results for sections in~\cite[\S{11}]{SeiThesis}. The only moduli spaces relevant to the definition of fixed point Floer cohomology are those of virtual dimension $\leq 1$. This means that by choosing $J$ generically, we can avoid sphere bubbles for dimension reasons; since the set of points lying on a $J$-holomorphic sphere with $c_1 = 0$ forms a codimension $4$ subspace of $X$, roughly speaking. See in particular~\cite[Proposition 14.2]{SeiThesis}.

Continuing with the definition of Floer cohomology, we define the differential
\begin{align*}
d\colon CF^*(X,\phi)\to CF^{*+1}(X,\phi)
\end{align*}
to be
\begin{align}\label{fixedpointfloerdifferentialdefinition}
d\coloneqq\sum_{\substack{B\in\pi_2(\nu^{-},\nu^{+}) \\ \ind(B) = 1}} \sum_{u\in\mathcal{M}(E_{\phi},J;B)}T^{\lambda(B)}\sigma_u \,.
\end{align}
Here, $\sigma_u\colon o_{\nu^{-}}\to o_{\nu^{+}}$ denotes the isomorphism of orientation lines associated to the element $u\in\mathcal{M}(E_{\phi},J;B)$. The fact that $d^2 = 0$ follows from the gluing theorem for sections.
\begin{definition}\label{fixedpointfloercohomologydefinition}
The \textit{fixed point Floer cohomology} $HF^*(X,\phi)$ groups of $(X,\phi)$ are defined to be the cohomology groups of the chain complex $(CF^*(X,\phi),d)$.
\end{definition}
It is a standard fact that the cohomology groups $HF^*(X,\phi)$ are independent of the Hamiltonian perturbations and almost complex structures chosen in their construction, up to isomorphism. Moreover, these groups are invariant under deformations of symplectic structure in the sense described on~\cite[p. 34]{SeiThesis}.

\subsection{Pseudoholomorphic multisections and domain-dependent complex structures}\label{multiplecoverfixedpointfloer}
To define fixed point Floer cohomology groups $HF^*(X,\phi^d)$, where $d\neq0$, we can apply the construction given above to $(T_{\phi^d},\Theta_{\phi^d})$. There is an alternative way to construct these groups by counting pseudoholomorphic \textit{multisections} of $E_{\phi}\to C$. We can avoid equivariant transversality problems by introducing domain-dependent almost complex structures as follows.

First note that the symplectic fiber bundle $(T_{\phi^d},\Theta_{\phi^d})$ is naturally the pullback of $(T_{\phi},\Theta_{\phi})$ under the covering map
\begin{align*}
    p_d\colon S^1 &\to S^1 \\
    z &\mapsto z^d \,.
\end{align*}
There is also a natural correspondence of horizontal sections of these bundles. Incorporating Hamiltonian perturbations into this description, recall that if $H\in C^{\infty}(T_{\phi},\mathbb{R})$ is a smooth function then the perturbed torus bundle $(T_{\phi},\Theta_{\phi,H})$ is naturally identified with the symplectic mapping torus of $\phi\circ\phi_1^H$~\eqref{perturbedmappingtorusidentification}. Thus the horizontal sections of
\[(T_{(\phi\circ\phi^H_1)^d},\Theta_{(\phi\circ\phi^H_1)^d})\]
correspond to $d$-fold covers of horizontal sections of $T_{\phi}$ under the natural map
\[T_{\phi^d} = p_d^*T_{\phi}\to T_{\phi} \,.\]
Similar remarks apply to $E_{\phi^d}$ and $E_{\phi}$.

Suppose that $J\in\mathcal{J}(T_{\phi^d},\Theta_{\phi^d})$, and recall from~\eqref{fibrationacssplitting} that this determines an $S^1$-family of tame almost complex structures on $X$. Under the natural map $E_{\phi^d}\to E_{\phi}$, any $J$-holomorphic section $\tilde{u}$ of $E_{\phi^d}\to C$ corresponds to a $d$-fold multisection $u$ of $E_{\phi}$. This multisection is easily seen to the pseudoholomorphic with respect to a domain-dependent almost complex structure on $E_{\phi}$ whose values lie in $\mathcal{J}(T_{\phi},\Theta_{\phi})$. Conversely, given any pseudoholomorphic cylinder (with respect to a domain-dependent almost complex structure of this type) $\tilde{u}\colon C\to E_{\phi}$, the composition of this map with the bundle projection is an ordinary holomorphic map from the cylinder to itself, which must be a covering map by the Riemann--Hurwitz formula.

Thus the fixed point Floer complex of $\phi^d$, as we have defined it, can be canonically identified with chain complexes whose differential is defined by counting pseudoholomorphic cylinders in $E_{\phi}$. This perspective seems essential for proving Theorem~\ref{stanleyreisnerisom}, whereas it is more convenient to work with Floer complex defined along the lines of~\cite{SeiThesis} when studying the twisted closed-open map~\eqref{twistedclosedopenfirstmention}.

\subsection{The pair of pants product}
A straightforward application of the general results of~\cite{SeiThesis} on counting sections allows us to define a product~\eqref{pairofpantsproduct} on fixed point Floer cohomology. Some details of this are, for example, described in~\cite{SalamonProduct}, where a symplectic fiber bundle over the pair of pants $\mathbb{CP}^1\setminus\lbrace 0,1,\infty\rbrace$ which has holonomy $\phi^{-d_0}$ and $\phi^{-d_1}$ about the punctures at $0$ and $1$, and holonomy $\phi^{d_0+d_1}$ about the puncture at $\infty$, is constructed. In the definition of the product, the punctures at $0$ and $1$ are treated as negative punctures, and the puncture at $\infty$ is treated as a positive puncture. We will provide a slightly different construction of this fiber bundle than the one in~\cite{SalamonProduct} which allows us to naturally realize pseudoholomorphic sections contributing to~\eqref{pairofpantsproduct} as pseudoholomorphic pairs of pants in $\mathcal{X}^{\circ}\cong E_{\phi}$, with respect to a domain-dependent almost complex structure.

Define the bundle $E_{\phi^{d_0},\phi^{d_1}}\to\mathbb{CP}^1\setminus\lbrace 0,1,\infty\rbrace$ as the pullback bundle $p^*E_{\phi}$ under the branched covering map
\begin{align*}
p\colon\mathbb{CP}^1\setminus\lbrace 0,1,\infty\rbrace&\to\mathbb{C}^* \\
z&\mapsto z^{d_0}(z-1)^{d_1} \,.
\end{align*}
This is a symplectic fiber bundle, equipped with a fiberwise nondegenerate $2$-form $\Omega_{\phi^{d_0},\phi^{d_1}}$, in the sense of~\cite[Definition 7.1]{SeiThesis}, although the pullback symplectic form $p^*\Omega_{\phi}$ is not a symplectic form on the total space.

We can choose negative cylindrical ends
\begin{align*}
\epsilon_{i}\colon\mathbb{R}_{<0}\times S^1\to\mathbb{CP}^1\setminus\lbrace 0,1,\infty\rbrace
\end{align*}
for $i\in\lbrace 0,1\rbrace$ given by
\begin{align*}
\lim_{s\to-\infty}\epsilon_i(s,t) = i
\end{align*}
near the negative punctures and a positive cylindrical end
\begin{align*}
\epsilon_{\infty}\colon\mathbb{R}_{>0}\times S^1\to\mathbb{CP}^1\setminus\lbrace 0,1,\infty\rbrace
\end{align*}
such that
\begin{align*}
\lim_{s\to\infty}\epsilon_{\infty}(s,t) = \infty \,.
\end{align*}

In these cylindrical ends, we have that $\epsilon_i^*E_{\phi^{d_0},\phi^{d_1}}$ is isomorphic to $E_{\phi^{d_i}}\mid_{\mathbb{R}_{<0}\times S^1}$ for $i = 0,1$, and that $\epsilon_{\infty}^*E_{\phi^{d_0},\phi^{d_1}}$ isomorphic to $E_{\phi^{d_0+d_1}}\mid_{\mathbb{R}_{>0}\times S^1}$ as symplectic fiber bundles.
\begin{definition}\label{fibrationsacsspants}
Let $\mathcal{J}(E_{\phi^{d_0},\phi^{d_1}},\Omega_{\phi^{d_0},\phi^{d_1}})$ denote the space of almost complex structure $J$ on $E_{\phi^{d_0},\phi^{d_1}}$ such that:
\begin{itemize}
\item[(i)] The projection map $E_{\phi^{d_0},\phi^{d_1}}\to\mathbb{CP}^1\setminus\lbrace 0,1,\infty\rbrace$ is $(J,j)$-holomorphic, where $j$ is the standard almost complex structure on $\mathbb{CP}^1\setminus\lbrace 0,1,\infty\rbrace$.
\item[(ii)] $\Omega_{\phi^{d_0},\phi^{d_1}}(\cdot,J\cdot)$ is a symmetric bilinear form on the tangent bundle $TE_{\phi^{d_0},\phi^{d_1}}$ which is positive-definite on the vertical bundle $TE_{\phi^{d_0},\phi^{d_1}}^v$.
\item[(iii)] In each cylindrical end, $J$ restricts to an almost complex structure in $\mathcal{J}(T_{\phi^{d_i}},\Theta_{\phi^{d_i}})$ (in the negative cylindrical ends) or in $\mathcal{J}(T_{\phi^{d_0+d_1}},\Theta_{\phi^{d_0+d_1}})$ (in the positive end.) In particular, $J$ is translation-invariant in the cylindrical ends.
\end{itemize}
\end{definition}
Choose a generic $J\in \mathcal{J}(E_{\phi^{d_0},\phi^{d_1}})$. For fixed generators $x_i\in\mathcal{H}(T_{\phi^{d_i}},\Theta_{\phi^{d_i}})$, where $i = 0,1$, and $x_{\infty}\in\mathcal{H}(T_{\phi^{d_0+d_1}},\Theta_{\phi^{d_0+d_1}})$, let $\pi_2(x_0,x_1;x_{\infty})$ denote the set of homotopy classes of secions of $E_{\phi^{d_0},\phi^{d_1}}$ asymptotic to the horizontal section $x_i$ near the puncture $i\in\lbrace 0,1,\infty\rbrace$. As explained in~\cite[\S{4}]{SeiPants}, we need to use Hamiltonian perturbation terms to achieve transverality for pseudoholomorphic pairs of pants. These take the form of a perturbation term
\[ Y_z\in\Omega^1(\mathbb{CP}^1\setminus\lbrace 0,1,\infty\rbrace, T(E_{\phi}))\]
which we can take to be constant along the cylindrical ends (cf.~\cite[\S{8}]{ShelukhinZhao}.)

For any $B\in\pi_2(x_0,x_1;x_{\infty})$, we denote its energy by $\lambda(B)\in\Lambda$ and its index by $\ind(B)\in\mathbb{Z}$. Let $\mathcal{M}(E_{\phi^{d_0},\phi^{d_1}},J;B)$ denote the moduli space of inhomgeneous sections $u$ of $E_{\phi^{d_0},\phi^{d_1}}$, meaning that
\[ (Du-Y_z)\circ J = j\circ(Du-Y_z) \]
where $j$ denotes the almost complex structure on the pair of pants and $Y_z$ denotes the value of $Y$ at $z\in\mathbb{CP}^1\setminus\lbrace 0,1,\infty\rbrace$, such that $[u] = B$. We can show that the moduli spaces with $\ind(B)\leq 1$ are transversely cut out for generic almost complex structures.

We define the pair of pants product
\begin{align}\label{pantsproducthchainlevel}
\star\colon CF^*(X,\phi^{d_0})\widehat{\otimes}_{\Lambda_0} CF^*(X,\phi^{d_1})\to CF^*(X,\phi^{d_0+d_1}) \\
-\star -\coloneqq\sum_{\substack{B\in\pi_2(x_0,x_1;x_{\infty}) \\ \ind(B) = 1}}\sum_{u\in\mathcal{M}(E_{\phi^{d_0},\phi^{d_1}},J;B)}\sigma_u(-,-)
\end{align}
where $\sigma_u\colon o_{x_0}\otimes o_{x_1}\to o_{x_{\infty}}$ denotes the isomorphism of orientation lines associated to any $u\in\mathcal{M}(E_{\phi^{d_0},\phi^{d_1}},J;B)$. The moduli spaces ${M}(E_{\phi^{d_0},\phi^{d_1}},J;B)$ of virtual dimension $1$ also admit Gromov compactifications by broken pseudoholomorphic sections with sphere bubbles which project to constant maps in the base. This is a consequence of the standard Gromov compactness theorem, where we equip $E_{\phi^{d_0},\phi^{d_1}}$ with a symplectic form analogous to~\eqref{symplecticformfromfiberwise} that tames the given almost complex structure $J$. By enumerating the boundary strata of these moduli spaces one shows that that the pair of pants product descends to cohomology, obtaining the product~\eqref{pairofpantsproduct}. Similar considerations imply the following.
\begin{lemma}
The ring $R_{\phi}$~\eqref{mirrorring} is associative.
\end{lemma}
\begin{proof}
This is a standard argument using Gromov compactness applied to $1$-dimensional moduli spaces of sections of symplectic fiber bundles over a sphere with three negative punctures and one positive puncture. We can construct the symplectic fiber bundles involved by pulling back $E_{\phi}$ under a branched covering map of the bases, and we define these moduli spaces using almost complex structures analogous, in the obvious way, to those in Definitions~\ref{fibrationsacss} and~\ref{fibrationsacsspants}.
\end{proof}
We finish this section with a key technical lemma in which we interpret~\eqref{pantsproducthchainlevel} as a count of curves in $E_{\phi}$. This will be used in the proof of Theorem~\ref{stanleyreisnerisom}.

\begin{lemma}\label{pantsinthedegenerationlemma}
Any pseudoholomorphic section of $E_{\phi^{d_0},\phi^{d_1}}$, with respect to an almost complex structure as in Definition~\ref{fibrationsacsspants}, can be identified with a pseudoholomorphic pair of pants in $E_{\phi}$, defined with respect to a domain-dependent almost complex structure on $\mathbb{CP}^1\setminus\lbrace 0,1,\infty\rbrace$ valued in $\mathcal{J}(E_{\phi})$. Conversely, any such pseudoholomorphic pair of pants in $E_{\phi}$ which is asymptotic to $x_{\infty}\in $ can be realized as the image of a section of some $E_{\phi^{d_0},\phi^{d_1}}$.
\end{lemma}
\begin{proof}
    The forward direction is trivial, since there is a canonical map $E_{\phi^{d_0},\phi^{d_1}}\to E_{\phi}$ coming from the pullback bundle construction. For the converse, note that any pseudoholomorphic pair of pants asymptotic to the generators $x_i\in\mathcal{H}(T_{\phi^{d_i}},\Theta_{\phi^{d_i}})$, where $i = 0,1$, and $x_{\infty}\in\mathcal{H}(T_{\phi^{d_0+d_1}},\Theta_{\phi^{d_0+d_1}})$, thought of as multiple covers of sections in $\mathcal{H}(T_{\phi},\Theta_{\phi})$, must be a multisection of $C$ with a single branch point, by the Riemann--Hurwitz theorem. The correspondence between domain-dependent almost complex structures and almost complex structures on $E_{\phi^{d_0},\phi^{d_1}}$ follows from the same considerations outlined in \S{\ref{multiplecoverfixedpointfloer}}.

    In more detail, any almost complex structure on $E_{\phi^{d_0},\phi^{d_1}}$ splits as in~\eqref{fibrationacssplitting}, giving us a family of almost complex structures on the fibers. Thus when we project a section of this bundle to $E_{\phi}$, we can use~\eqref{fibrationacssplitting} again to obtain a domain-dependent almost complex structure valued in $\mathcal{J}(T_{\phi},\Theta_{\phi})$ on the pair of pants.
\end{proof}

\section{Systems of adapted tubular neighborhoods}\label{adaptedtubularneighborhoodssection}
As explained in \S{\ref{mainsection}}, we use the well-known results of~\cite{MTZ} to organize our computations of Floer-theoretic invariants. The results allow us to construct a system of tubular neighborhoods for the strata of the central fiber of an snc model $\mathscr{X}\to\mathbb{D}$ for $\mathcal{X}^{\circ}\to\mathbb{D}^{\circ}$ (cf. Definition~\ref{Omega-regularization-definition} below.) We also make the elementary observation in Lemma~\ref{orbitsinstandardneighborhoods} that, using the regularized symplectic forms~\eqref{standardlocalformtubularnbhd} on these tubular neighborhoods, the symplectic monodromy~\eqref{symplecticmonodromyrep} can be made to agree with fiberwise rotations in these tubular neighborhoods. In this section we also explain how to perturb the symplectic monodromy to make it nondegenerate in a controlled way, and we define a local version of fixed point Floer cohomology. This section and the subsequent one follow~\cite{GP2}.

\subsection{Symplectic regularizations in snc models}\label{symplectic-regularization-section} Given a family $\mathcal{X}^{\circ}\to\mathbb{D}^{\circ}$ of Calabi--Yau manifolds of the type considered in \S{\ref{mainsection}}, fix a semistable degeneration $\pi_{\mathscr{X}}\colon\mathscr{X}\to\mathbb{D}$ with an expression~\eqref{KXexpression}. Let $\mathbf{D}\coloneqq\mathscr{X}_0 = \pi^{-1}_{\mathscr{X}}(0)$ denote the special fiber of this family. We write $\mathbf{D}$ as the union of its smooth irreducible components
\begin{align*}
    \mathbf{D}\coloneqq\bigcup_{i\in R}D_i
\end{align*}
where $R$ is a finite index set. 

Choose a relatively ample line bundle $\mathscr{E}\to\mathscr{X}$ whose restriction to $\mathcal{X}^{\circ}$ is
\begin{align*}
    \mathscr{E}\mid_{\mathcal{X}^{\circ}} = \mathcal{E} \,,
\end{align*}
the polarization~\eqref{polarizationassumption} defining the symplectic form on $\mathcal{X}^{\circ}$. This line bundle similarly determines a symplectic form $\Omega$ on $\mathscr{X}$, which restricts to the given symplectic form on $\mathcal{X}^{\circ}$. More precisely, using a relative version of the Kodaira embedding theorem (e.g.~\cite[29.38.4]{StacksProject}), some power of $\mathscr{E}$ determines an embedding of $\mathscr{X}$ into $\mathbb{CP}^N\times\mathbb{D}$ for some large $N>0$, and we can take $\Omega$ to be the pullback of the Fubini--Study form on $\mathbb{CP}^N$ plus a large multiple of the standard symplectic form on the disk. Note that by Proposition~\ref{gsresolutionmodel}, we can take $\mathscr{X}$ to be projective over $\mathbb{D}$.

In several places below, see Section \S{\ref{relfuksection}}, we will need to make use of stabilizing \textit{horizontal divisors} to achieve transversality for various moduli spaces of pseudoholomorphic curves using the method of~\cite{CieliebakMohnke}.
\begin{definition}
    A \textit{stabilizing divisor} for $(\mathscr{X},\mathbf{D})$ is a relatively ample divisor
    \begin{align*}
        \mathbf{E}\coloneqq\bigcup_{q\in Q}E_q\subset\mathscr{X}
    \end{align*}
    where $Q$ is a finite index set, such that the \textit{full divisor}
    \begin{align}\label{fulldivisors}
        \mathbf{V}\coloneqq\mathbf{D}\cup\mathbf{E}
    \end{align}
    is also a simple normal crossings divisor. Letting $S = R\sqcup Q$, we will sometimes write
    \begin{align}
        \mathbf{V}\coloneqq\bigcup_{i\in S}V_i = \bigcup_{i\in R}D_i\cup\bigcup_{q\in Q}E_q
    \end{align}
    as the union of its irreducible components. Furthermore, we demand that each $E_q$ is relatively ample and that $\pi\mid_{E_q}\colon E_q\to\mathbb{D}$ is a submersion for all $q\in Q$.
\end{definition}
In practice, we will always take the components $E_q$ of $\mathbf{E}$ to be of the form $s_q^{-1}(0)\subset\mathscr{X}$, where $s_q\in\Gamma(\mathscr{E})$ is a generic section of the relatively ample line bundle $\mathscr{E}\to\mathscr{X}$.

We will construct an $\Omega$-\textit{regularization} of $\mathbf{V}$ using the results of~\cite[Theorem 2.13]{MTZ}. Although we are primarily interested in using the regularization to determine the fixed points of some representative $\phi$ of the symplectic monodromy (cf.~\eqref{symplecticmonodromyrep}), which we could in principle achieve with a regularization of $\mathbf{D}$ alone, we also need this representative of $\phi$ to preserve each component of $\mathbf{E}$. This can be achieved by regularizing the full divisor $\mathbf{V}$.

\begin{definition}
    The \textit{divisorial stratification} of $\mathbf{V}$ has strata which are indexed by subets $I\subset S$. Define
\begin{align}
V_I\coloneqq\bigcap_{i\in I}V_i \,.
\end{align}
We refer to the open strata of the stratification as
\begin{align}
\overset{\circ}{V_I} = V_I\setminus\bigcup_{i\not\in I}V_i
\end{align}
for any nonempty $I$.
\end{definition}

There are standard local models for symplectic forms near a stratum $V_I$ with $I\subset S$. We will first describe these neighborhoods for an irreducible component $V_i$ of $\mathbf{D}$ following~\cite{GP2, MTZ, McLeanLogCanonical}. Let $\pi: N_{\mathscr{X}} V_i = N\to V_i$ denote the normal bundle of $V_i$ in $\mathscr{X}$, which is a complex line bundle. Call $(\rho,\nabla)$ a \textit{Hermitian structure} on $\pi\colon N\to V_i$ if $\rho$ is a Hermitian metric on $N$, and $\nabla$ is a $\rho$-compatible connection on $N$. Let $j$ denote the complex structure on $N$, and say that a $(\rho,\nabla)$ is \textit{compatible} with a symplectic structure $\Omega$ on $N$ if $\Omega(\cdot,j\cdot) = \operatorname{Re}\rho(\cdot,\cdot)$ as usual. The Hermitian structure induces a splitting
\begin{align*}
    TN\cong TN^v\oplus TN^h
\end{align*}
into vertical and horizontal components, where $TN^v = \ker(D\pi)$. There is a standard angular $1$-form $d\varphi\in\Gamma(T^*(N\setminus V_i)^v)$ on the vertical tangent space. We lift this to a connection $1$-form $\theta_e\in\Omega^1(N\setminus V_i)$ by requiring that for each point $p\in N\setminus V_i$,
\begin{align}
\theta_e\mid_{TN_p^v} &= d\varphi \\
\theta_e\mid_{TN_p^h} &= 0 \,.
\end{align}
Letting $\Omega_{V_i}$ denote the symplectic form on $V_i$ obtained by restricting $\Omega$ on $\mathscr{X}$, we can define a $2$-form
\begin{align}\label{standardneighborhoodsymplecticform}
\Omega_{(\rho,\nabla)}\coloneqq\pi^*\Omega_{V_i}+\frac{1}{2}d(\rho\theta_e)
\end{align}
which is symplectic near $V_i$.

If $V_I$ is a stratum of $\mathbf{V}$ of higher codimension, then given a collection of Hermitian line bundles $\lbrace N_i = (N_i,\rho_i,\nabla_i)\rbrace_{i\in I}$ where $N_i = N_{\mathscr{X}} V_i |_{V_I}$, we obtain connection $1$-forms $\lbrace\theta_{e,i}\rbrace_{i\in J}$, and we can define a $2$-form on $\bigoplus_{i\in I}N_i$ which is symplectic near $V_I$ by
\begin{align}\label{standardlocalformtubularnbhd}
    \Omega_{\lbrace(\rho_i,\nabla_i)\rbrace_{i\in I}}\coloneqq\pi^*\Omega_{V_I}+\frac{1}{2}\sum_{i\in I}\pi_{I,{i'}}^*d(\rho_i\theta_{e,{i'}})
\end{align}
where $\Omega_{V_I} = \Omega\mid_{V_I}$ and $\pi_{I,{i'}}\colon\bigoplus_{i\in I}N_i\to N_{i'}$ is the projection map.

A \textit{regularization} of $V_I$, as defined in~\cite[Definition 2.8]{MTZ}, is a tubular neighborhood $\psi_I\colon U_I\to\mathscr{X}$, where $U_I\subset N_{\mathscr{X}}V_I$ is a neighborhood of the zero section in the normal bundle to $V_I$, such that the normal component
\begin{align*}
    D\psi_I\mid_{V_I}\colon N_{NV_I}V_I = N_{\mathscr{X}}V_I\to N_{\mathscr{X}}V_I
\end{align*}
is the identity on $N_{\mathscr{X}}V_I$. Following~\cite[Definition 2.10]{MTZ}, we say that a \textit{system of regularizations} is a collection $\lbrace\psi_I\colon U_I\to\mathscr{X}\mid I\subset S\rbrace$ of regularizations, where $\psi_I$ is a regularization of $V_I$, such that
\begin{align}
    \psi_I((N_{V_I'}V_I)\cap U_I) = V_{I'}\cap\psi_I(U_I)
\end{align}
for all $I'\subset I\subset S$ (cf.~\cite[(2.14)]{GP2}.)

Given a system of regularizations, one can ask them to be suitably compatible with passing to tubular neighborhoods of different strata. To explain this, we need some preliminary notation. There is a canonical identification of total spaces (but not of bundles)
\begin{align*}
    N_{\mathscr{X}}V_I = \pi_{I;I'}^*(N_{V_I'}V_I)
\end{align*}
for each $I'\subset I$, where
\begin{align*}
    \pi_{I;I'}\colon N_{V_{I\setminus I'}}V_I\to V_I
\end{align*}
is the projection map. There are natural maps of normal bundles
\begin{align*}
    \mathfrak{D}\psi_{I;I'}\colon\pi^*_{I;I'}(N_{V_I'}V_I)\mid_{\psi^{-1}_I(V_I')}\to N_{\mathscr{X}}V_I'\mid_{V_I'\cap\psi_I(U_I)}
\end{align*}
defined in terms of $d\psi_I$ and other natural quotients and inclusions of tangent and normal bundles (see~\cite[(2.15)]{MTZ} or~\cite[(5-7)]{McLeanLogCanonical} for the precise definition.) The map $\mathfrak{D}\psi_{I;I'}$ identifies the normal bundle of $N_{V_{I\setminus I'}}V_I$ inside $N_{\mathscr{X}}V_I$ near the zero section with the normal bundle $N_{\mathscr{X}}V_{I'}$ near $V_I$ using the derivative of $\psi_I$. We say that a \textit{regularization} for $\mathbf{V}$ is a system of regularizations $\lbrace\psi_I\colon U_I\to \mathscr{X}\mid I\subset S\rbrace$ such that
\begin{align*}
    \mathfrak{D}\psi_{I;I'}(U_I) = U_{I'}\mid_{V_{I'}\cap\psi_I(U_I)}
\end{align*}
and
\begin{align*}
    \psi_I = \psi_{I'}\circ\mathfrak{D}\psi_{I;I'}\mid_{U_I}
\end{align*}
We can now describe the systems of tubular neighborhoods which we will need.

\begin{definition}[{\cite[Definition 2.12]{MTZ} and~\cite[Definition 2.2]{GP2}}]\label{Omega-regularization-definition}
    Let $\Omega$ denote the symplectic form on the total space $\mathscr{X}$. An $\Omega$-regularization\footnote{Some authors, e.g.~\cite{Pomerleano}, refer to $\Omega$-regularizations simply as regularizations, but we prefer to distiniguish between the symplectic and topological notions.}
    \begin{align*}
        \mathfrak{R} = (\psi_I,(\rho_{I,i},\nabla_{I,i})_{i\in I}))_{I\subset S}
    \end{align*}
    of $\mathbf{V}$ is a regularization $\lbrace\psi_I\colon U_I\to \mathscr{X}\mid I\subset S\rbrace$ of $\mathbf{V}$, together with a collection of Hermitian structures $\lbrace(\rho_{I,i},\nabla_{I,i})\rbrace_{i\in I}$ on the normal bundles
    \begin{align*}
        N_{\mathscr{X}}V_i\mid_{V_I} = N_{V_{I\setminus\lbrace i\rbrace}}V_I
    \end{align*}
    such that $\psi_I^*\Omega = \Omega_{\lbrace(\rho_{I,i},\nabla_{I,i})\rbrace_{i\in I}}$, and such that the maps $\mathfrak{D}\psi_{I;I'}$ are be products of Hermitian isomorphisms.
\end{definition}
\begin{theorem}[{\cite[Theorem 2.13]{MTZ}}]\label{MTZthm}
The symplectic form $\Omega_0 \coloneqq \Omega$ can be deformed through a family of symplectic forms $\lbrace\Omega_t\rbrace_{t\in[0,1]}$ with $[\Omega_t] = [\Omega]\in H^2(\mathscr{X})$ such that $\mathbf{D}$ admits an $\Omega_1$-regularization.
\end{theorem}
We will identify the tubular neighborhoods $U_I$ with their images in $\mathscr{X}$ without referring to the regularizations. By shrinking the neighborhoods $U_I$ if necessary, we can assume that $U_I = \bigcap_{i\in I}U_i$. For a disk $\mathbb{D}_{\epsilon}\subset\mathbb{D}$ of sufficiently small radius $\epsilon>0$, the restriction of $\mathscr{X}$ to $\mathbb{D}_{\epsilon}$ is contained in the union of regularized neighborhoods $U_I$ for the strata $D_I$ of the central fiber, i.e.,
\begin{align}\label{smallrestrictedfamily}
    \mathscr{X}_{\mathbb{D}_{\epsilon}}\subset\bigcup_{I\subset R} U_I.
\end{align}

Of course, if we just replace $\Omega$ with the regularized symplectic form $\Omega_1$, it may not determine a symplectic fibration in the sense of~\cite[Definition 7.1]{SeiThesis}, as the family of symplectic forms $(\mathcal{X}^{\circ},\Omega_1\mid_{(T\mathcal{X}^{\circ})^v})$ may not be locally trivial, where the vertical bundle $(T\mathcal{X}^{\circ})^v = \ker(d\pi)$ is defined with respect to the original projection $\pi\colon \mathcal{X}^{\circ}\to\mathbb{D}^{\circ}$. This would mean that we would not be able to parallel transport along curves in $\mathbb{D}^{\circ}$ (cf.~\cite[Lemma 7.2]{SeiThesis}.) To correct for this, we borrow a result from~\cite{McLeanLogCanonical}, which allows us to define a small perturbation of $\pi$ with respect to which $\Omega_1$ determines a symplectic fibration.

From a regularization $\psi_i\colon U_i\to\mathscr{X}$ for $D_i$, we construct, a line bundle
\[ \mathcal{O}_{\mathscr{X}}(D_i)\coloneqq(\pi^*_{N_{\mathscr{X}}D_i}N_{\mathscr{X}}D_i)\mid_{U_i}\sqcup((\mathscr{X}\setminus D_i)\times\mathbb{C})/\sim\]
where, for 
\begin{align*}
    (v,cv) &\in (\pi^*_{N_{\mathscr{X}}D_i}N_{\mathscr{X}}D_i)\mid_{U_i} \,, \\
    (\psi_i(v),c) &\in (\mathscr{X}\setminus D_i)\times\mathbb{C} \,,
\end{align*}
we set $(v,cv)\sim(\psi_i(v),c)$ for all $v\in N_{\mathscr{X}}D_i\setminus D_i$ and $c\in\mathbb{C}$. The projection map
\[ \pi_{\mathcal{O}_{\mathscr{X}}(D_i)}\colon \mathcal{O}_{\mathscr{X}}(D_i)\to\mathscr{X}\]
is defined by setting
\begin{align*}
    \pi_{\mathcal{O}_{\mathscr{X}}(D_i)}\mid_{\pi^*_{N_{\mathscr{X}}D_i}N_{\mathscr{X}}D_i}(v,w) &= \psi_i(v) \\
    \pi_{\mathcal{O}_{\mathscr{X}}(D_i)}\mid_{(\mathscr{X}\setminus D_i)\times\mathbb{C}}(x,c) &= x \,.
\end{align*}
There is a canonical section $\sigma_{D_i}\colon\mathscr{X}\to\mathcal{O}_{\mathscr{X}}(D_i)$ of this bundle given by
\begin{align*}
    \sigma_{D_i}(x)\coloneqq
    \begin{cases}
        (\psi_i^{-1}(x),\psi_i^{-1}(x)) & \text{ if }x\in\psi_i(U) \\
        (x,1) & \text{ if } x\in\mathscr{X}\setminus D_i \,.
    \end{cases}
\end{align*}
We also set $\mathcal{O}_{\mathscr{X}}(0) = \mathcal{O}_{\mathscr{X}}(\emptyset)$ to be the trivial line bundle $\mathscr{X}\times\mathbb{C}$. This construction generalizes the line bundle associated to a (Cartier) divisor, justifying the notation. Note that $\mathcal{O}_{\mathscr{X}}(D_i)$ depends on the choice of regularization.

Define the line bundle
\begin{align}\label{semistablelinebundletrivial}
    \mathcal{O}_{\mathscr{X}}(\sum_{i\in R} D_i)\coloneqq\bigotimes_{i\in R}\mathcal{O}_{\mathscr{X}}(D_i) \,.
\end{align}
This line bundle also comes with a canonical section
\begin{align*}
    \sigma\coloneqq\bigotimes_{i\in R}\sigma_{D_i} \,.
\end{align*}
The construction of $\mathcal{O}_{\mathscr{X}}(D_i)$ shows that we have a canonical identification
\begin{align}\label{structuresheafnormalDIidentity}
    \mathcal{O}_{\mathscr{X}}(D_i)\mid_{D_I} = N_{\mathscr{X}}D_i\mid_{D_{I}}
\end{align}
for all $i\in I\subset R$, and combining this with the canonical sections above gives us maps
\begin{align*}
    \Pi_{D_i;I}\colon N_{\mathscr{X}}D_I = \bigoplus_{j\in I}N_{\mathscr{X}}D_j\mid_{D_{I}} &\to\mathcal{O}_{\mathscr{X}}(D_i)\mid_{D_I}  \\
    (v_j)_{j\in I}\mapsto\begin{cases}
        v_i \text{ if }i\in I \\
        \sigma_{D_i}(x)\text{ if } i\not\in I
    \end{cases}
\end{align*}
where $(v_j)_{j\in I}\in N_{\mathscr{X}}D_I\mid_x$ for $x\in D_I$. Taking the tensor product of these maps, we get a map of disk bundles
\begin{align*}
    \Pi_{I}\coloneqq\bigotimes_{i\in R}\Pi_{D_i;I}\colon N_{\mathscr{X}}D_I\to\mathcal{O}_{\mathscr{X}}(\sum D_i)\mid_{V_i} \,.
\end{align*}

Since our divisor $\mathbf{D}$ is the special fiber of a semistable degeneration $\pi\colon\mathscr{X}\to\mathbb{D}$, it follows that the line bundle~\eqref{semistablelinebundletrivial} is trivial, with a trivialization denoted by
\begin{align}\label{semistabletrivialization}
    (\pi_{ \mathcal{O}_{\mathscr{X}}(\sum_{i\in S} D_i)},\tau)\colon  \mathcal{O}_{\mathscr{X}}(\sum_{i\in R} D_i)\to \mathscr{X}\times\mathbb{C} \,.
\end{align}
To simplify notation, we will just refer to this trivialization by its $\mathbb{C}$-component $\tau$. Note that the composition $\tau\circ\sigma$ is germ-equivalent to $\pi_{\mathscr{X}}\colon\mathscr{X}\to\mathbb{D}$ near the central fiber $\mathbf{D}$. The results of~\cite{McLeanLogCanonical} show that we can homotope this trivialization to one compatible with the regularization $\mathfrak{R}_{\mathbf{D}}$ obtained from Theorem~\ref{MTZthm}, from which we will obtain our regularized symplectic fibration.

For any $i\in R$, let 
\begin{align}\label{smallradiusregularizednbhds}
    U_{i,r}\subset U_i
\end{align}
be the disk bundle of radius $r>0$ in the normal bundle $U_i\to D_i$, with respect to the metric on $U_i$ induced by the Hermitian structure. Slightly abusing notation, we let
\begin{align}
    \rho_i\colon U_i\to\mathbb{R}
\end{align}
denote the function whose value is the norm-squared distance to the zero section. Also set
\begin{align}
    U_{I,r}\coloneqq\bigcap_{i\in I}U_{i,r} \,.
\end{align}
For a constant $C>0$, choose a smooth function $a\colon[0,\infty)\to[0,\infty)$ such that
\begin{itemize}
    \item[(i)] $a'(x)>0$ if $x\in [0, 3C/4)$;
    \item[(ii)] $a(x) = x$ if $x\leq C/4$, and;
    \item[(iii)] $a(x) = 1$ if $x\geq 3C/4$.
\end{itemize}
Let $U\subset\mathscr{X}$ be an open set. We say that a trivialization
\begin{align*}
        (\pi_{ \mathcal{O}_{\mathscr{X}}(\sum_{i\in S}D_i)},\tau_{reg})\colon  \mathcal{O}_{\mathscr{X}}(\sum_{i\in S}D_i)\to \mathscr{X}\times\mathbb{C} \,.
\end{align*}
of the line bundle determined by a regularization $\mathfrak{R}_{\mathbf{D}}$ of $\mathbf{D}$ is (radius $C$) compatible with the regularization along $U\cap D_I$ if
\begin{align}\label{compatiblewithregularizationdefinitingequation}
    \tau_{reg}(\sigma(x)) = \tau_{reg}(\Pi_{I}(\psi^{-1}_I(x)))\cdot\prod_{i\in I}\left(\frac{\sqrt{a(\rho_i(x))}}{\sqrt{\rho_i(x)}}\right)
\end{align}
for all $x\in U_{I,C}\cap(\psi_I(U_I)\mid_{(V_j\cap U)\setminus\cup_{i\in R\setminus I}U_{i,3C/4}})$, and if the norm of the map 
\begin{align*}
    \tau_{reg}\colon\bigotimes_{i\in I}N_{\mathscr{X}}D_i\mid_{x}\to\mathbb{C}
\end{align*}
restricted to the fiber of ${\bigotimes_{i\in I}N_{\mathscr{X}}D_i}$ over $x$ (using the identification~\eqref{structuresheafnormalDIidentity}) is $1$ for all $x\in D_I\cap U\setminus\bigcup_{i\in R\setminus I}U_{i, 3C/4}$. If the trivialization is (radius $C$) compatible with the regularization along $U\cap D_I$ for each $I\subset R$, we say that it is (radius $C$) \textit{compatible with the regularization} $\mathfrak{R}_{\mathbf{D}}$.
\begin{remark}
    Compatibility with $\mathfrak{R}_{\mathbf{D}}$ says that the canonical section $\sigma$ looks approximately like
    \begin{align*}
        (x_1,\ldots,x_{n+1})\mapsto \prod_{i = 1}^k \frac{a(|x_i|)}{|x_i|}x_i
    \end{align*}
    in a local chart where $D_I$ is cut out by $x_1,\ldots,x_k$ (cf.~\eqref{semistablelocaldescription}.)
\end{remark}
\begin{lemma}[{\cite[Lemma 5.10]{McLeanLogCanonical}}]\label{homotopyoftrivializations}
    Let $U\subset\mathscr{X}$ be a relatively compact open set. Then the trivialization $\tau$ of $\mathcal{O}_{\mathscr{X}}(\sum_i D_i)$ determined by the semistable degeneration $\mathscr{X}\to\mathbb{D}$ is homotopic to a trivialization $\tau_{reg}$ compatible with a given regularization as in Theorem~\ref{MTZthm} along $U$ through trivializations of $\mathcal{O}_{\mathscr{X}}(\sum_i D_i)$.
\end{lemma}
All of the fibers $\pi_{reg}\coloneqq\tau_{reg}\circ\sigma\colon\mathscr{X}\to\mathbb{C}$ over $t\in\mathbb{C}$ with $|t|\neq0$ sufficiently small will be symplectic manifolds by the proof of Lemma 5.32 in~\cite{McLeanLogCanonical}. By taking $\epsilon$ to be sufficiently small, we can also assume that $a\colon\mathbb{R}\to\mathbb{R}$ is the identity. Hence we define the \textit{regularized symplectic fibration}
\begin{align}
    \pi_{reg}\colon(\mathscr{X}_{\mathbb{D}_{\epsilon}},\Omega_1)\to\mathbb{D}_{\epsilon} \,.
\end{align}
The discussion on~\cite[p. 1011--1012]{McLeanLogCanonical} shows that $\pi_{reg}$ is a symplectic fiber bundle when restricted to $\mathbb{D}_{\epsilon}\setminus\lbrace 0\rbrace$, by the compatibility of the trivialization $\pi_{reg}$ with the $\Omega$-regularization $\mathfrak{R}$. This is a nearly regular symplectic fibration~\cite[Definition 2.6]{TZsmooth} whose only singular fiber is $\pi_{reg}^{-1}(0) = \pi^{-1}(0) = \mathbf{D}$.

In particular, letting $X_t\coloneqq\pi_{reg}^{-1}(t)$, the monodromy $\phi_{reg}\colon X_t\to X_t$, about a circle of radius $\epsilon' = |t| <\epsilon$ gives a well-defined symplectomorphism of a smooth fiber. Moreover, it follows from Moser's argument that there is an isomorphism of symplectic mapping tori
\begin{align*}
    T_{\phi}\cong T_{\phi_{reg}}
\end{align*}
determined by the homotopy of trivializations in Lemma~\ref{homotopyoftrivializations} (see the proofs of Lemmata 5.34 and 5.36 in~\cite{McLeanLogCanonical}.) This also leads to an identification
\begin{align}\label{identificationofregularpartwithmappingtorus}
    (\mathscr{X}_{\mathbb{D}_{\epsilon}},\Omega_1)\cong(E_{\phi_{reg}},\Omega_{\phi_{reg}})\mid_{(0,\log(\epsilon))\times S^1},
\end{align}
which we fix once and for all.

The next lemma lets us compute the fixed points of the symplectic monodromy $\phi_{reg}$, and is essentially the statement of~\cite[Lemma 2.4]{GP2} adapted to our setting. Consider the Ehresmann connection on $\mathscr{X}_{\mathbb{D}_{\epsilon}}$ defined using the regularized symplectic form $\Omega_1$, with respect to which we define the horizontal lift in $T\mathscr{X}_{\mathbb{D}_{\epsilon}}$ of the (clockwise) rotational vector field $-it\partial_t\in T\mathbb{D}^{\circ}_{\epsilon}$ on the punctured disk. We can extend this vector field, by zero, over the central fiber of $\mathscr{X}_{\mathbb{D}_{\epsilon}}$ to obtain a \textit{Hamiltonian} vector field $\xi^{reg;\epsilon}\in T\mathscr{X}_{\mathbb{D}_{\epsilon}}$, since $\mathscr{X}_{\mathbb{D}_{\epsilon}}$ is simply connected. Let
\begin{align}
    h_{\mathscr{X}}\colon \mathscr{X}_{\mathbb{D}_{\epsilon}}\to\mathbb{R}
\end{align}
denote the Hamiltonian associated to this vector field. Identifying $\phi_{reg}$ with the flow of a Hamiltonian vector field on $U_{\epsilon}$ will allow us to apply results in the (absolute) log Calabi--Yau setting, e.g. in~\cite{GP1, GP2, Pomerleano}, more easily.

For each $i\in R$, let $\pi_i\colon U_i\to D_i$ denote the bundle projection, and let $\pi_I\colon U_I\to D_I$ denote the iterated projection
\begin{align}
\pi_I\coloneqq\pi_{i_1}\circ\cdots\circ\pi_{i_{|I|}} \,.
\end{align}
This is a symplectic fibration with structure group $U(1)^{|I|}$ and whose fibers are (open subsets of) a product of standard disks.

\begin{lemma}\label{orbitsinstandardneighborhoods}
The flow of $\xi^{reg;\epsilon}$ restricts to rotation in the fibers of $\pi_I\colon U_I\to D_I$, for each stratum $D_I$. 
\end{lemma}
\begin{proof}
Fix a fiber $X_t = \pi_{reg}^{-1}(t)$ in $U_{\epsilon}$. At any point $x\in X_t$ contained in a regularized tubular neighborhood $U_i$, notice that the tangent space $T_x U_{\epsilon}$ can be written as the direct sum
\[ T_xU_{\epsilon} = T_x X_t\oplus\mathbb{C}\xi^{reg;\epsilon}_x\oplus\mathbb{C}\partial_{\rho_i}\]
where $\mathbb{C}\partial_{\rho_i}$ is the tangent vector in the fiber radial direction. The $1$-form $\Omega_1(\xi^{reg;\epsilon},\cdot)$ vanishes on the first two components of this decomposition. The vanishing of this form on $T_x X_t\oplus\mathbb{C}\xi_x$ implies that $h_{U_{\epsilon}}\mid_{U_I}\to\mathbb{R}$ is a function of the radial coordinates $(\rho_{i_1},\ldots,\rho_{i_{|I|}})$.

An elementary computation using the local forms~\eqref{standardneighborhoodsymplecticform} shows that the symplectic orthogonal to any fiber of $\pi_i\colon U_i\to D_i$ is contained a level set of the radial function $\rho_i$~\cite[Lemma 2.4(1)]{GP2}. The local expression for $\Omega_1(\xi^{reg;\epsilon},\cdot)$ above shows that it vanishes on any vector tangent to an intersection of level sets of $\rho_{i_1},\ldots,\rho_{i_{|I|}}$. Thus $\xi^{reg;\epsilon}$ must be orthogonal to the symplectic orthogonal of the tangent space of the fibers, meaning that $\xi^{reg;\epsilon}$ is tangent to the fibers of $\pi_i$. This implies that the flow of $\xi^{reg;\epsilon}$ rotates the disk factors of the fibers of the iterated fibrations $\pi_I$.
\end{proof}

The time-$1$ flow of $\xi^{reg;\epsilon}$ restricted to a general fiber $X_t$ of $U_{\epsilon}$ gives $\phi_{reg}$, which has the same fixed point Floer cohomology as the monodromy symplectomorphism $\phi$ of~\eqref{symplecticmonodromyrep}. The closed orbits of the flow can be thought of as multiple covers of horizontal sections of $T_{\phi_{reg}}$, as explained in \S{\ref{multiplecoverfixedpointfloer}}, and also as multiply covered circles in the fibers of $\pi_I\colon U_I\to D_I$ by Lemma~\ref{orbitsinstandardneighborhoods}.

Since our regularization $\mathfrak{R}_{\mathbf{D}}$ of $\mathbf{D}$ came from a regularization $\mathfrak{R}$ of $\mathbf{D}\cup\mathbf{E}$, we obtain the following.
\begin{lemma}\label{divisorconstruction-1}
    The flow of $\xi^{reg;\epsilon}$ preserves each component $E_q$ of $\mathbf{E}$.
\end{lemma}
\begin{proof}
    Recall that $\mathbb{D}_{\epsilon}$ was chosen so that $\mathscr{X}_{\mathbb{D}_{\epsilon}}$ is contained in the union of regularized neighborhoods $U_I$ (in $\mathfrak{R}$) of the strata $D_I$ of the central fiber. Thus for each $K\subset Q$ it follows that the regularized neighborhoods $U_q$ of $E_q$ are also contained in union of these neighborhoods. This, combined with compatibility of $\tau_{reg}$ with $\mathfrak{R}_{\mathbb{D}}$ tells us that
    \begin{align*}
        (T_x E_q)^{\perp}\subset T_x X_t
    \end{align*}
    for all $x\in E_q\cap X_t$. In more detail, this follows because, with respect to $\Omega_1$, the fibers of the horizontal lift of the tangent bundle to $\mathbb{D}_{\epsilon}$ are contained in the fibers of the vertical bundles to the projections $\pi_I\colon U_I\to D_I$. Thus the inclusion above follows by the orthogonality of the divisors $E_q$ to the strata $D_I$ with respect to $\Omega_1$. This implies that $\xi^{reg;\epsilon}_x$ is tangent to $E_q$ at $x\in E_q$, so its flow fixes $E_q$.
\end{proof}

\subsection{Isolating neighborhoods and Hamiltonian perturbations}
We will construct an explicit Hamiltonian perturbation $H_{\delta}\in C^{\infty}(T_{\phi_{reg}},\mathbb{R})$ following~\cite[\S{4.1}]{GP2} which makes the perturbed, as in~\ref{hamiltonianperturbationsmappingtori}, pair $(T_{\phi_{reg}},\Theta_{\phi_{reg}})$ nondegenerate. We emphasize that $\Theta_{\phi_{reg}}$ arises from the \textit{regularized} symplectic form $\Omega_1$. This perturbation is supported on a disjoint union of \textit{isolating neighborhoods} of Morse--Bott families of orbits of $h_{\mathscr{X}}$, which will also appear in the proof of Theorem~\ref{stanleyreisnerisom}. By choosing this perturbation in a controlled way, the resulting flow trajectories will still be contained in the fibers of $\pi_I\colon U_I\to D_I$. This uses a \textit{spinning} construction in Morse--Bott theory~\cite{KwonVanKoert}, we which will recall in some detail.

Any closed orbit contained in the tubular neighborhood $U_I$ can be written as a time-$1$ orbit of the vector field
\begin{align*}
    \sum_{i\in R}-2\pi v_i\partial_{\varphi_i}
\end{align*}
where $\mathbf{v}\in\mathbb{Z}^R_{\geq0}$ is an integer vector which has nonzero $i$th component $v_i$ if and only if $i\in I$. We restrict our attention to orbits corresponding to the fixed points of $\phi_{reg}\colon X_t\to X_t$, i.e. orbits in $U_I$ which project onto the circle of radius $\epsilon'$ centered at $0\in\mathbb{D}_{\epsilon}$.

Let
\begin{align}
    \mathcal{F}_{\mathbf{v}}
\end{align}
denote the corresponding orbit sets, and for any closed orbit $x_0$ of $\xi = \xi^{reg;\epsilon}$, define its \textit{multiplicity vector} to be the unique $\mathbf{v}\in\mathbb{Z}^R_{\geq0}$ such that $x_0$ is contained in $\mathcal{F}_{\mathbf{v}}$. These orbit sets are contained in the set of points in the level sets of $\rho_i$ in $U_I$, for some $i\in I$ corresponding to some constant denoted
\[\rho_{i,\mathbf{v}}\in\mathbb{R} \,.\]
They are also manifolds with corners, as the argument in Step 2 of the proof of Theorem 5.16 of~\cite{McLeanMinimalDiscrepancy} shows. The corner strata occur in some level set of $\rho_i$ corresponding to some fixed value denoted
\[\rho_{\mathbf{v},i}^c\in\mathbb{R}\]
where $i\not\in I$ (cf.~\cite[\S{4.1}]{GP2}.)

Choose isolating neighborhoods
\begin{align}\label{isolatingneighborhoods}
    U_{\mathbf{v}}
\end{align}
for $\mathcal{F}_{\mathbf{v}}$, by which we mean open neighborhoods of $\mathcal{F}_{\mathbf{v}}$ which do not contain any other orbit set. We can choose these neighborhoods so that $U_{\mathbf{v}}\cap U_{\mathbf{w}} = \emptyset$ whenever $\mathbf{v}\neq\mathbf{w}$~\cite{GP2}. For all $i\in R$, choose shrunken regularized neighborhoods $U_{i,r}$ as in~\eqref{smallradiusregularizednbhds}. Let $D_I'\coloneqq D_I\setminus\bigcup_{i\not\in I}U_{i,r}$, and let $S_I'$ denote the induced $T^{|I|}$-bundle over $D_I'$.\footnote{The neighborhoods $U_{i,r}$ are chosen more carefully in~\cite{GP2}, and are used to define the isolating neighborhoods of the orbit sets. The discussion there takes place locally near strata of $\mathbf{D}$, and applies verbatim to our situation, so we will not repeat those details here.} Fix a Morse function $\hat{h}_I\colon S_I'\to\mathbb{R}$ which is a function of the distance functions $\rho_i$ near the corners and which points outward along the boundary.

The flow of $\xi$ generates an $S^1$-action on $\mathcal{F}_{\mathbf{v}}$ which extends to the isolating neighborhood $U_{\mathbf{v}}$, which is locally Hamiltonian, as noted in the proof of Lemma~\ref{orbitsinstandardneighborhoods}. In such an isolating neighborhood, the inverse $S^1$-action has the associated Hamiltonian $K_{\mathbf{v}}\colon U_{\mathbf{v}}\to\mathbb{R}$ given by $K_{\mathbf{v}} = \sum\pi v_i\rho_i$. Let $\Delta_t$ denote the time-$t$ flow of $X_{K_{\mathbf{v}}}$. Define a time-dependent function $h_I\coloneqq\hat{h}_I \circ\Delta_t(x)$, and let $h_{\mathbf{v}}$ denote its pullback to $U_{\mathbf{v}}$ under the projection map. Lastly, choose cutoff functions $\chi_{\mathbf{v}}$ which are identically $1$ on slightly smaller isolating neighborhoods $U_{\mathbf{v}}'\subset U_{\mathbf{v}}$, and vanish outside of $U_{\mathbf{v}}$.

The desired Hamiltonian perturbation\footnote{Unlike~\cite{GP1, GP2, Pomerleano}, we do not perturb our Hamiltonian near $\mathbf{D}$.} is
\begin{align}\label{explicithamiltonianperturbation}
    H_{\delta}\coloneqq\sum_{\mathbf{v}}\delta_{\mathbf{v}}\chi_{\mathbf{v}}h_{\mathbf{v}}
\end{align}
for sufficiently small constants $\delta_{\mathbf{v}}>0$. By~\cite[\S{4.1}]{GP1}, the data $\chi_{\mathbf{D}}$ and $\delta_{\mathbf{D}}$ can be chosen so that the flow of
\begin{align}\label{perturbedHamiltonianondegeneration}
    H_{\mathscr{X}}\coloneqq h_{\mathscr{X}}+H_{\delta}
\end{align}
fixes $\mathbf{D}$.
\begin{lemma}[cf. {\cite[Lemma 4.2]{GP2}}]\label{hamiltonianperturbationlemma}
Let $S^1_{\epsilon'}\subset\mathbb{D}^{\circ}_{\epsilon}$ be the circle of radius $\epsilon'<\epsilon$, and identify $U_{\epsilon}\mid_{S^1_{\epsilon'}}$ with the mapping torus $T_{\phi_{reg}}$. Consider the restriction $H^{\epsilon'}_{\delta}$ of $H_{\delta}$ to $U_{\epsilon}\mid_{S^1_{\epsilon'}}$. For sufficiently small $\delta_{\mathbf{v}}$, all horizontal (multi-)sections of the perturbed symplectic fiber bundle
\[(T_{\phi_{reg}},\Theta_{\phi_{reg},H^{\epsilon'}_{\delta}})\cong(T_{\phi_{reg}\circ\phi^1_{H^{\epsilon'}_{\delta}}},\Theta_{\phi_{reg}\circ\phi^1_{H^{\epsilon'}_{\delta}}})\]
arise from the manifolds $\mathcal{F}_{\mathbf{v}}$ as critical points of $h_{\mathbf{v}}$.
\end{lemma}
\begin{proof}
By compactness, these orbits must be contained in $U_{\mathbf{v}}'$ when $\delta_{\mathbf{v}}$ is chosen sufficiently small. There are no fixed points near the corner strata, where $\rho_i = \rho_{\mathbf{v},i}^c$, since $\hat{h}_I$ points outward along the strata. With the above understood, the lemma follows from standard facts in Morse--Bott theory.
\end{proof}

An important technical point for us is that by choosing the Hamiltonian perturbations of the form~\eqref{explicithamiltonianperturbation} carefully, we can assume that their flows preserve $\mathbf{E}$ and that their low degree orbits lie away from $\mathbf{E}$.
\begin{lemma}\label{divisorconstruction0}
    There exist Hamiltonian perturbations of the form~\eqref{explicithamiltonianperturbation} above such that
    \begin{itemize}
        \item[(a)] the Hamiltonian vector field associated to $H_{\mathscr{X}}$ preserves $E_q$ for all $q\in Q$, and;
        \item[(b)] all Hamiltonian orbits $\nu$ contained in any $E_q\cap\mathcal{X}^{\circ}$ have degree $\mu(\nu)\geq2$.
    \end{itemize}
\end{lemma}
\begin{proof}
    This follows more or less immediately from the proof of Lemma 4.5 in~\cite{Pomerleano}, and amounts to choosing $h_{\mathbf{v}}$ and $\chi_{\mathbf{v}}$, so that their Hamiltonian flows preserve each $E_q$. We have seen that $h_{\mathscr{X}}$ already has this property by Lemma~\ref{divisorconstruction-1}. To see that this can be done so that (b) holds, we note that the relevant perturbations are carried out in $\mathcal{F}_{\mathbf{v}}\cap\mathbf{V}$, i.e. near the orbit sets, so the argument in \textit{loc. cit.} applies with no changes.
\end{proof}

\subsection{Completing the regularized symplectic fibration}
To make the geometric setup discussed so far suitable for counting pseudoholomorphic sections (cf. \S{\ref{fixedpointfloersection}}), we introduce families
\begin{align}\label{familiesoverC}
    \mathscr{X}_{\epsilon'}\to\mathbb{C}
\end{align}
whose fiber over the origin is $\mathbf{D}$ and whose restriction to $\mathbb{C}^*$ is a symplectic fiber bundle isomorphic to 
\begin{align}
    \mathscr{X}_{\epsilon'}\mid_{\mathbb{C}^*}\cong E_{\phi_{reg}\circ\phi^1_{H^{\epsilon'}_{\delta}}} 
\end{align}
where $\phi_{reg}$ denoes the monodromy about a circle of radius $\epsilon>0$.

As in Lemma~\ref{hamiltonianperturbationlemma}, we consider the restriction $\phi_{reg}$ to a circle of radius $\epsilon'$ in the disk $\mathbb{D}^{\circ}_{\epsilon}$, which is isomorphic to the mapping torus
\begin{align}\label{fullymodifiedmappingtorus}
    (T_{\phi_{reg}\circ\phi^1_{H^{\epsilon'}_{\delta}}},\Theta_{\phi_{reg}\circ\phi^1_{H^{\epsilon'}_{\delta}}}) \,.
\end{align}
Observe that divisors $E_q$, for $q\in Q$, restrict to horizontal divisors in this mapping torus, and so these give rise to stabilizing divisors in $E_{\phi_{reg}\circ\phi^1_{H^{\epsilon'}_{\delta}}}$ by taking their products with the $\mathbb{R}$-factor. Similar remarks apply to the regularized neighborhoods $U_I\cap U_{\epsilon}$ and to the isolating neighborhoods $U_{\mathbf{v}}$~\eqref{isolatingneighborhoods}.

\begin{definition}\label{familiesoverCdefinition}
    We form the families $\mathscr{X}_{\epsilon'}\to\mathbb{C}$ as in~\eqref{familiesoverC} by attaching a positive end
    \[T_{\phi_{reg}\circ\phi^1_{H^{\epsilon'}_{\delta}}}\times[\log(\epsilon'),\infty) \]
    of~\eqref{fullymodifiedmappingtorus} to $\overline{U}_{\epsilon'}$, the family obtained by restricting $U_{\epsilon}\to\mathbb{D}_{\epsilon}$ to the closed disk of radius $\epsilon'<\epsilon$ along their common boundaries, via~\eqref{identificationofregularpartwithmappingtorus}. We refer to this cylindrical end as the \textit{shell} in $\mathscr{X}_{\epsilon'}$.
\end{definition}
It is now obvious that the stabilizing divisors $\mathbf{E}$, the regularized neighborhoods $U_{I}$, and the isolating neighborhoods $U_{\mathbf{v}}$ all extend naturally to $\mathscr{X}_{\epsilon}$ by stabilization with the $\mathbb{R}$-factor. Moreover, the divisors $E_q$ and the neighborhoods $U_I$ all project submersively onto $\mathbb{C}$.

We will exclusively work with the families of Definition~\ref{familiesoverCdefinition} for the rest of this section and all of the next section, and thus we impose the following working assumption.
\begin{assumption}\label{completedfamilyassumption}
    Going forward, $\mathscr{X}$ will always mean one of the families $\mathscr{X}_{\epsilon'}$, unless explicitly stated otherwise. Similarly, $\phi$ will refer to $\phi_{reg}$ composed with a Hamiltonian perturbation of the form~\eqref{explicithamiltonianperturbation}.
\end{assumption}

\subsection{Enhanced fixed point Floer cohomology}\label{enhancedfixedpointfloersection}  Having chosen an extension $\mathscr{X}$ of $\mathcal{X}^{\circ}$ over $\mathbb{D}$, a Hamiltonian perturbation~\eqref{explicithamiltonianperturbation} for the horizontal lift of the rotational vector field, and a system of isolating neighborhoods~\eqref{isolatingneighborhoods} for the orbits of this perturbed flow, we can construct an an enhanced version of fixed point Floer cohomology which is an invariant of the pair $(\mathscr{X},\mathbf{D})$ (cf.~\cite[\S{4.3.1}]{Pomerleano}.) This version of fixed point Floer cohomology will be linear over the ring $\Lambda_{\mathscr{X}}$ which is defined as follows (cf.~\cite[\S{4}]{HoferSalamonNovikov}).

Consider the group
\begin{align*}
    G = \Bbbk[\pi_2(\mathscr{X})/\pi_2(\mathscr{X})_{tors}]
\end{align*}
where $\pi_2(\mathscr{X};\mathbb{Z})_{tors}$ is the torsion subgroup of $\pi_2(\mathscr{X})$. Let $\Lambda_{\mathscr{X}}$ be the $\Bbbk$-module consisting of all functions
\begin{align*}
    G &\to \Bbbk \\
    A &\mapsto \lambda_A
\end{align*}
such that the set $\lbrace A\in G\mid\lambda_A\neq 0\text{ and }\Omega(A)<c\rbrace$ is finite for every constant $c\in\mathbb{R}$. The product on $\Lambda_{\mathscr{X}}$ is the convolution product
\begin{align*}
    (\lambda\ast\lambda')_A = \sum_{B\in G}\lambda_B\lambda'_{B^{-1}A} \,.
\end{align*}

Observe that any orbit $\nu\in\mathcal{H}(T_{\phi^d},\Theta_{\phi^d})$ is contractible in $\mathscr{X}$, and therefore admits capping disks (see e.g.~\cite[\S{5}]{HoferSalamonNovikov}.) We say that two capping disks $u_1$ and $u_2$ for $\nu$ are \textit{equivalent} if
\begin{align}\label{cappingdiskequivalence}
u_1\#\overline{u_2} \in \pi_2(\mathscr{X})_{tors}
\end{align}
where $\overline{u_2}$ denotes $u_2$ with its reversed orientation and $u_1\#\overline{u_2}$ is the (topological) $2$-sphere obtained by gluing the disks along their common boundary. We say that a \textit{capped orbit} $\tilde{\nu}$ is a pair of an orbit $\nu\in\mathcal{H}(T_{\phi^d},\Theta_{\phi^d})$ and an equivalence class of capping disks $[u_{\nu}]$ for $\nu$, and let $[\nu,u_{\nu}]$ denote an equivalence class of capped orbits.

We define the enhanced Floer cochain groups
\begin{align*}
    CF^*_{\mathscr{X}}(X_t,\phi^d)\coloneqq\bigoplus_{\tilde{\nu}}o^{\Bbbk}_{\tilde{\nu}}
\end{align*}
where the sum is over all capped orbits, and here $o^{\Bbbk}_{\tilde{\nu}}$ denotes a copy of the orientation line $o_{\nu}$ of $\nu$ over $\Bbbk$. This naturally has the structure of a $\Lambda_{\mathscr{X}}$-module determined by taking connected sums with spheres in $\mathscr{X}$. By Lemma~\ref{orbitsinstandardneighborhoods}, any horizontal section $\nu$ has a canonical capping disk $F_{\nu}$, corresponding to a disk in a fiber of $\pi_I\colon U_I\to D_I$, which allows us to identify this cochain group with the free $\Lambda_{\mathscr{X}}$-module
\begin{align}\label{freemoduleenhancedfloer}
    CF^*_{\mathscr{X}}(X_t,\phi^d) = \bigoplus_{\nu\in\mathcal{H}(T_{\phi^d},\Theta_{\phi^d})} o^{\Lambda_{\mathscr{X}}}_{[\nu,F_{\nu}]}
\end{align}
where this time $o^{\Lambda_{\mathscr{X}}}_{[\nu,F_{\nu}]}$ is a copy of the orientation line associated to $\nu$ over $\Lambda_{\mathscr{X}}$.

The Floer differential in this context is defined to be
\begin{align}\label{enhancedfixedpoinfloerdifferential}
    d_{\mathscr{X}}\coloneqq\sum_{\tilde{\nu}_{-}}\sum_{\substack{B\in\pi_2(\nu_{-},\nu_{+}) \\ \ind(B) = 1}}\sum_{u\in\mathcal{M}(E_{\phi},J;B)}\sigma_{u}
\end{align}
where the first sum is over all equivalence class of capped orbits $\tilde{\nu}_{-} = [\nu_{-},u_{\nu_{-}}]$, and the isomorphism
\begin{align*}
    \sigma_u\colon o^{\Bbbk}_{[{\nu_{-}},u_{\nu_{-}}]}\to o^{\Bbbk}_{[\nu_{+},u_{\nu_{+}}]}
\end{align*}
is a copy of the isomorphism of orienation lines $o_{\nu_{-}}\to o_{\nu_{+}}$, where the capping disk $u_{\nu_{+}} = u_{\nu_{-}}\# B$ is in the equivalence class determined by the boundary sum of $u_{\nu_{-}}$ with $B$.
\begin{definition}
    We call $(CF^*_{\mathscr{X}}(X_t,\phi^d),d_{\mathscr{X}})$ the \textit{enhanced fixed point Floer cochain complex}, and its homology the \textit{enhanced fixed point Floer cohomology.}
\end{definition}
The pair of pants product~\eqref{pantsproducthchainlevel} admits a natural extension to the enhanced fixed point Floer cochain complex. Note that since the symplectic form $\Omega$ on $\mathscr{X}$ represents an integral cohomology class, there is a natural homomorphism
\begin{align*}
    \Lambda_{\mathscr{X}}\to\Bbbk(\!(T)\!)
\end{align*}
so we can change the coefficients of the enhanced Floer cochain complex to $\Bbbk(\!(T)\!)$. Moreover, using~\eqref{freemoduleenhancedfloer}, one easily sees that the differential and product on the fixed point Floer cochain complex can be defined over the smaller ring $\Bbbk[\![T]\!]$, and we let
\begin{align}\label{powerseriesfixedpointfloercochains}
    CF^*_{\Bbbk[\![T]\!]}(X_t,\phi^d;\Bbbk[\![T]\!])
\end{align}
denote the version of the fixed point Floer cochain complex defined over $\Bbbk[\![T]\!]$. Of course, the fixed point Floer \textit{cohomology} is only a Hamiltonian isotopy over the universal Novikov field $\Lambda$. We can change the coefficients of~\eqref{powerseriesfixedpointfloercochains} to $\Bbbk$, using the natural homomorphism $\Bbbk[\![T]\!]\to\Bbbk$ and by taking the leading term of the differential.
\begin{definition} Denote by
\begin{align}\label{lowenergyfixedpointfloercochains}
    CF^*_{low}(X_t,\phi^d;\Bbbk)\coloneqq  CF^*_{\Bbbk[\![T]\!]}(X_t,\phi^d;\Bbbk[\![T]\!])\otimes_{\Bbbk[\![T]\!]}\Bbbk
\end{align}
the enhanced Floer cochain complex over $\Bbbk$, and define the \textit{low energy fixed point Floer cohomology}
\begin{align}~\label{lowenergyfixedpointfloercohomology}
    HF^*_{low}(X_t,\phi^d;\Bbbk)
\end{align}
to be the cohomology of this complex.
\end{definition}
The differential on~\eqref{lowenergyfixedpointfloercochains} is just the $T^0$-term of the differential on~\eqref{powerseriesfixedpointfloercochains}. We can describe this differential in geometric terms, modifying~\eqref{enhancedfixedpoinfloerdifferential}, to be differential that only counts sections in homology classes $B\in\pi_2(\nu_{-},\nu_{+})$ with the property that
\begin{align}\label{lowenergycylinderclassdef}
    [-F_{\nu_{-}}]\# B = [-F_{\nu_{+}}]
\end{align}
i.e. $B$ respects canonical capping disks (with reversed orientations.) There is a product on~\eqref{lowenergyfixedpointfloercohomology} which can be characterized as a count of pairs of pants in similarly-defined homotopy classes.

Changing the coefficients of the enhanced fixed point Floer cohomology groups to $\Lambda$ recovers the fixed point Floer cohomology groups of Definition~\ref{fixedpointfloercohomologydefinition}. There is a natural chain isomorphism over $\Lambda$ given by
\begin{align*}
    CF^*_{\mathscr{X}}(X_t,\phi^d)\otimes_{\Lambda_{\mathscr{X}}}\Lambda &\to CF^*(X_t,\phi^d)
\end{align*}
coming from maps of orienation lines
\begin{align*}
    o_{[\nu,u_{\nu}]}^{\Bbbk}\to o_{\nu}
\end{align*}
given by multiplication by $T^{\omega(u_{\nu})}$, where $\omega(u_{\nu})$ is the symplectic area of the capping disk.

Consider the cochain complex~\eqref{powerseriesfixedpointfloercochains} over $\Bbbk[\![T]\!]$ with the $(T)$-adic filtration
\[\lbrace F^q CF^p_{\mathscr{X}}(X_t,\phi^d;\Bbbk[\![T]\!])\rbrace_{q\in\mathbb{Z}}\,.\]
The results of~\cite[\S{6.3}]{FOOOI} show that there is a spectral sequence $\lbrace E^{*,*}_r\rbrace$ such that
\begin{itemize}
    \item[(i)] The $E_2$-page of the spectral sequence is determined by the low-energy fixed point Floer cohomology via
    \begin{align*}
        E^{p,q}_2 = HF^p_{low}(X_t,\phi^d;\Bbbk)\otimes(T^{q}\Bbbk[\![T]\!]/T^{q+1}\Bbbk[\![T]\!])
    \end{align*}

    \item[(ii)] There is a filtration $\mathfrak{F}^*HF^*_{\mathscr{X}}(X_t,\phi^d;\Bbbk[\![T]\!])$ on the fixed point Floer cohomology over $\Bbbk[\![T]\!]$ such that
    \begin{align*}
        E^{p,q}_{\infty} \cong \frac{\mathfrak{F}^{q}HF^p(X_t,\phi^d;\Bbbk[\![T]\!])}{\mathfrak{F}^{q+1}HF^p(X_t,\phi^d;\Bbbk[\![T]\!])} \,.
    \end{align*}
\end{itemize}

\section{The low-energy log PSS isomorphism}\label{logPSSmaps-section}
This section is devoted to the proof of Theorem~\ref{stanleyreisnerisom}, which largely proceeds by adapting the work of Ganatra and Pomerleano in~\cite{GP1, GP2, Pomerleano}.

\subsection{Log cohomology and Stanley--Reisner rings}
The natural domain of the log PSS map, the log cohomology ring of~\cite{GP1, GP2}, is slightly larger than the Stanley--Reisner ring $\SR(\mathcal{X}^{\circ})$. Since we are at present only interested in the Stanley--Reisner ring, our presentation of the log PSS moduli spaces simplify compared to that in those sources (cf.~\cite{Pomerleano}.) We begin by describing the Stanley--Reisner rings from this perspective.

For any subset $I\subset I'\subset R$, there is an inclusion $D_{I'}\subset D_I$, which induces a restriction map on cohomology, which we denote by
\begin{align*}
    \res_{I;I'}\colon H^0(D_I)\to H^0(D_{I'}) \,.
\end{align*}
Suppose that we have fixed formal variables $\theta_i$, for $i\in R$, and for any (multiplicity) vector $\mathbf{v}\in\mathbb{Z}_{\geq0}^R$, we set
\begin{align*}
    \theta^{\mathbf{v}} = \prod_{i\in R}\theta_i^{v_i} \,.
\end{align*}
For any $I\subset R$, define the multiplicity vector $\mathbf{v}_I$ to have components $v_i = 1$ if $i\in I$ and $v_i = 0$ otherwise. Define the graded Stanley--Reisner rings to have underlying $\Bbbk$-vector spaces
\begin{align}\label{gradedsrring}
    \SR^*(\mathscr{X})\coloneqq\bigoplus_{I\subset R}\theta^{\mathbf{v}_I}H^0(D_I)[\theta_i\mid i\in I].
\end{align}
We remind the reader that we can arrange that all strata are connected by blowing up $\mathscr{X}$, so each of these summands are one-dimensional. The multiplication on $\SR^*(\mathscr{X})$ is induced by the cup product, in the sense that if $I' = I_1\cup I_2$ and $\alpha_i\in H^0(D_{I_i})$ for $i = 1,2$, then 
\begin{align}
    \alpha_1\theta^{\mathbf{v}_1}\star_{\SR}\alpha_2\theta^{\mathbf{v}_2}\coloneqq(\res_{I_1;I'}\alpha_1\cup \res_{I_2;I'}\alpha_2)\theta^{\mathbf{v}_1+\mathbf{v}_2} \,.
\end{align}
To define the grading, we recall the expression for $K_{\mathscr{X}}+\mathbf{D}$ from Proposition~\ref{gsresolutionmodel}, and set
\begin{align}\label{stanleyreisnergradingconvention}
\deg(\alpha\theta^{\mathbf{v}})\coloneqq 2\sum_{i\in R}a_iv_i \,.
\end{align}
It follows that the degree $0$ part of $\SR^*(\mathscr{X})$ is just the Stanley--Reisner ring $\SR(\mathcal{X}^{\circ})$ discussed in \S{\ref{mainsection}}.\footnote{Our grading convention is consistent with the grading on log cohomology given in~\cite[(3.15)]{GP2}.} Let $S_0\subset R$ denote the set of indices $i\in R$ such that $a_i = 0$. In other words, $S_0$ is the set of \textit{good} strata of $\mathbf{D}$.

\subsection{Low energy computations for fixed point Floer cohomology}
Before beginning our discussion of the log PSS map, we will first show that its domain and codomain are \textit{abstractly} isomorphic as vector spaces. This allows for some simplifications to the proof that the log PSS map is an isomorphism. We find it convenient to work with the model of Floer cohomology studied in~\cite{DostoglouSalamon, DostoglouSalamonBook} for this purpose. Recall that there is a tautological correspondence between horizontal sections of the mapping torus $T_{\phi^d}$ and fixed points of $\phi^d$. There is, similarly, a correspondence between elements of the moduli spaces $\mathcal{M}(E_{\phi^d},J;B)$, for any $B\in\pi_2(\nu^{-},\nu^{+})$, and Floer cylinders counted in the construction of fixed point Floer cohomology in~\cite[\S{2}]{DostoglouSalamon}, which essentially follows from Gromov's trick~\cite[Ch. 8]{MS2}. We will use this to compute $HF^0_{low}(X_t,\phi^d;\Bbbk)$ using the well-known argument~\cite[Lemma 7.1]{HoferSalamonNovikov}.
\begin{lemma}\label{PSS-domain-codomain-abstract-iso}
    There is an abstract isomorphism of $\Bbbk$-modules
    \begin{align*}
        \bigoplus_{d=0}^{\infty}HF^0_{low}(X_t,\phi^d;\Bbbk) \cong \bigoplus_{\mathbf{v}\in\mathbb{Z}_{\geq0}^{S_0}} H^0(U_{\mathbf{v}_0})\,.
    \end{align*}
\end{lemma}
\begin{proof}
    First recall that a standard local computation shows that the orbits $\nu\in\mathcal{H}(T_{\phi^d},\Theta_{\phi^d})$ of degree zero are all contained in the neighborhoods $U_I$ where $I\subset S_0$ corresponds to a good stratum $D_I$ of $\mathbf{D}$. Before perturbing using~\eqref{explicithamiltonianperturbation}, these orbits correspond to the degree $0$ cohomology classes of the Morse--Bott families of orbits in these neighborhoods (cf.~\eqref{stanleyreisnergradingconvention}.) We will show that after the perturbation, the lowest energy cylinders correspond to Morse trajectories of Morse functions on the isolating neighborhoods $U_{\mathbf{v}}$.

    For any fixed $d\geq0$, the fixed points of $\phi = \phi_{reg}$ form Morse--Bott submanifolds of $X_t$. The Hamiltonian perturbations~\eqref{explicithamiltonianperturbation} correspond to time-dependent Hamiltonian perturbations $\mathbb{R}\times X_t\to\mathbb{R}$. Fix some $H_{\delta}$ and let $H_m\coloneqq\frac{1}{m}H_{\delta}$. Given any sequence $u_m$ of $t$-dependent Floer cylinders with respect to $H_m$, a compactness argument detailed in the proof of Lemma 7.1 in~\cite{HoferSalamonNovikov} shows that one can construct a sequence $v_m$ converging to a $t$-independent Floer cylinder $v_{\infty}$. Moreover by comparing Fredholm theories, it follows that for sufficiently large $m$ the $v_m$ coincide with $v_{\infty}$, and it follows from their construction that the $u_q$ are $t$-independent as well. Thus by taking $m$ sufficiently large we obtain the expected correspondence between Morse trajectories and low energy Floer cylinders with respect to a perturbation of the form~\eqref{explicithamiltonianperturbation}. This correspondence also preserves orientation theories by~\cite[Lemma 8.7]{KwonVanKoert}, which implies the desired isomorphism.
\end{proof}

\subsection{Log PSS moduli spaces} For the rest of this section, we will replace the completed family $\mathscr{X}\to\mathbb{C}$ (see Assumption~\ref{completedfamilyassumption}) with its base change $\mathscr{X}\to\Sigma$, where $\Sigma\coloneqq\mathbb{CP}^1\setminus\lbrace 0\rbrace$, under the map $\mathbb{C} \to \Sigma, z \mapsto 1/z$. We think of $\Sigma$ as having a distinguished marked point at $w_{in}\coloneqq\infty$, and the special fiber of $\mathscr{X}$ lies over this point after the base change.

We equip $\Sigma$ with a negative cylindrical end
\begin{align}\label{logPSScylindricalend}
    \epsilon_{in}\colon\mathbb{R}\times S^1&\to\Sigma\\
(s,t)&\mapsto\exp(2\pi(s+it))
\end{align}
We will also think of $\Sigma$ as the base of the family $\mathscr{X}\to\Sigma$ (cf. Assumption~\ref{completedfamilyassumption}.) Following~\cite{MTZ, Tehrani, Pomerleano}, we need to consider almost complex structures which are of a standard form near $\mathbf{D}$.
\begin{definition}\label{almostkahleracsdefinition}
    Let $AK(\mathscr{X},\mathbf{D})$ denote the space of almost complex structures on $\mathscr{X}$ such that:
    \begin{itemize}
        \item[(i)] each component $D_i$ of $\mathbf{D}$ is $J$-holomorphic;
        \item[(ii)] for any $i\in R$ and $p\in D_i$, the image of the Nijenhuis tensor defined using $J$ on $T_p\mathscr{X}$ is contained in $T_p D_i$;
        \item[(iii)] $J$ respects the decomposition
        \begin{align*}
            T_p U_I\cong\pi^*_ITD_I\oplus\pi_I^* ND_I
        \end{align*}
        and;
        \item[(iv)] with respect to this decomposition, $J$ is a direct sum of almost complexes structures whose $\pi_I^*TD_I$-component is pulled back from $D_I$ and whose second component, which is an almost complex structure on
        \begin{align*}
            \pi_I^* ND_I \cong \pi_I^*(\bigoplus_{i\in I}ND_i)
        \end{align*}
        agrees with the direct sum of pullbacks of almost complex structures on $ND_i\mid_{D_I}$ induced by the regularization.
    \end{itemize}
    We say that an almost complex structure $J\in AK(\mathscr{X},\mathbf{D})$ is \textit{standard} near $\mathbf{D}$.
\end{definition}
The domain-dependent almost complex structures which we use to define $\PSS_{\log}$ are given as follows.
\begin{definition}\label{logPSSdomaindependentacs}
    Let $\mathcal{J}_{\Sigma}(\mathscr{X},\mathbf{D})$ denote the space of domain-dependent almost complex structures $J_{\Sigma}$ on $\Sigma$ such that
    \begin{itemize}
        \item[(i)] for all $z\in\Sigma$, the full divisor $\mathbf{V} = \mathbf{D}\cup\mathbf{E}$ is $J_{\Sigma,z}$-holomorphic;
        \item[(ii)] for each $J_{\Sigma,z}$, the projection map $\pi_{\mathscr{X}}\colon\mathscr{X}\to\Sigma$ is $(J_{\Sigma,z},j)$-holomorphic, where $j$ is the standard almost complex structure on $\Sigma$, and $\Omega(\cdot,J_{\Sigma,z}\cdot)$ is symmetric on $T\mathscr{X}$ and positive-definite on the vertical bundle $T\mathscr{X}^v = \ker(D\pi_{\mathscr{X}})$;
        \item[(iii)] in a neighborhood of $w_{in}$, the domain-dependent almost complex structure $J_{\Sigma}$ is constant, with value $J_0\in AK(\mathscr{X},\mathbf{D})$, and;
        \item[(iv)] when restricted to the region of the cylindrical end~\eqref{logPSScylindricalend} where $s<\!<0$, the almost complex structures $J_{\Sigma,z}$ depend only on the $t$--coordinate, and their restrictions to the shell region of $\mathscr{X} = \mathscr{X}_{\epsilon'}$ lie in
        \begin{align}\label{fibractionacswithstabilizingdivisor}
        \mathcal{J}(T_{\phi},\Theta_{\phi};\mathbf{E})
        \end{align}
        the subspace of $\mathcal{J}(T_{\phi},\Theta_{\phi})$ consisting of all almost complex structures for which all components $E_q$ of $\mathbf{E}$ are almost complex submanifolds.
    \end{itemize}
\end{definition}
We need to restrict to domain-dependent almost complex structures that are standard near the marked point $w_{in}$ in order to apply the compactness arguments of~\cite[Proposition 3.15]{Tehrani}. The (log) PSS solutions which the log PSS map counts are defined to be pseudoholomorphic multisections with respect to some almost complex structure in $\mathcal{J}_{\Sigma}(\mathscr{X},\mathbf{D})$.
\begin{definition}
    Fix $J_{\Sigma}\in\mathcal{J}_{\Sigma}(\mathscr{X},\mathbf{D})$. We say that a(n ordinary) PSS solution is a map $u\colon\Sigma\to\mathscr{X}$ such that
\begin{itemize}
    \item[$\bullet$] $u$ is $J_{\Sigma}$-holomorphic, and;
    \item[$\bullet$] $u$ converges to a horizontal section $\nu$ in $\mathcal{H}(T_{\phi^d},\Theta_{\phi^d})$ under the correspondence discussed in Section \ref{multiplecoverfixedpointfloer}, meaning that
    \[
        \lim_{s\to-\infty}u(\epsilon(s,t)) = \nu \,.
    \]
\end{itemize}
\end{definition}

\begin{definition}
    For any capped orbit $\tilde{\nu} = (\nu,[u_{\nu}])$, let $\mathcal{M}(\tilde{\nu})$ denote the moduli space of PSS solutions $u\colon\Sigma\to\mathscr{X}$ asymptotic to $\nu$ and such that $-u$ is equivalent to $u_{\nu}$ as a capping disk. Define
    \begin{align*}
        \mathcal{M}(\nu)\coloneqq\bigcup_{[u_{\nu}]}\mathcal{M}(\nu,u_{\nu})
    \end{align*}
\end{definition}
We come now to the definition of a log PSS solution.
\begin{definition}
    A \textit{(low energy) log PSS solution} with multiplicity $\mathbf{v}$ is an element $u\in\mathcal{M}(\nu)$ such that
    \begin{itemize}
        \item[$\bullet$] $u(z)\not\in\mathbf{D}$ for $z\neq w_{in}$, and;
        \item[$\bullet$] the intersection multiplicity of $u$ with $D_i$ at $w_{in}$ is $v_i$.
    \end{itemize}
   We denote the moduli space of (low energy) log PSS solutions by $\mathcal{M}(\mathbf{v},\nu)$.
\end{definition}
By the Riemann--Hurwitz theorem, the composition $\pi_{\mathscr{X}}\circ u\colon\Sigma\to\Sigma$ is necessarily a $d$-fold branched cover with a single branch point at $w_{in}$.

The moduli spaces $\mathcal{M}(\mathbf{v},\nu)$ have virtual dimension
\begin{align}
    \operatorname{vdim}\mathcal{M}(\mathbf{v},\nu) = \mu(\nu)
\end{align}
by~\cite[Lemma 4.12]{GP1}. Following~\cite{Tehrani, Pomerleano}, we will show that the log PSS moduli spaces admit \textit{log compactifications} $\overline{\mathcal{M}}(\mathbf{v},\nu)$ with the property that all compactifying strata $\overline{\mathcal{M}}(\mathbf{v},\nu)\setminus\mathcal{M}(\mathbf{v},\nu)$ have virtual codimension two or higher. We will regularize the compactified log PSS moduli spaces using the stabilizing divisor $\mathbf{E}$~\cite{CieliebakMohnke}, and once this is done it will follow that we can rule out sphere bubbling, including sphere bubbles contained in the divisor $\mathbf{D}$, in all log PSS moduli spaces of virtual dimension at most $1$.

As usual, the nodal curves appearing in this compactification are modeled on trees $\underline{\Gamma}$. We denote by $\Ver {\underline{\Gamma}}$ the set of vertices of $\underline{\Gamma}$ and by $\Edge(\underline{\Gamma})$ the set of \textit{internal} edges. For any vertex $v\in\Ver\underline{\Gamma}$, let $\Edge(v)$ denote the set of internal edges adjacent to $v$. Similarly let $\Edge^{\to}(\underline{\Gamma})$ (resp. $\Edge^{\to}(v)$) denote the set of oriented internal edges of $\Gamma$ (resp. the set of oriented internal edges adjacent to $v$.) Given a pair of adjacent vertices $v$ and $v'$, the oriented edge from $v$ to $v'$ is written as $e_{v,v'}$. Finally, the unbounded edges of a tree $\Gamma$ are called \textit{leaves}.
\begin{definition}
    Define an $\ell$-\textit{marked} tree $\underline{\Gamma}^{(\ell)}$ to be a rooted tree, with root $v_{root}$, a distinguished leaf $\lambda_{in}$, and a collection of additional leaves $\lambda_i$ labeled by elements of $\lbrace 2,\ldots,\ell+1\rbrace$ such that for all vertices other than the root we have that
\begin{align*}
    |\Edge(v)|+|\Leaf(v)|\geq 3
\end{align*}
where $\Leaf(v)$ denotes the set of leaves adjacent to $v\in\Ver\underline{\Gamma}$.
\end{definition}

The trees underlying the domains of a log PSS solution with sphere bubbles carry two additional pieces of data. Let $\mathcal{P}(S)$ denote the power set of $S = R\cup Q$. First, given an $r$-marked tree $\underline{\Gamma}^{(\ell)}$, a \textit{stabilized depth function} is a pair of functions
\begin{align*}
    I^{v}\colon\Ver(\underline{\Gamma}^{(\ell)})\to\mathcal{P}(S) \\
    I^e\colon\Edge(\underline{\Gamma}^{(\ell)})\to\mathcal{P}(S)
\end{align*}
such that $I^{v_{root}} = \emptyset$ and for any edge $e$ with endpoints $v$ and $v'$ we have that $I^v\cup I^{v'} = I^e$. Second, a \textit{contact function} is a map
\begin{align*}
    \cont(-)\colon\Edge^{\to}(\underline{\Gamma}^{(\ell)})\to\mathbb{Z}^S
\end{align*}
such that for any pair of adjacent vertices $v$ and $v'$ we have that
\begin{align*}
    \cont(e_{v,v'}) = -\cont(e_{v',v})
\end{align*}
and that the support of $\cont(e_{v,v'})$ is contained in $I^e$, where $e$ denotes the edge underlying the oriented edge $e_{v,v'}$.

\begin{definition}
    A \textit{log PSS tree} $\Gamma = (\Gamma^{(\ell)},I^{\bullet},\cont(-))$ consists of
    \begin{itemize}
        \item[(i)] an $\ell$-marked tree $\underline{\Gamma}^{(\ell)}$ such that the complement of the subtree consisting of $v_{root}$ and $\Leaf(v_{root})$ is connected;
        \item[(ii)] a depth function $I^v$ and $I^e$, and;
        \item[(iii)] a contact function $\cont(-)$.
    \end{itemize}
    We call the trees obtained by deleting $v_{root}$ and all of its adjacent leaves as in (i) the \textit{bubble trees} of $\underline{\Gamma}^{(\ell)}$.
\end{definition}

Next we describe the spaces of domains and perturbation data underlying bubbled log PSS solutions. Suppose that $C$ is a nodal Riemann surface (with marked points) obtained by attaching trees of sphere bubbles to $\Sigma$. Such a curve will be modeled on some $\Gamma^{(\ell)}$, with $v_{root}$ corresponding to the $\Sigma$-component. Let $\overline{C}$ denote the domain obtained by compactifying $\Sigma$ to $\mathbb{CP}^1$. For any integer $\ell\geq1$, consider the Deligne--Mumford moduli space $\overline{\mathcal{R}}_{\ell+2}$ of stable genus zero curves with $\ell+2$ marked points, which comes with a universal curve
\begin{align}
    \overline{\mathcal{S}}_{\ell+2}\to\overline{\mathcal{R}}_{\ell+2}
\end{align}
with sections determined by the marked points. Fix a domain-dependent almost complex structure $J_{\Sigma}\in\mathcal{J}_{\Sigma}(\mathscr{X},\mathbf{D})$ (see Definition~\ref{logPSSdomaindependentacs}) which is constant with value $J_0\in AK(\mathscr{X},\mathbf{D})$ in some small open neighborhood $U_{\Sigma}$ of $w_{in}$.
\begin{definition}[\cite{RuanTian}]
    A \textit{(Ruan--Tian) perturbation datum} is a one-form
    \begin{align}
        \upsilon\in C^{\infty}(\overline{\mathcal{S}}_{\ell+2}\times\mathscr{X},\Omega^{0,1}_{\overline{\mathcal{S}}_{\ell+2}/\overline{\mathcal{R}}_{\ell+2}}\otimes_{\mathbb{C}} T\mathscr{X})
    \end{align}
    supported away from all nodes and marked points, where $T\mathscr{X}$ is given the almost complex structure $J_0$.
\end{definition}
The existence of coherent systems of perturbation data, in a suitable sense, follows from standard arguments~\cite[Lemma 9.5]{SeiBook}.

For any capped orbit $\tilde{\nu} = (\nu,[u_{\nu}])$ whose underlying orbit $\nu$ is contained in $\mathscr{X}\setminus\mathbf{E}$, define
\begin{align*}
    \ell(\tilde\nu)\coloneqq\sum_{q\in Q}[u_{\nu}]\cdot E_q \,.
\end{align*}

\begin{definition}[{\cite[Definition 4.13]{Pomerleano}}]
    Let $\tilde{\nu} = (\nu,[u_{\nu}])$ be a capped orbit with $\nu\subset\mathscr{X}\setminus\mathbf{E}$ and $\ell = \ell(\tilde\nu)\geq 1$, and let $\underline{\Gamma}^{(\ell(\tilde\nu))}$ be a marked tree.
    A \textit{(perturbed) bubbled PSS solution} based on the rooted tree $\underline{\Gamma}^{(\ell(\tilde\nu))}$ and asymptotic to $\nu\in\mathcal{H}(T_{\phi^d},\Theta_{\phi^d})$ consists of a domain $C$ modeled on $\underline{\Gamma}^{(\ell(\tilde\nu))}$, meaning that $v_{root}$ is assigned a copy of $\Sigma$, together with the following.
    \begin{itemize}
        \item[(i)] A holomorphic map $\tau\colon\overline{C}\to\overline{\mathcal{S}}_{\ell+2}$ onto a fiber of the universal family which is an isomorphism.
        \item[(ii)] A map $u\colon C\to\mathscr{X}$ which satisfies 
        \begin{align*}
            (Du)^{0,1}-(\tau,u)^*\upsilon = 0 \,.
        \end{align*}
        Letting $C_{v}$ denote the domain of the curve $u_v\colon C_v\to\mathscr{X}$ associated to $v\in\Ver(\underline{\Gamma}^{(\ell)})$, and $z_{e_{v,v'}}$ denote the marked point of $C_{v}$ associated to the oriented edge $e_{v,v'}$ (and $z_{e_{v',v}}$ denote the corresponding marked point of $C_{v'}$), we demand that
        \begin{align*}
            u_v(z_{e_{v,v'}}) = u_{v'}(z_{e_{v',v}}) \,.
        \end{align*}
        Furthermore, each tree of sphere bubbles corresponding to a bubble tree of $\underline{\Gamma}^{(\ell)}$ is required to map to a single fiber of $\mathscr{X}\to\mathbb{C}$.
    \end{itemize}
    We also require that the marked point corresponding to the edge closest to the leaf $\lambda_{in}$ is the distinguished marked point $w_{in}$.
\end{definition}
There is an obvious notion of an isomorphism of bubbled PSS solutions, and we say that a bubbled PSS solution, denoted $(C,u) = (C_v,u_v)_{v\in\Ver(\underline{\Gamma}^{(\ell)})}$, is \textit{stable} if its automorphism group is finite. Let
\begin{align}
    \mathcal{M}(\Gamma^{(\ell)},\nu)
\end{align}
denote the moduli space of stable bubbled PSS solutions asymptotic to $\nu\in\mathcal{H}(T_{\phi^d},\Theta_{\phi^d})$. If we are given a bubbled PSS solution, we can associate a new function $I^{\bullet}_{(C_v,u_v)_{v\in\Ver(\underline{\Gamma}^{(\ell)})}}$ to it as follows. We say that $u_v\colon C_v\to\mathscr{X}$ has \textit{depth} $I_{u_v}\subset R$ if $u_v(C_v)\subset D_{I_{u_v}}$ and $u_v(C_v)\not\subset D_j$ for all $j\not\in I_{u_v}$. We define
\begin{itemize}
    \item[$\bullet$] $I^v_{(C,u)}\coloneqq I_{u_v}$, and;
    \item[$\bullet$] for any edge $e\in\Edge(\underline{\Gamma}^{(\ell)})$ with endpoints denoted by $v$ and $v'$, define $I^e_{(C,u)}$ to be the depth of the point $u_v(z_{e_{v,v'}})$, meaning the subset corresponding to the deepest stratum containing this point.
\end{itemize}
This is not a depth function in general, because we only have that $I^{v}_{(C,u)}\cup I^{v'}_{(C,u)}\subset I^e_{(C,u)}$, but we will restrict ourselves to bubbled PSS solutions for which this assignment is a depth function.

Suppose that $(C_v,u_v)$ is a component of $(C,u)$ corresponding to a vertex in the bubble tree attached to $\lbrace\infty\rbrace$ in $\Sigma = \mathbb{CP}^1\setminus\lbrace 0\rbrace$. This implies that $u_v(C_v)$ is contained in some irreducible components $D_i$ of $\mathbf{D}$. Following~\cite{Tehrani2}, we can enhance the moduli space of bubbled stable PSS solutions with additional data allowing us to define the order function~\eqref{order-function} below for components $u_v\colon C_v\to\mathscr{X}$ with $D_i$. Note that if $u_v(C_v)\subset D_i$, then the operator $D_{u_v}(\overline{\partial}-\upsilon)$ descends to an operator
\begin{align*}
    D_{u_v}^{ND_i}(\overline{\partial}-\upsilon)\colon\Gamma(C_v,u_v^*ND_i)\to\Gamma(\Sigma,\Omega^{0,1}_{C_v}\otimes_{\mathbb{C}}u_v^*ND_i)
\end{align*}
Since the almost complex structure $J_0$ is integrable in the normal direction (see Definition~\ref{almostkahleracsdefinition}(ii)), this operator gives $u_v^*ND_i$ the structure of a holomorphic line bundle by~\cite[Lemma 3.5]{Tehrani2}. Thus we can form the space $\Gamma_{mero}(C_v,u_v^*ND_i)$ of meromorphic sections of this line bundle. The special form the perturbation datum $v$ as in Definition \ref{defn:rt-perturbation} ensures that the space of meromorphic sections defined using $D_{u_v}^{ND_i}(\overline{\partial}-\upsilon)$ are invariant under the $\mathbb{C}^*$-scaling. Let $[\zeta]\in\Gamma_{mero}(C_v,u_v^*ND_i)/\mathbb{C}^*$ denote an equivalence class of meromorphic sections modulo rescaling.
\begin{definition}
    A \textit{log curve} $(C,u,[\zeta])$ consists of a Riemann surface $C$ and  a pseudoholomorphic map $u\colon C\to\mathscr{X}$ with an sequence of equivalence classes of meromorphic sections
    \begin{align*}
        [\zeta] = ([\zeta_i])_{i\in I_u}\in\prod_{i\in I_u}\Gamma_{mero}(\Sigma,u^*ND_i)/\mathbb{C}^*.
    \end{align*}
\end{definition}
For any log curve, there is an order function
\begin{align}\label{order-function}
    \ord\colon C &\to\mathbb{Z}^R \\
    z &\mapsto(\ord_{u,i}(z))_{i\in R} \nonumber
\end{align}
where $\ord_{u,i}(z)$ is the order of any zero or pole of $\zeta_i$ at $z$ (for any choice of lift of $[\zeta_i]$) if $i\in I_u$, and where $\ord_{u,i}(z)$ is the standard order of contact with $D_i$ if $i\not\in I_u$.

We also need to restrict to perturbation data which are compatible with the regularization $\mathfrak{R}$, in the sense described in~\cite[Definition 3.7]{Tehrani2}. We follow the exposition in~\cite{Pomerleano}. Recall from~\cite[(2.11)]{Tehrani2} that there is a \textit{logarithmic tangent bundle} $T\mathscr{X}(-\log\mathbf{D})$ defined using the regularization $\mathfrak{R}$, which comes with a natural map
\begin{align}\label{logtangentbundletotangentbundle}
    T\mathscr{X}(-\log\mathbf{D})\to T\mathscr{X}
\end{align}
which is an isomorphism away from $\mathbf{D}$. If $\pi_I\colon U_I\to D_I$ is a regularizing neighborhood, where $I\subset R$, then over $U_I^{\circ}\coloneqq\pi_I^{-1}(D_I\setminus\bigcup_{j\not\in I}D_j)$ there is a decomposition
\begin{align}\label{logtangentbundlelocaldecomposition}
    T\mathscr{X}(-\log\mathbf{D})\mid_{U_I^{\circ}}\cong\pi_I^*(TD_I)\oplus\mathbb{C}^I
\end{align}

\begin{definition}[{\cite[Definition 4.15]{Pomerleano}}]\label{defn:rt-perturbation}
    A Ruan--Tian perturbation datum $\upsilon$ is compatible with the regularization $\mathfrak{R}$ if it lifts to a section
    \begin{align*}
        \upsilon_{\log}\in C^{\infty}(\overline{\mathcal{S}}_{\ell+2}\times\mathscr{X},\Omega^{0,1}_{\overline{\mathcal{S}}_{\ell+2}}\otimes_{\mathbb{C}} T\mathscr{X}(-\log\mathbf{D})) \,.
    \end{align*}
    and if, with respect to the decomposition~\eqref{logtangentbundlelocaldecomposition}, $\upsilon$ looks like
    \begin{align*}
        \upsilon = \pi_I^*\upsilon_I\oplus\pi_I^*\upsilon_F
    \end{align*}
    where in the horizontal component
    \[
        \upsilon_I\in C^{\infty}(\overline{\mathcal{S}}_{\ell+2}\times D_I,\Omega^{0,1}_{\overline{\mathcal{S}}_{\ell+2}}\otimes_{\mathbb{C}}TD_I)
    \]
    is a $1$-form on $D_I$, and on the vertical component
    \[
          \upsilon_F\in C^{\infty}(\overline{\mathcal{S}}_{\ell+2}\times D_I,\Omega^{0,1}_{\overline{\mathcal{S}}_{\ell+2}}\otimes_{\mathbb{C}}\mathbb{C}^I)
    \]
    is a sum of $1$-forms in the factors of $\mathbb{C}^I$.
\end{definition}
In the language of~\cite[(3.18) and Definition 3.7]{Tehrani2}, the Ruan--Tian perturbation datum $\upsilon$ is said to be associated to the \textit{log perturbation datum} $\upsilon_{\log}$.

Let $\mathbf{v}_{E_q}\in\mathbb{Z}^S$ denote the vector with value $1$ at the index $q\in Q$ and value $0$ at all other indices in $S = R\cup Q$. For a fixed orbit $\nu\in\mathcal{H}(T_{\phi^d},\Theta_{\phi^d})$, each class $[u_{\nu}]$ has a well-defined intersection number $[u_{\nu}]\cdot E_q$ for each component of $\mathbf{E}$, and we define a \textit{labeling} $\mathbf{b}\colon\lbrace 2,\ldots,\ell+1\rbrace\to Q$ to be a function such that $|\mathbf{b}^{-1}(q)| = [u_{\nu}]\cdot E_q$ for all $q\in Q$.
\begin{definition}
    Let $\Gamma = (\underline{\Gamma}^{(\ell)},I^{\bullet},\cont(-))$ be a log PSS tree. A \textit{pre-log (stable) map of mulitplicity $\mathbf{v}$} modeled on $\underline{\Gamma}^{(\ell)}$ and asymptotic to $\nu\in\mathcal{H}(T_{\phi^d},\Theta_{\phi^d})$ consists of the data
    \[(C_v,u_v,[\zeta]_v)_{v\in\Ver(\underline{\Gamma}^{(\ell)})}\]
    defined as follows.
    \begin{itemize}
        \item[(i)] $(C_v,u_v)_{v\in\Ver(\underline{\Gamma}^{(\ell)})}$ define a stable PSS solution $C\to\mathscr{X}$ asymptotic to $\nu$ and modeled on $\underline{\Gamma}^{(\ell)}$ such that the function $I^{\bullet}_{(C_v,u_v)}$ (induced by the depth of the $u_v$'s) is equal to the given depth function $I^{\bullet}$.

        \item[(ii)] For each non-root vertex $v\in\Ver(\underline{\Gamma}^{(\ell)})\setminus\lbrace v_{root}\rbrace$, the data $(C_v,u_v,[\zeta]_v)$ defines a log curve such that the associated order functions $(\ord_v)_{v\in\Ver\underline{\Gamma}^{(\ell)}}$ satisfying the following constraints.
        \begin{itemize}
            \item[(a)] $\ord_v$ is non-vanishing only at the marked points corresponding to leaves and to edges $e_{v,v'}\in\Edge^{\to}(\underline{\Gamma}^{(\ell)})$;

            \item[(b)] we have that $\ord_v(w_{in}) = \mathbf{v}$, where $w_{in}$ denotes the marked point corresponding to $\ell_{in}$, and that for any vertex $v$,
            \begin{align*}
                \ord_v(z_e) = \cont(e_{v,v'})
            \end{align*}
            and;
            \item[(c)] at the point $w_j$ corresponding to the leaf $\ell_j$, where $i\in\lbrace 2,\ldots,\ell+1\rbrace$, we have that $u(w_i)\in E_{\mathbf{b}(i)}$ and $\ord_u(w_i) = \mathbf{v}_{E_{\mathbf{b}(i)}}$.
        \end{itemize}
    \end{itemize}
    Denote by
    \begin{align*}
        \mathcal{M}_{\mathbf{E}}^{\plog}(\mathbf{v},\Gamma,\tilde{\nu},\mathbf{b})
    \end{align*}
    the moduli space of pre-log stable maps of multiplcity $\mathbf{v}$ modeled on $\Gamma$, and with stabilizing data $\mathbf{b}$, which for a given capped orbit $\tilde{\nu} = (\nu,[u_{\nu}])$ is asymptotic to $\nu$ and is equivalent, as a capping disk, to $u_{\nu}$.
\end{definition}
\begin{remark}
    To match this with the language of~\cite{PerutzSheridan} which we use below, we remark that the labeling used to define the $\mathbf{E}$-marked moduli spaces determines a \textit{canonical tangency data}~\cite[(19)]{PerutzSheridan}. Though more general tangency data can be considered, when we restrict to moduli spaces of virtual dimension $\leq 1$, only those moduli spaces defined with canonical tangency data are nonempty. It is unclear whether we can follow the methods of~\cite{PerutzSheridan} when defining the log-PSS maps to, say, make our perturbations constant on sphere bubble components or define the log-PSS classes directly over fields of positive characteristic (cf.~\cite[Remark 4.23]{Pomerleano} and~\cite[\S{4.5}]{Tehrani2}.) 
\end{remark}
To define the version log PSS moduli spaces that we will need, we make use of the obstruction map $ob_{\Gamma}$ on the moduli space $\mathcal{M}_{\mathbf{E}}^{\plog}(\mathbf{v},\Gamma,\tilde{\nu},\mathbf{b})$, whose definition is identical to the one given in~\cite[(3.33)]{Pomerleano}, and which we will not repeat here.
\begin{definition}
    A(n $\mathbf{E}$-\textit{marked) log map} modeled on $\Gamma$ consists of a pre-log map
    \[
    (C_v,u_v,[\zeta]_v)_{v\in\Ver(\underline{\Gamma}^{(\ell)})}\in \mathcal{M}_{\mathbf{E}}^{\plog}(\mathbf{v},\Gamma,\tilde{\nu},\mathbf{b})
    \]
    with vanishing obstruction $ob_{\Gamma}(u)$, and for which there exist functions
    \begin{align*}
        \cont_{v}\colon\Ver(\underline{\Gamma}^{(\ell(\tilde\nu))})&\to\mathbb{R}^S \\
        \lambda_e\colon\Edge(\underline{\Gamma}^{(\ell(\tilde\nu))})&\to\mathbb{R}_{>0}
    \end{align*}
    such that
    \begin{itemize}
        \item[$\bullet$] $\cont_v\in \mathbb{R}_{>0}^{I^v}\times\lbrace 0\rbrace^{S\setminus I^v}$, and;
        \item[$\bullet$] for each pair of vertices $v$ and $v'$ connected by an oriented edge $e_{v,v'}$, we have that
        \begin{align*}
            \cont_v-\cont_{v'} = \lambda_e(e_{v,v'})\cont(e_{v,v'}) \,.
        \end{align*}
    \end{itemize}
    The moduli spaces of $\mathbf{E}$-marked log stable maps are denoted
    \begin{align}\label{markedlogmodulispacesgraphs}
        \mathcal{M}_{\mathbf{E}}(\mathbf{v},\Gamma,\tilde{\nu},\mathbf{b})
    \end{align}
    and when $\Gamma$ has a single vertex we set
    \begin{align*}
        \mathcal{M}_{\mathbf{E}}(\mathbf{v},\tilde{\nu},\mathbf{b})\coloneqq\mathcal{M}_{\mathbf{E}}(\mathbf{v},\Gamma,\tilde{\nu},\mathbf{b}) \,.
    \end{align*}
    Finally, we let
    \begin{align}\label{markedlogmodulispaces}
        \mathcal{M}_{\mathbf{E}}(\mathbf{v},\tilde{\nu}) \coloneqq\bigsqcup_{\substack{\mathbf{b}\colon\lbrace 2,\ldots,\ell+1\rbrace\to Q \\ \ell\geq1}}\mathcal{M}_{\mathbf{E}}(\mathbf{v},\tilde{\nu},\mathbf{b})
    \end{align}
    denote the union of moduli spaces ranging over all labellings.
    
\end{definition}
Having defined the moduli spaces of log stable maps as such, general dimension counting arguments~\cite[Propositions 3.18 and 3.21]{Pomerleano} show that the only moduli spaces~\eqref{markedlogmodulispaces} of virtual dimension $\mu(\nu) = 0$ are only nonempty when $\Gamma$ is a point, and all other such moduli spaces have virtual dimension at most $\mu(\nu)-2$. Using the stabilizing divisor $\mathbf{E}$, we can show that the moduli spaces of marked log maps of virtual dimension at most $1$ are transversely cut out.
\begin{lemma}
    For generic $(J_{\Sigma},\upsilon)$, the moduli spaces~\eqref{markedlogmodulispacesgraphs} are empty when $\deg(\nu)\leq 1$ and the tree underlying $\Gamma$ has more than one vertex.
\end{lemma}
The proof of this is lemma similar to the proof the analogous result given in~\cite{Pomerleano} (see the proof of Lemma 4.17 and the references therein) and we omit it. We end this subsection by stating the analogue of~\cite[Lemma 4.18]{Pomerleano}. To help clarify the statement of this lemma, we introduce a small modification to the definition of the moduli spaces $\mathcal{M}(E_{\phi^d};B)$ for $B\in\pi_2(\nu^{-},\nu^{+})$. Namely, we note that any such class $B$ has a well-defined intersection number with $E_q$ for any $q\in Q$, and all of these intersection numbers are nonnegative whenever the corresponding moduli space is nonempty. Thus if we consider the cylinder with a sequence of stabilizing marked points, we can formulate the notion of a labeling $\mathbf{b}$, and define $\mathbf{E}$-marked versions of the moduli spaces of cylinders
\begin{align*}
    \mathcal{M}_{\mathbf{E}}(E_{\phi^d};B;\mathbf{b})
\end{align*}
in the obvious way. In this context, these moduli spaces are isomorphic to the original moduli spaces $\mathcal{M}(E_{\phi^d};B)$ for any choice of labeling (compare with the proof of Lemma 4.10 in~\cite{Pomerleano}.)

\begin{lemma}\label{logPSScompactness}
    Let $\tilde{\nu} = (\nu,[u_{\nu}])$ be a capped orbit such that $[u_{\nu}]\cdot\mathbf{D} = \mathbf{v}$. Then
    \begin{itemize}
        \item[(i)] The moduli spaces~\eqref{markedlogmodulispaces} are compact when $\deg(\nu) = 0$.

        \item[(ii)] When $\deg(\nu) = 1$, the moduli spaces~\eqref{markedlogmodulispaces} admit Gromov compactifications denoted
        \begin{align}\label{markedlogmodulispacesgromovcompact}
            \overline{\mathcal{M}}_{\mathbf{E}}(\mathbf{v},\tilde{\nu})
        \end{align}
        with boundary strata
    
        \begin{align*}
            \partial\overline{\mathcal{M}}_{\mathbf{E}}(\mathbf{v},\tilde{\nu})\coloneq\bigsqcup_{(\nu',B)}\bigsqcup_{(P,\mathbf{b}_0,\mathbf{b}_1)}\mathcal{M}_{\mathbf{E}}(\mathbf{v},\tilde{\nu}',\mathbf{b}_0)\times\mathcal{M}_{\mathbf{E}}(E_{\phi^d};B;\mathbf{b}_1)
        \end{align*}
        
        where
        \begin{itemize}
            \item[$\bullet$] $\nu'\in\mathcal{H}(T_{\phi^d},\Theta_{\phi^d})$ is an orbit with $\mu(\nu') = 0$;
            \item[$\bullet$] $B\in\pi_2(\nu',\nu)$ is a homotopy class of sections with positive limit $\nu$ and negative limit $\nu'$;
            
            \item[$\bullet$] $P$ is a partition
            \begin{align*}
                \lbrace 2,\ldots,\ell(\tilde\nu)+1\rbrace = P_0\cup P_1
            \end{align*}
            and;
            \item[$\bullet$] $\mathbf{b}_i\colon P_i\to Q$ are labelings for the respective moduli spaces.
            
        \end{itemize}
     \end{itemize}
\end{lemma}
\begin{proof}
    First note that a sequence of log PSS solutions cannot limit to a curve broken along a Hamiltonian oribt of~\eqref{perturbedHamiltonianondegeneration} in our situation, since $H_{\mathscr{X}}$ vanishes along $\mathbf{D}$.

    The moduli spaces~\eqref{markedlogmodulispacesgromovcompact} are topologized as subspaces of the Gromov compactification of ordinary PSS solutions. Given a sequence of log PSS solutions, its limit may contain a new component which sinks into a component $D_i$ of $\mathbf{D}$, and this component needs to inherit an equivalence class of meromorphic sections in $\Gamma_{mero}(C_v,u_v^*ND_i)/\mathbb{C}^*$. Such meromorphic sections can be produced using the rescaling construction of~\cite[\S{3.2}]{Tehrani}, and this is the only place in our discussion where it is necessary to use domain-dependent almost complex structures taking values in $AK(\mathscr{X},\mathbf{D})$ near $w_{in}$ (since it implicitly uses the comparison of rescaling and gluing parameters of~\cite[Proposition 3.15]{Tehrani}.) This rescaling construction shows that if we are given any sequence of stable log PSS solutions, there is a subsequence which converges as a sequence of log PSS solutions in the same sense described in~\cite[Definition 3.25]{Pomerleano}.

    Since we are considering moduli spaces of virtual dimension at most $1$, the dimension-count above shows that there are no sphere bubbles for generic choices of $J_{\Sigma}$ (be they contained in $\mathbf{D}$ or otherwise.) In (ii), the boundary strata correspond to cylinder breakings along orbits away from $\mathbf{D}$ as expected.
\end{proof}

\subsection{The log PSS map} For the next definition, we restrict to fields $\Bbbk$ of characteristic $0$. For any $\mathbf{v}\in\mathbb{Z}_{\geq0}^{S_0}$, define
\begin{align*}
    \PSS^{\log}(\theta^{\mathbf{v}})\in CF^0_{\mathscr{X}}(X_t,\phi^d)
\end{align*}
to be the element
\begin{align}\label{logPSSclasses}
    \PSS^{\log}(\theta^{\mathbf{v}}) = \sum_{\substack{\tilde{\nu} = [\nu,u_{\nu}] \\ \mu(\nu) = 0}}\sum_{u\in\mathcal{M}_{\mathbf{E}}(\mathbf{v},\tilde{\nu})} \frac{1}{\ell(\tilde\nu)!}T^{\Omega(u)}\sigma_u
\end{align}
where the sum is over all equivalence classes $\tilde{\nu}$ of capped orbits with underlying orbit $\nu$ of degree zero, $\Omega(u)\coloneqq\int_{\Sigma}u^*\Omega$, and $\sigma_u\in o_{\tilde{\nu}}$ is an element of the orientation line defined using standard gluing and orientation results~\cite[Ch. 11]{SeiBook}.
\begin{lemma}
    The elements~\eqref{logPSSclasses} determined a cohomology classes in $HF^0_{\mathscr{X}}(X_t,\phi^d)$ for each $\mathbf{v}\in\mathbb{Z}^{S_0}_{\geq0}$.
\end{lemma}
\begin{proof}
    This is immediate from Lemma~\ref{logPSScompactness}.
\end{proof}

For the rest of this subsection and all of the next subsection, we allow our ground field $\Bbbk$ to be of arbitrary characteristic once again. The \textit{low energy log PSS map}
\begin{align}\label{lowenergylogPSSmap}
    \PSS^{\log}_{low}\colon \SR(\mathcal{X}^{\circ})\to \bigoplus_{d=0}^{\infty}HF^0_{low}(X_t,\phi^d;\Bbbk)
\end{align}
is defined as in~\eqref{logPSSclasses}, except that we only count log PSS solutions in the class $\tilde{\nu} = [\nu,-F_{\nu}]$, where $-F_{\nu}$ is the canonical (fiber) capping disk with reversed orientation. More precisely its values at the chain level are given by
\begin{align}
    \PSS^{\log}_{low}(\theta^{\mathbf{v}}) = \sum_{\mu(\nu) = 0}\sum_{u\in\mathcal{M}_{\mathbf{E}}(\mathbf{v},[\nu,-F_{\nu}])} \sigma_u \in CF^0_{\mathscr{X}}(X_t,\phi^d;\Bbbk)
\end{align}
Again, this is a chain map by Gromov compactness applied to the one-dimensional moduli spaces of low energy log PSS solutions. We note that our restriction to low energy log PSS solutions means that the pseudoholomorphic cylinders appearing as components in the relevant broken log PSS solutions inhabiting the boundaries of these moduli spaces are necessarily cylinders counted in the differential on~\eqref{lowenergyfixedpointfloercochains}.

\subsection{Low energy log PSS is a ring homomorphism}
We currently lack the tools to show that~\eqref{logPSSclasses} defines a \textit{ring} homomorphism, since this would presumably require a symplectic construction of the log Gromov--Witten invariants of~\cite{GSintrinsicmirrors}. In particular, one would need to make sense of negative contact orders in symplectic terms. Fortunately, we can show that the low energy log PSS map~\eqref{lowenergylogPSSmap} is a ring map by adapting the arguments of~\cite[\S{3.4}]{GP2}. Since we are only working with degree zero Floer cohomology groups, this essentially reduces to a standard result about gluing constant sphere bubbles.

For some small $b>0$ and any $q\in(0,b]$, let $\Sigma_a$ denote $\mathbb{CP}^1\setminus\lbrace 0\rbrace$, with a negative cylindrical end~\eqref{logPSScylindricalend} as before, with two marked points $w_{in} = \infty$ and $w_{a} = -1/a$. We can equip each of these domains with suitable domain-dependent almost complex structures $J_{\Sigma_a}$ which we take to be standard, i.e. constant with value in $AK(\mathscr{X},\mathbf{D})$, near both marked points $w_{in}$ and $w_a$.
\begin{definition}
    Define $\mathcal{M}(\mathbf{v}_1,\mathbf{v}_2;\nu)$ to be the moduli space of pairs
    \begin{align*}
        \lbrace(a,u)\mid a\in(0,b] \text{ and }u\colon\Sigma_a\to\mathscr{X}\rbrace
    \end{align*}
    where $u$ is pseudoholomorphic with respect to $J_{\Sigma_a}$, and which is asymptotic to $\nu\in\mathcal{H}(T_{\phi^d},\Theta_{\phi^d})$, and has tangency conditions at $(w_{in},w_a)$ specified by
    \begin{itemize}
        \item[$\bullet$] $u(z)\not\in\mathbf{D}$ unless $z = w_{in}$ or $z = w_a$, and;
        \item[$\bullet$] $u(w_{in})$ intersects $\mathbf{D}$ with multiplicity $\mathbf{v}_1\in\mathbb{Z}^{S_0}_{\geq0}$ and $u(w_a)$ intersect $\mathbf{D}$ with multiplicity $\mathbf{v}_2\in\mathbb{Z}^{S_0}_{\geq0}$.
    \end{itemize}
    Furthermore, we only consider solutions $u$ such that $u\in[-F_{\nu}]$ as capping disks.
\end{definition}
By our restrictions on $\mathbf{v}_1,\mathbf{v}_2\in\mathbb{Z}^{S_0}_{\geq0}$, the moduli spaces $\mathcal{M}(\mathbf{v}_1,\mathbf{v}_2;\nu)$ have virtual dimension $\mu(\nu)+1$. We can define the moduli spaces
\[ \mathcal{M}_b(\mathbf{v}_1,\mathbf{v}_2;\nu)\]
which are the restrictions of the moduli spaces above to domains $\Sigma_b$, where $a = b$. When $\mu(\nu)\leq 1$, these moduli spaces are smooth manifolds of the expected dimension for generic choices of $J_{\Sigma_a}$, and they are also compact when $\mu(\nu) = 0$.
\begin{lemma}\label{lemma:pss-low-energy-ring-1}
    The operation
    \begin{align*}
        \SR(\mathcal{X}^{\circ})\otimes_{\Bbbk}\SR(\mathcal{X}^{\circ}) &\to \bigoplus_{d=0}^{\infty} HF^0_{low}(X_t,\phi^d;\Bbbk)
    \end{align*}
    defined at the chain level by
    \begin{align}\label{planewithtwopointsoperation}
        \theta^{\mathbf{v}_1}\otimes\theta^{\mathbf{v}_2} &\mapsto \sum_{\nu\in\mathcal{H}(T_{\phi^d},\Theta_{\phi^d})}\sum_{u\in\mathcal{M}_b(\mathbf{v}_1,\mathbf{v}_2;\nu)}\sigma_u
    \end{align}
    coincides with
    \begin{align*}
            \PSS^{\log}_{low}(\theta^{\mathbf{v}_1})\star_{low}\PSS^{\log}_{low}(\theta^{\mathbf{v_2}}) \,.
    \end{align*}
\end{lemma}
\begin{proof}
    This is a standard TQFT argument proved the same way as~\cite[Lemma 3.26]{GP2}.
\end{proof}

Next, we will show that:
\begin{lemma}[cf. {\cite[Lemma 3.29]{GP2}}]\label{lemma:pss-low-energy-ring-2}
    The operation~\eqref{planewithtwopointsoperation} agrees with
\begin{align}\label{PSSfollowedbySR}
    \PSS^{\log}_{low}(\theta^{\mathbf{v}_1}\star_{\SR}\theta^{\mathbf{v}_2})
\end{align}
cohomologically, where $\star_{low}$ denotes the low energy pair of pants product.
\end{lemma}
\begin{proof}
    We consider the $1$-dimensional moduli spaces $\mathcal{M}(\mathbf{v}_1,\mathbf{v}_2;\nu)$, with $\mu(\nu) = 0$, as smooth manifolds fibered over $(0,b]$, corresponding to a degeneration of the domains as $a\to 0$. We can complete this moduli space over $a\to 0$, as $w_a = -1/a$ approaches $w_{in} = \infty$, by gluing in a constant sphere bubble~\cite[Ch. 10]{MS2}. Thus by counting elements in these boundary strata, we see that~\eqref{PSSfollowedbySR} coincides with~\eqref{planewithtwopointsoperation} in cohomology.
\end{proof}

\subsection{The log SSP map}
Following~\cite{PSS}, we will construct a (one-sided) inverse to the (low energy) log PSS map. This is also achieved by counting pseudoholomorphic planes with a tangency condition. In this section, let $\Sigma^{\vee}$ denote $\mathbb{C}$, and consider the family $\mathscr{X} = \mathscr{X}_{\epsilon'}\to\Sigma^{\vee}$. Similarly to $\Sigma$, the domain $\Sigma^{\vee}$ comes with a distinguished marked point $w_{out} = 0$ as well as a \textit{positive} cylindrical end
\begin{align}
    \epsilon_{out}\colon \mathbb{R}\times S^1\to\Sigma^{\vee} = \mathbb{R}\times S^1\cup\lbrace 0\rbrace\subset\mathbb{CP}^1
\end{align}
defined in the obvious way (cf.~\eqref{logPSScylindricalend}.) We can also define a space $J_{\Sigma^{\vee}}(\mathscr{X},\mathbf{D})$ of domain-dependent almost complex structures on $\Sigma^{\vee}$ analogously to Definition~\ref{logPSSdomaindependentacs}.

\begin{definition}
    Let $\tilde{\nu} = (\nu,[-F_{\nu}])$ be a capped orbit, with $\nu\in\mathcal{H}(T_{\phi^d},\Theta_{\phi^d})$, such that $[-F_{\nu}]\cdot\mathbf{D} = \mathbf{v}\in\mathbb{Z}_{\geq0}^{S_0}$. Then we define the \textit{low energy} log SSP moduli space
    \begin{align*}
        \mathcal{M}^{\SSP}(\tilde{\nu},\mathbf{v})
    \end{align*}
    to be the moduli space of $J_{\Sigma^{\vee}}$-holomorphic maps, for given $J_{\Sigma^{\vee}}$, which are asymptotic to $\nu$ at the positive end and such that
    \begin{itemize}
        \item[$\bullet$] $u(z)\not\in\mathbf{D}$ for $z\neq w_{out}$, and;
        \item[$\bullet$] the intersection multiplicity of $u$ with $D_i$ at $w_{in}$ is $v_i$.
    \end{itemize}
\end{definition}
Again, by the Riemann--Hurwitz theorem, all log SSP solutions are $d$-fold multisections of $\mathscr{X}\to\Sigma^{\vee}$ with a single branch point at $w_{out}$. Using the standard conventions for the Maslov index, these moduli spaces have virtual dimension
\begin{align*}
    \operatorname{vdim}\mathcal{M}^{\SSP}(\tilde{\nu},\mathbf{v}) = 2n+2-\mu(\nu)
\end{align*}
where we recall that $\mathscr{X}$ is of complex dimension $n+1$ and that $w_{out}$ is constrained to lie in a stratum $D_I$ of $\mathbf{D}$ determined by $\mathbf{v}$. If we restrict to $\nu$ such that $\mu(\nu) = 0$, then the moduli spaces $\mathcal{M}^{\SSP}(\tilde{\nu},\mathbf{v})$ are compact smooth manifolds of dimension $2n+2$ for generic $J_{\Sigma^{\vee}}$.\footnote{We note that strip-breaking is ruled out for degree reasons and sphere bubbles are ruled out by our restriction to SSP solutions equivalent to fiber disks.} There is an evaluation map
\begin{align*}
    \ev_{out}\colon\mathcal{M}_{\SSP}(\tilde{\nu},\mathbf{v}) &\to D_I \\
    [u,w_{out}]&\mapsto u(w_{out})
\end{align*}

\begin{definition}
    The \textit{low energy log SSP map}
    \begin{align*}
        \SSP^{\log}_{low}\colon\bigoplus_{d=0}^{\infty} HF^0(X_t,\phi^d)\to\SR(\mathcal{X}^{\circ})
    \end{align*}
    is defined by
    \begin{align*}
        \SSP^{\log}_{low}([\nu,F_{\nu}])\coloneqq\sum_{\mathbf{v}\in S_0}PD((\ev_{out})_*[\mathcal{M}(\tilde{\nu},\mathbf{v})])
    \end{align*}
    where $PD$ denotes the Poincar{\'e} duality isomorphism.
\end{definition}

We will show that $\SSP^{\log}_{low}$ is a left inverse for $\PSS^{\log}_{low}$ using an argument reminiscent of~\cite{PSS}. To that end, we will need a confinement argument for low energy log PSS and log SSP solutions, and we will need to introduce suitable moduli spaces of pseudoholomorphic spheres. The first of these is achieved in the following lemma.
\begin{lemma}\label{confinementlemma}
    For any low energy log PSS solution $u\in\mathcal{M}(\mathbf{v},[\nu,-F_{\nu}])$, the image of $u$ is contained in $U_I$, where $I$ is the support of $\mathbf{v}\in\mathbb{Z}^{S_0}_{\geq0}$. The same is true for all low energy log SSP solutions as well.
\end{lemma}
\begin{proof}
    We will only spell this out for log PSS solutions, as the other case is proved identically. The proof is based on an energy estimate for log PSS solutions combined with the well-known monotonicity lemma. Note that the by Lemma~\ref{divisorconstruction0} and Stokes' theorem, the symplectic area of $u\colon\Sigma\to\mathscr{X}$ for $u\in\mathcal{M}(\mathbf{v},[\nu,-F_{\nu}])$ is determined by the integral of a (locally defined) primitive of $\Omega$ along the boundary of $u$ at infinity. From this observation it follows that the symplectic area depends on the constant $\epsilon'>0$ chosen when completing $\mathscr{X}\mid_{\mathbb{D}_{\epsilon}}$ (see Assumption~\ref{completedfamilyassumption}.) Moreover, the symplectic area of any low energy log PSS solution approaches $0$ as $\epsilon'\to0$.

    Suppose to a contradiction that there exists some point $p\in\mathscr{X}_{\epsilon'}$ and some low energy log PSS solution $u\colon\Sigma\to\mathscr{X}_{\epsilon'}$ whose image contains $p$. Such a point would lie in some region $U_K$, where $K\subset R$ corresponds to a stratum of $\mathbf{D}$ other than $I$. Furthermore, the portion of $u$ lying in $U_K$ is the image of a pseudoholomorphic disk, and the composition of this disk with $\pi_K\colon U_K\to D_K$ is a holomorphic disk with image in $D_K$. It follows that we can find a ball $B_K(\pi_K(p))$ in $U_K$ centered at $\pi_K(p)$ of some positive radius $r_K'>0$, independent of the particular point $p$, which is disjoint from $U_I$. Monotonicity implies that the area of this disk is bounded below by some expression $C_Kr_K^2$, where $r_K$ is smaller than both $r_K'$ and the injectivity radius of $D_K$ (see~\cite[Lemma 4.8]{GP2} and the references therein.) Because almost complex structures which are standard at infinity are also split in the sense of~\cite[Definition 4.9]{GP2}, we can use the same argument as in~\cite[Lemma 4.10]{GP2} to show that the total symplectic area of $u$ is bounded below by a constant multiple of the symplectic area of $\pi_K(u)$. In particular we can make $\pi_K(u)$ smaller than the \textit{a priori} bound $C_Kr_K^2$, so no such $p$ can exist.
\end{proof}

With this understood, we return to the proof that $\PSS^{\log}_{low}$ is an injection. From the families $\mathscr{X}_{\epsilon'}\to\Sigma$ and $\mathscr{X}_{\epsilon'}\to\Sigma^{\vee}$, we can form a family
\begin{align}\label{degenerationoverp1}
    \mathscr{X}_{\mathbb{CP}^1}\to\mathbb{CP}^1
\end{align}
by a clutching construction. The regularizing neighborhoods $U_I$ for each $I\subset S_0$ can be glued together in $\mathscr{X}_{\mathbb{CP}^1}$, from which we obtain open subsets of the family over $\mathbb{CP}^1$ characterized as follows
\begin{lemma}
    For each $I\subset S_0$, there is an open subset $PD_I\subset\mathscr{X}_{\mathbb{CP}^1}$ admitting a natural identification
    \begin{align}\label{projectivizednbhd}
        PD_I\cong\mathbb{P}(N_{\mathscr{X}}D_I\oplus\mathcal{O}_{D_I}) \,.
    \end{align}
\end{lemma}
There are two natural sections of $\mathbb{P}(N_{\mathscr{X}}D_I\oplus\mathcal{O}_{D_I})$ which correspond to the copies of $D_I$ lying over the points $0,\infty\in\mathbb{CP}^1$. We observe that~\eqref{degenerationoverp1} also has sections (smooth) sections. One can obtain such a section by gluing together two sections of $\mathscr{X}_{\epsilon'}$ in the relative homology class of a fiber disk. Let $F_{\mathbb{CP}^1}\in H_2(\mathscr{X}_{\mathbb{CP}^1})$ denote the resulting homology class of sections. Note that this is homology class coming from a homology class in one of the \textit{fibers} of~\eqref{projectivizednbhd}. The proof of the following result is based on counting pseudoholomorphic spheres in $\mathscr{X}_{\mathbb{CP}^1}$ contained in the class $F_{\mathbb{P}}^1$ of~\eqref{projectivizednbhd}.

\begin{proposition}\label{lowenergylogPSSinjection}
    We have that
    \begin{align}\label{SSP-PSS-composition}
        \SSP^{\log}_{low}\circ\PSS^{\log}_{low}(\theta^{\mathbf{v}}) = \theta^{\mathbf{v}} \,.
    \end{align}
    Consequently $\PSS^{\log}_{low}$ is an injection.
\end{proposition}
\begin{proof}
    Let $C$ denote the $2$-sphere with two distinguished marked points $w_{in}$ and $w_{out}$ at $0$ and $\infty$. Choose a domain-dependent almost complex structure $J_C$ on $C$ which, near the marked points, is constant with value in $AK(\mathscr{X},\mathbf{D})$. The space of almost complex structures from which we choose $J_C$ is defined completely analogously to Definition~\ref{logPSSdomaindependentacs}, except that we do not require the stabilizing divisor $\mathbf{E}$ to be an almost complex manifold.
    
    Consider the moduli space $\mathcal{M}_2(\mathbf{v}_0;dF_{\mathbb{CP}^1})$ of $J_C$-holomorphic spheres in the class $dF_{\mathbb{CP}^1}$, such that $w_{out}$ is tangent to the copy of $\mathbf{D}$ over $0\in\mathbb{CP}^1$ and $w_{in}$ is tangent to the copy of $\mathbf{D}$ over $\infty\in\mathbb{CP}^1$, with the same order of tangency $\mathbf{v}_0$ at both marked points. By choosing $J_C$ generically, the usual arguments~\cite[\S{3.4}]{MS2} show that these moduli spaces are smooth manifolds of the expected virtual dimension. From here a standard gluing argument (see e.g.~\cite[Ch. 10]{MS2} or~\cite[Lemma 3.26]{GP2}) and the confinement result (Lemma~\ref{confinementlemma}) above shows that by gluing in a cylinder interpolating between low energy log PSS and log SSP solutions with the same asymptotic condition, we can identify the left hand side of~\eqref{SSP-PSS-composition} with the natural map $\SR(\mathcal{X}^{\circ})\to\SR(\mathcal{X}^{\circ})$ determined by counting elements of $\mathcal{M}_2(\mathbf{v}_0;dF_{\mathbb{CP}^1})$. The virtual count of such sections is $1$, from which the result follows.
\end{proof}

\begin{lemma}\label{lowenergylogpssringiso}
    The low energy log-PSS map $\PSS^{\log}_{low}$ is a ring isomorphism.
\end{lemma}
\begin{proof}
    The injectivity of $\PSS^{\log}_{low}$ and the computations in Lemma~\ref{PSS-domain-codomain-abstract-iso} show that $\PSS^{\log}_{low}$ is an isomorphism of vector spaces, which can be further upgraded to an isomorphism of rings combining Lemma~\ref{lemma:pss-low-energy-ring-1} and Lemma~\ref{lemma:pss-low-energy-ring-2}. 
\end{proof}
To upgrade this to a proof of Theorem~\ref{stanleyreisnerisom}, we also need the following result, since the images of the low energy log-PSS map do not themselves define classes in fixed point Floer cohomology (instead they are the low energy terms of the log-PSS classes.)
\begin{proposition}\label{PSS-classes-survive}
    There is a natural isomorphism of $\Lambda$-modules
    \begin{align}\label{lowenergymoduleiso}
        HF^0_{low}(X_t,\phi^d;\Bbbk)\otimes\Lambda \cong HF^0(X_t,\phi^d)
    \end{align}
    for all $d\geq0$.
\end{proposition}
\begin{proof}
    First assume that $\operatorname{char}\Bbbk = 0$. The log PSS cycles $\PSS^{\log}(\theta^{\mathbf{v}})$, which are non-exact for degree reasons, determine a canonical class in $HF^0(X_t,\phi^d)$ for each generator of $\SR(\mathcal{X}^{\circ})$. Proposition~\ref{lowenergylogPSSinjection} shows that each of these classes is nonzero, and this together with the energy spectral sequence proves the isomorphism.

    The argument above implies that the differential on the fixed-point Floer cochain complex vanishes in degree $0$ over a field of characteristic $0$. We can use this to prove the desired isomorphism over also fields similarly to~\cite[Corollary 4.22]{Pomerleano}. Since enhanced fixed point Floer cochain complex is supported in nonnegative degrees, and since the differential vanishes over $\mathbb{Q}[\![T]\!]$, it follows that the enhanced fixed point Floer differential also vanishes over $\mathbb{Z}[\![T]\!]$ as well. This implies the isomorphism~\eqref{lowenergymoduleiso} over any field $\Bbbk[\![T]\!]$.
\end{proof}

\begin{proof}[Proof of Theorem~\ref{stanleyreisnerisom}]
    The isomorphism~\eqref{stanleyreisnerisom} follows from Lemma~\ref{lowenergylogpssringiso} and Proposition~\ref{PSS-classes-survive}.
\end{proof}

\section{The relative Fukaya category and the twisted closed-open map}\label{relfuksection}
Continuing in the setup of \S{\ref{fixedpointfloersection}}, we will construct the twisted closed-open map by counting pseudoholomorphic sections of symplectic fiber bundles with Lagrangian boundary conditions in the sense of~\cite[p. 235]{SeiBook}. A similar construction for Riemann surfaces appeared in~\cite{JYZ}, and we will explain how to modify this construction in higher dimensions for the relative Fukaya category. A new technical issue in this setting is that we need to add marked points to our curves which are mapped to a system of divisors in the total space $\mathcal{X}^{\circ}$ to stabilize the domains of curves used to construct the relative Fukaya category, notably including the domains of disk bubbles. On that note, in order to incorporate bounding cochains in the definition of the twisted closed-open map, we must also contend with the fact that the construction of the relative Fukaya category of $X$, as written in~\cite{PerutzSheridan}, depends on a \textit{fixed} choice of compatible almost complex structure.\footnote{As stated in~\cite[\S{1.1}]{PerutzSheridan}, it is expected that the relative Fukaya category is invariant under deformations of the almost complex structure.} This means that when considering disk bubbles, we will need to keep track of how the almost complex structures on the fibers vary.

\subsection{Preliminaries on the relative Fukaya category} We will briefly recall some aspects of the relative Fukaya category following~\cite{PerutzSheridan}, explaining how our discussion fits into their framework. We do \textbf{not} attempt to give a self-contained overview of~\cite{PerutzSheridan}, but we will mention the choices involved in the construction of the relative Fukaya category which need to be accounted for in the construction of~\eqref{twistedclosedopenfirstmention}.

Recall from \S{\ref{mainsection}} that all closed symplectic manifolds $X$ we consider belong to a smooth family $\mathcal{X}^{\circ}\to\mathbb{D}^{\circ}$ of Calabi--Yau $n$-folds, equipped with a smooth family of integrable almost complex structures and a relatively ample line bundle $\mathcal{E}$, which determines a family of K{\"a}hler forms. Let $J_0$ denote the (integrable) almost complex structure on $X$. We take this to be the \textit{fixed} almost complex structure used to construct the relative Fukaya category. The relative Fukaya category is defined using a collection of $J_0$-holomorphic divisors $\mathbf{E}_t = \lbrace E_{q,t}\rbrace_{q\in Q}$, where $Q$ is an index set, called a \textit{system of divisors}  (see~\cite[Definition 1.1]{PerutzSheridan} for a precise definition.) Each of these divisors lies in the complement of a Liouville subdomain $W\subset X$. To simplify our exposition, we can assume that these divisors are obtained by restricting the stabilizing divisors used above in the construction of the log-PSS map (see Lemma~\ref{divisorconstruction-1}) but this is not essential for the proof.

We need to prove the existence of a suitable system of divisors (\cite[Definition 1.1]{PerutzSheridan}) for our purposes. As mentioned in \S{\ref{mainsection}}, these divisors will be constructed as small perturbations of a fixed holomorphic divisor $E_t\subset X_t$.

\begin{lemma}\label{divisorconstruction1}
Suppose that $L\subset X_t$ is a Lagrangian rational homology sphere. Then there is a smooth divisor $E_t\subset X_t$ which is disjoint from $L$ and such that $[E_t]\in H_{2n-2}(X_t;\mathbb{Z})$ is Poincar{\'e} dual to $N[\omega]$, for some positive integer $N$. There is also a system of divisors $\mathbf{E}_t\coloneqq\bigcup_{q\in Q} E_{q,t}$, where $Q$ is a finite index set, contained in an open neighborhood of $E_t$. Moreover, all of the components $E_{q,t}$ are preserved by $\phi\colon X_t\to X_t$.
\end{lemma}
\begin{proof}
By~\cite[Theorem 0.2]{guedj-divisor} or \cite[Theorem 2, Remark 1(b)]{AGM-divisor}, the line bundle $\mathcal{E}_t^{\otimes N}$ on $X_t$ contains a holomorphic section $s_t\in\Gamma(\mathcal{E}_t)$ such that $E_t\coloneqq s^{-1}(0)$ is a smooth holomorphic submanifold of $X_t$ disjoint from $L$. In particular, $E_t$ has the desired homology class. By appealing to~\cite[Lemma 3.8]{SheridanVersality}, we can choose a family of sections $s_{t,q}\in\Gamma(\mathcal{E}_t)$, for $q\in Q$ whose zero loci $V_{t,q}\coloneqq s_{t,q}^{-1}(0)$ form a system of divisors on $X_t$. Since these divisors are constructed as small perturbations of $s_t$, using Bertini's theorem, we can take the divisors $E_{q,t}$ to be disjoint from $L$.

We can extend each of these sections to sections $s_q\in\Gamma(\mathscr{E})$, where $\mathscr{X}$ denotes an snc model for $\mathcal{X}^{\circ}$, and consider the horizontal divisors $E_q\coloneqq s_q^{-1}(0)\subset\mathscr{X}$. Recall that $\phi$ is obtained from the flow of a Hamiltonian vector field on $\mathscr{X}$, and by Lemma~\ref{divisorconstruction1}, we can isotope isotope this vector field, hence modifying $\phi$ by an isotopy, so that it fixes the divisors $E_q$. From this, it follows that $\phi$ fixes the divisors $E_{q,t}$.
\end{proof}
With this choice of system of divisors, $\phi$ acts on the relative Fukaya category $\Fuk(X_t,\mathbf{E}_t)$, and the Lagrangian submanifolds $\phi^d(L)$, for all $d\in\mathbb{Z}$, become objects of the relative Fukaya category. More generally, the objects of the relative Fukaya category are supported on embedded exact Lagrangian submanifolds of $W\subset X$ equipped with spin structures (and orientations.) To obtain a $\mathbb{Z}$-graded category, we will also assume that all Lagrangian submanifolds under consideration have vanishing Maslov class and are equipped with a grading~\cite{SeiGraded}.

The construction of the relative Fukaya category follows the general method of~\cite{SeiBook}, combined with the use of stabilizing divisors as in~\cite{CieliebakMohnke} to achieve transversality. One begins by considering a compactified Deligne--Mumford moduli space $\overline{\mathcal{R}}_{k,\ell}$, for $k+2\ell\geq2$, of stable disks equipped with $k+1$ boundary marked points $\zeta_0,\ldots,\zeta_k$ (in order) and $\ell$ interior marked points $z_1,\ldots,z_{\ell}$, which we will also call \textit{stabilizing marked points}. This a smooth manifold with corners that comes with a universal family $\overline{\mathcal{S}}_{k,\ell} \to \overline{\mathcal{R}}_{k,\ell}$. For any point $r\in\overline{\mathcal{R}}_{k,\ell}$, let $\Sigma_r$ denote the corresponding fiber of the universal family, and let $\Sigma_r^{\circ}$ denote the complement of the boundary marked points and boundary nodes.

A parametrized cylindrical end for one of the stabilizing marked points is a holomorphic embedding
\begin{align}\label{cylindricalenddef}
\epsilon_i\colon\mathbb{D}\to\Sigma^{\circ}_r
\end{align}
such that $\epsilon_i(0) = z_i$. We will consider parametrized cylindrical modulo the natural $S^1$-action on the disk. A strip-like end for $\zeta_{i}$ is a proper holomorphic embedding
\begin{align}\label{striplikeenddef}
\epsilon_{i}\colon Z^{\pm}\to\Sigma^{\circ}_r
\end{align}
where $Z^+\coloneqq\mathbb{R}_{\pm}\times[0,1]$. The domain of a strip-like end is $Z^{-}$ if $i = 0$, and $Z^{+}$ otherwise. We also require that $\epsilon_{i}$ sends the boundary intervals $\mathbb{R}_{\pm}\times\lbrace 0\rbrace$ and $\mathbb{R}_{\pm}\times\lbrace 1\rbrace$ to $\partial\Sigma^{\circ}$, and that it limits to $\zeta_{i}$ as $s$ approaches $\pm\infty$, depending on whether or not $i = 0$, where $s$ is the coordinate on $Z^{\pm}$. Following~\cite[(9g)]{SeiBook}, one can make a \textit{consistent, universal} choice of strip-like and cylindrical ends for $\Sigma_r^{\circ}$. The symmetric group $S_{\ell}$ acts $\overline{\mathcal{S}}_{k,\ell}$ by permuting interior marked points, and we require the choice of cylindrical ends to be $S_{\ell}$-equivariant.

For each pair of exact (Spin, Maslov zero) Lagrangian submanifolds $(L_0,L_1)$ in the Liouville domain $W$, we make a choice of Floer data which is used to write Floer's equation for strips with boundary on $(L_0,L_1)$. Let $\mathcal{H}(X_t)\subset C^{\infty}(X_t,\mathbb{R})$ denote the subspace of smooth functions supported on $W$. Choose, for each $(L_0,L_1)$, a smooth function $H_{01}\colon[0,1]\to\mathcal{H}(X_t)$ such that the time-$1$ Hamiltonian flow $X_{H_{01}(t)}$, when applied to $L_0$, makes it transverse to $L_1$. The set of Hamiltonian chords $\mathcal{C}(L_0,L_1)$ from $L_0$ to $L_1$ correspond to intersection points between this perturbation of $L_0$ and $L_1$, and so it is finite.

The other component of a Floer datum is a smooth function on the interval valued in a certain space of smooth sections of $\End TX$. One defines~\cite[\S{3.3}]{PerutzSheridan}  the spaces
\begin{align*}
\mathcal{Y}(X_t,\mathbf{E}_t)&\coloneqq\lbrace Y\in C^{\infty}(\End TX_t)\mid J_0Y+YJ_0 = 0 \text{ and }Y(TE_{q,t})\subset TE_{q,t}\text{ for all }q\in Q\rbrace \\
\mathcal{Y}_*(X_t,\mathbf{E}_t)&\coloneqq\left\lbrace Y\in\mathcal{Y}(X_t,\mathbf{E}_t)\mid |\!|Y|\!|_{C^0}<\log\frac{3}{2}\right\rbrace \\
\mathcal{Y}_*^{\max}(X_t,\mathbf{E}_t)&\coloneqq\lbrace Y\in\mathcal{Y}_*(X_t,\mathbf{E}_t)\mid\operatorname{supp}(Y)\subset W\rbrace \,.
\end{align*}
We assign to each pair $(L_0,L_1)$ a smooth function $Y_{01}\colon[0,1]\to\mathcal{Y}_*^{\max}(X_t,\mathbf{E}_t)$, which gives rise to a path of almost complex structures via $J_{01}\coloneqq J_0\exp(Y_{01})$.

The Floer equation for Floer trajectores between chords $y_0,y_1\in\mathcal{C}(L_0,L_1)$ for strips $u\colon\mathbb{R}\times[0,1]\to W$ with boundary conditions $u(s,i)\in L_i$ for $i = 0,1$ and asymptotic conditions
\begin{align*}
    \lim_{s\to-\infty}u(s,\cdot) = y_1(\cdot) \\
    \lim_{s\to\infty}u(s,\cdot) = y_0(\cdot)
\end{align*}
is
\begin{align*}
    \partial_s u + J_{01}(t)(\partial_t u-X_{H_{01}(t)}\circ u) = 0 \,.
\end{align*}
For fixed choices of Hamiltonians $H_{01}$ associated to $(L_0,L_1)$, there is a comeager subset of choices of $Y_{01}$ for which the moduli space of Floer trajectories is regular.

For any (punctured) stable curve $\Sigma_r^{\circ}$ equipped with a tuple $\mathbf{L}$ of Lagrangian labels $L_0,\ldots,L_k$ (with $k\geq0$) for its boundary components, we need to choose a perturbation datum. Let $\tilde{\Sigma}_r^{\circ}$ denote the normalization of $\Sigma_r^{\circ}$. A \textit{perturbation datum} $P = (Y,K)$ consists of a pair
\begin{align}\label{perturbationdatum}
Y &\in C^{\infty}(\tilde{\Sigma}_r^{\circ},\mathcal{Y}_*(X_t,\mathbf{E}_t)) \,, \\ 
K &\in\Omega^1(\tilde{\Sigma}_r^{\circ},\mathcal{H})
\end{align}
subject to the constraints given in~\cite[Definition 5.4]{PerutzSheridan}. By~\cite[Lemma 5.6]{PerutzSheridan}, one can inductively make consistent universal choices of perturbation data (see Definition 5.5 of~\cite{PerutzSheridan}) for arbitrary integers $N = k+2\ell\geq2$. For fixed $k,\ell\in\mathbb{Z}$, let $\mathbf{P}$ be a choice of universal perturbation data for $\overline{\mathcal{S}}_{k,\ell}$.

Suppose we are given a Lagrangian labeling $\mathbf{L} = (L_0,\ldots,L_k)$ and Hamiltonian chords $\mathbf{y} = (y_0,\ldots,y_k)$ with $y_0\in\mathcal{C}(L_0,L_k)$ and $y_i\in\mathcal{C}(L_{i-1},L_i)$ for $i = 1,\ldots,k$. We will consider maps $u\colon\Sigma_r^{\circ}\to X_t$, with boundary conditions given by $\mathbf{L}$ and asymptotic conditions $\mathbf{y}$. Denote by $\pi_2(\mathbf{y})$ the set of homotopy classes of such maps. To each $A\in\pi_2(\mathbf{y})$, we can associate intersection numbers $A\cdot E_q$ for each $q\in Q$. For any map $u\colon\Sigma^{\circ}_r\to X_t$, we can record its intersection numbers with the divisors $E_q$ at a point $z$ in the domain of $u$ with a \textit{tangency vector} $\iota(u,z)\in \mathbb{Z}_{\geq0}^Q$.

We only need to consider classes $A$ for which all of these intersection numbers are nonnegative (cf.~\cite[Lemma 3.4]{PerutzSheridan}.) Given $A\in\pi_2(\mathbf{y})$ as such, a choice of \textit{tangency data} is a number $\ell\geq0$ together with 
\begin{align*}
    \mathbf{t}\colon\lbrace 1,\ldots,\ell\rbrace\to \mathbb{Z}_{\geq0}^Q
\end{align*}
such that $\sum_{i=1}^{\ell}\mathbf{t}(i)_q = A\cdot E_q$.

We can now pose the perturbed pseudoholomorphic curve equations for the domains $\Sigma^{\circ}_r$ equipped with the choices of perturbation data as above. Given a perturbation datum $P = (Y,K)$, which we emphasize is supported away from the system of divisors, there is an associated domain-dependent almost complex structure $J_z$ on $\Sigma_r^{\circ}$ determined by $Y$, and an inhomogeneous $1$-form $\eta$ on $\Sigma_r^{\circ}$ valued in Hamiltonian vector fields on $X_t$. Let $j$ denote the complex structure on $\Sigma_r^{\circ}$ and let $C_0,\ldots,C_k$ denote its boundary components.
\begin{definition}
Let $\mathcal{M}(\mathbf{y},A,\mathbf{t},\mathbf{P})$ denote the moduli space of stable disks $u\colon\Sigma_r^{\circ}\to X$ satisfying the domain-dependent Cauchy--Riemann equations
\begin{align*}
(Du-\eta)\circ j = J_z\circ(Du-\eta)
\end{align*}
(where $J_z$ is the value of the domain-dependent almost complex structure at $z\in\Sigma_r^{\circ}$) with boundary and asymptotic conditions
\begin{align*}
u(C_i)&\subset L_i \\
\lim_{s\to\pm\infty} u(\epsilon_i(s,\cdot)) &= y_i(\cdot)
\end{align*}
and such that $[u] = A\in\pi_2(\mathbf{y})$ and $\iota(u,z_i) = \mathbf{t}(i)$
for all $i = 1,\ldots,\ell$.
\end{definition}
For more explanation of the tangency conditions, see~\cite[Lemma 5.7]{PerutzSheridan}. We remind the reader that there are no sphere bubbles fully contained in any component of $\mathbf{E}_t$ (cf.~\cite[Definition 1.1]{PerutzSheridan} and~\cite[Definition 4.4]{Pomerleano}.)

If we set $\ell = A\cdot\mathbf{E}_t$ and let $\mathbf{q}\colon\lbrace 1,\ldots,\ell\rbrace\to Q$ be a function with $|\mathbf{q}^{-1}(q)| = A\cdot E_{q,t}$, we can define a \textit{canonical tangency datum} $\mathbf{t}^{\mathrm{can}}$ by setting $\mathbf{t}^{\mathrm{can}}(i)_q = 1$ if $\mathbf{q}(i) = q$ and $\mathbf{t}^{\mathrm{can}}(i)_q = 0$ otherwise. Set
\begin{align}\label{perutzsheridanmodulispaces}
\mathcal{M}(\mathbf{y},A,\mathbf{P}) = \mathcal{M}(\mathbf{y},A,\mathbf{t}^{\mathrm{can}},\mathbf{P}) \,.
\end{align}
This moduli space is independent of the choice of $\mathbf{q}$ used to define $\mathbf{t}^{\mathrm{can}}$, up to the action the symmetric group $S_{\ell}$. The subgroup $S_{\mathbf{q}}\subset S_{\ell}$ of permutations preserving $\mathbf{q}$ acts freely on $\mathcal{M}(\mathbf{y},A,\mathbf{P})$.

For generic choices of perturbation data, the Gromov compactifications of the moduli spaces~\eqref{perutzsheridanmodulispaces} of virtual dimension at most $1$ are regular, and they do not contain any curves with sphere bubbles attached~\cite[Corollary 5.10 and Lemma 5.16]{PerutzSheridan}. Thus they are compact oriented manifolds with boundary of the expected dimension. As usual, the only moduli spaces relevant to the definition of the relative Fukaya category are those of virtual dimension $\leq 1$. 

To each Hamolitonian chord $y \in \mathcal{C}(L_0,L_1)$, there is an associated orientation line $o_y$, which is a $1$-dimensional vector space over the Novikov field $\Lambda$. The Floer cochain spaces are defined be
\begin{align*}
CF^*(L_0,L_1)\coloneqq\bigoplus_{y\in\mathcal{C}(L_0,L_1)}o_y \,.
\end{align*}
Since $X$ is Calabi--Yau and $L_i$ both have vanishing Maslov class, these are $\mathbb{Z}$-graded vector spaces. Let $\lambda(A)$ denote the symplectic area, with respect to $\omega$, of $A\in\pi_2(\mathbf{y})$. The $A_{\infty}$ structure maps are given by
\begin{align}
&\mathfrak{m}_k\colon CF^*(L_0,L_1)\widehat\otimes_{\Lambda}\cdots\widehat\otimes_{\Lambda}CF^*(L_{k-1},L_k)\to CF^*(L_0,L_k)[2-k],\\
&\mathfrak{m}_k\coloneqq\sum_{u\in\mathcal{M}(\mathbf{y},A,\mathbf{P})/S_{\mathbf{q}}} T^{\lambda(A)}\sigma_u
\end{align}
where $\widehat\otimes_{\Lambda}$ denotes the completed tensor product, each $u\in\mathcal{M}(\mathbf{y},A,\mathbf{P})$ defines an isomorphism
\begin{align*}
\sigma_u\colon o_{y_1}\otimes\cdots\otimes o_{y_k}\to o_{y_0}
\end{align*}
and $\mathbf{y}$ and $A$ are chosen so that $\mathcal{M}(\mathbf{y},A,\mathbf{P})$ has virtual dimension $0$~\cite[Lemma 5.11]{PerutzSheridan}. The $A_{\infty}$-algebra structure on $CF^*(L)$ determined by these structure maps will \textit{a priori} have a nontrivial curvature term $\mathfrak{m}_0$, which we must correct for by choosing a bounding cochain for this $A_{\infty}$-algebra, called a bounding cochain of $L$ for short, when one exists~\cite[Definition 2.5]{PerutzSheridan}. We declare that the objects of the relative Fukaya category $\Fuk(X_t,\mathbf{E}_t)$ are exact Lagrangian submanifolds $L$ of $W\subset X_t$ with vanishing Maslov class equipped with Spin structures and bounding cochains. The hom-spaces are the Floer cochain spaces $CF^*(L_0,L_1)$, and the $A_{\infty}$ structure maps come from deforming the maps given above by bounding cochains (the formula for this deformation is given in e.g.~\cite[Equation (3.20)]{SeiBook} or~\cite[Definition 3.6.6]{FOOOI}.)

\subsection{Pseudholomorphic sections with Lagrangian boundary} Recall that in Assumption~\ref{completedfamilyassumption} we explained how to replace $\mathcal{X}^{\circ}$ with an isomorphic symplectic fiber bundle $E_{\phi}$ over the cylinder $C$ of the type considered in \S{\ref{fixedpointfloersection}}. Fix an almost complex structure $\mathbf{J}$ as in Definition~\ref{fibrationsacss} on $\mathcal{X}^{\circ}$ which restricts to $J_0$ on the fiber $X_t$.

Recall that we constructed a collection of stabilizing divisors $\mathbf{E} = \lbrace E_q\rbrace_{q\in Q}$ in an snc model $\mathscr{X}$ for $\mathcal{X}^{\circ}$ in Lemmata~\ref{divisorconstruction-1} and~\ref{divisorconstruction0}, from which we obtained the system of divisors $\mathbf{E}_t$ in $X_t$ used to define the relative Fukaya category in Lemma~\ref{divisorconstruction1}. In the other direction, the system of divisors $\mathbf{E}_t$ gives rise to systems of divisors $\mathbf{E}^d$ in $E_{\phi^d}$ for all $d\geq0$ by a mapping torus construction. Similarly $\mathbf{J}$ lifts to an almost complex structure $\mathbf{J}^d$ on $E_{\phi^d}$. It is these systems of divisors that we shall use to construct the twisted closed open map(s).

Consider a disk $\Sigma$ equipped with $k+1$ boundary marked points $\zeta_0,\ldots,\zeta_k$ (in order), stabilizing interior marked points $z_1,\ldots,z_{\ell}$, and a distinguished \textit{input (interior) marked point} $w_{in}$. Let $\Sigma^{\circ}$ denote the complement in $\Sigma$ of the input interior marked points and the boundary marked points. Also choose a parametrized strip-like ends
\begin{align*}
\epsilon_{in}\colon\mathbb{D}^{\circ}\to\Sigma^{\circ}
\end{align*}
for each input marked point $w_{in}$, which we require to be a proper holomorphic embedding for which $\lim_{z\to 0}\epsilon_{in}(z) = w_{in}$. We can consider symplectic fibrations $E\to\Sigma^{\circ}$ with fiber $X_t$ over such Riemann surfaces. The most important examples of such fibrations for us can be obtained by restricting the bundles $E_{\phi^d}$ over $\mathbb{D}^{\circ}$ to $\Sigma^{\circ}$, but we will spell out the following intrinsic description. If we identify the domain $\Sigma$ with the closed unit disk in $\mathbb{C}$ in such a way that $w_{in}$ is identified with $0\in\mathbb{C}$ and $\zeta_0$ is identified with $i\in\mathbb{C}$, then we can cut the domain along the straight line segment from $0$ to $i$, and form the desired fibration $E_{\phi^d,\Sigma^{\circ}}\to\Sigma^{\circ}$ as a mapping torus with monodromy $\phi^d\colon X_t\to X_t$, where we can assume without loss of generality that $t$ is a point on the boundary of $\Sigma^{\circ}$ close to $\zeta_0$. To simplify notation, we simply let $E$ denote one of these fiber bundles when the particular choice of $d\geq0$ is not important. This gives us a fibration over $\Sigma\setminus\lbrace\zeta_0,w_{out}\rbrace$, which we can then restrict to one over $\Sigma^{\circ}$. We will also let $\Omega_{\phi^d}$ denote the (fiberwise) symplectic form on $E_{\phi^d,\Sigma^{\circ}}$. There is a more precise notion of a \textit{fibration with cylindrical and strip-like ends} following~\cite[\S{3}]{JYZ} and~\cite[(17b)]{SeiBook}, which is what we will actually use to count pseudoholomorphic sections.

\begin{definition}\label{fibrationwithstriplikeendsdefinition}
A fibration $E_{\phi^d,\Sigma^{\circ}}$ as above with cylindrical and strip-like ends comes with an isomorphism between $\epsilon_{in}^*E$ and $E_{\phi^{d}}\mid_{\mathbb{R}_{<0}\times S^1}$ over the cylindrical end near the input puncture $w_{in}$, and a trivialization of $E$ over each strip-like end. We can also assume that $E$ is trivialized over the cylindrical ends for the stabilizing interior marked points.
\end{definition}

Fix an object $(L,b)$ of the relative Fukaya category $\Fuk(X_t,\mathbf{E}_t)$, where $b$ is a bounding cochain for the Lagrangian $L$. In the setting of Definition~\ref{fibrationwithstriplikeendsdefinition}, there is a natural Lagrangian boundary condition associated to $L$, following~\cite[p. 235]{SeiBook}.
\begin{definition}\label{tcoboundaryconditions}
    A \textit{Lagrangian boundary condition} $F$ for a symplectic fibration $E\to\Sigma^{\circ}$ is a collection of Lagrangian submanifolds in the fibers over $\partial\Sigma\setminus\lbrace\zeta_{out}\rbrace$, which restricts to $L\subset X_t$ over $t$, carried to each other by boundary parallel transport. In particular $F\subset E$ is a Lagrangian submanifold in $E$.
\end{definition}
It is useful to note that by deleting a neighborhood of the stabilizing divisors $\mathbf{E}^d$ restricted to the total spaces of $E_{\phi^d,\Sigma^{\circ}}$, we obtain a symplectic fibration over $\Sigma^{\circ}$ whose fibers are Liouville domains. One can then interpret a Lagrangian boundary condition as in Definition~\ref{tcoboundaryconditions} as a special case of the definition in \textit{loc. cit.}

We would like to choose a \textit{relative} perturbation datum for $\Sigma_r^{\circ}$ in the sense of~\cite[p. 236]{SeiBook}, which must also satisfy suitable analogues of~\cite[Definition 5.4]{PerutzSheridan}. We can define a space $\mathcal{Y}_*(E_{\phi^d,\Sigma^{\circ}},\mathbf{E}^d)$ of smooth sections of $\End TE_{\phi^d,\Sigma^{\circ}}$, where $\mathbf{E}^d$ is the system of divisors. Choosing any suitable element $Y$ of this space gives rise to an almost complex structure $\mathbf{J}_Y\coloneqq \mathbf{J}_d\exp(Y)$ on $E$. We can assume that $Y$ is chosen so that $\mathbf{J}_Y\in\mathcal{J}(T_{\phi^d},\Theta_{\phi^d})$. The other component of a perturbation datum is a $1$-form $K$ on $E = E_{\phi^d,\Sigma^{\circ}}$ which vanishes identically on the vertical tangent bundle $TE^v = \ker(D\pi)$ and for which $K|_F\equiv0\in\Omega^1(F)$. We will also require $K$ to vanish near all of the interior marked points. Finally, if $\ell = 0$, then $Y$ should lie in $\mathcal{Y}_*^{\max}(E_{\phi^d,\Sigma^{\circ}},\mathbf{E}^d)$. We make the elementary observation that the perturbation data of~\cite[Definition 5.4]{PerutzSheridan} can be thought of as relative perturbation data for a trivial bundle with strip-like ends over a disk with no interior input marked points.

We also need to equip $\Sigma^{\circ}$ asymptotic conditions at punctures and tangency data at the stabilizing marked points. Suppose we are given Hamiltonian chords $\mathbf{y} = (y_0,y_1,\ldots,y_k)$, where $y_i\in\mathcal{C}(L,L)$ for all $i = 1,\ldots,k$ and $y_0\in\mathcal{C}(L,\phi^{d}(L))$. Further suppose we are given a horizontal sections $\nu\in\mathcal{H}(T_{\phi^d},\Theta_{\phi^d})$. Let $\pi_2(\nu;\mathbf{y})$ denote the set of homotopy classes of sections of $E\to\Sigma^{\circ}$ which approach $\nu$ and $\mathbf{y} = (y_0,y_1,\ldots,y_k)$ at the input and boundary punctures. For any $B\in\pi_2(\nu;\mathbf{y})$, we can define the intersection number $B\cdot E_q$, and since we are primarily interested in pseudoholomorphic sections as before we only need to consider classes for which the intersection numbers, for all $q\in Q$, are nonnegative. If $u$ is a section representing the class $B$, we can form a tangency vector $\iota(u,z_i)\in\mathbb{Z}_{\geq0}^Q$ which records the intersection number of $u$ with $\mathbf{E}_q$ at the marked point $z_i$ for each $q\in Q$. Given this, we can define tangency data $\mathbf{t}\colon\lbrace 1,\ldots,\ell\rbrace\to\mathbb{Z}_{\geq0}^Q$, including canonical tangency data $\mathbf{t}^{\mathrm{can}}$, just as we defined tangency data for maps to $X$.

The definition of the relevant spaces of pseudoholomorphic sections is now a straightforward combination of the definitions appearing earlier in this section. The $1$-form $K$ appearing in the definition of a relative perturbation datum determines a section of the bundle $\Hom(\pi^*T\Sigma^{\circ},TE^v)$, defined as follows. For any $x\in E$, consider some $\xi\in T_t\Sigma^{\circ}$, where $\pi(x) = t$. We can lift $\xi$ to a tangent vector $\tilde\xi$ on the fiber $X_t=\pi^{-1}(t)$, and we obtain a function $K(\tilde\xi)$ on $X_t$ which is independent of the choice of the lift. We define $\eta(\xi)_x$ to be the value at $x$ of the associated Hamiltonian vector field on $X_t$.
\begin{definition}\label{tocsectiondefinition}
Given a relative perturbation datum $P = (Y,K)$, let $\mathcal{M}(B,\nu,\mathbf{y},\mathbf{t};P)$ denote the space of sections of $E$ which have a boundary condition as in Definition~\ref{tcoboundaryconditions}, and which are solutions to the perturbed Cauchy--Riemann equation
\begin{align*}
(Du-\eta)\circ j = J_z\circ(Du-\eta)
\end{align*}
(where $J_z$ denotes the value of the almost complex structure at $z\in\Sigma^{\circ})$ with negative limits $y_1,\ldots,y_k,$ and $\nu_1,\ldots,\nu_m$, and positive limit $y_0$. Additionally, we require that the local intersection numbers of $u$ with the divisors $\mathbf{E}_q$ at the points $z_i$ are given by $\iota(u,z_i)=\mathbf{t}(i)$ for all $i = 1,\ldots,\ell$. Denote by
\begin{align}\label{generaltcosections}
\mathcal{M}(B,\nu,\mathbf{y};P)= \mathcal{M}(B,\nu,\mathbf{y},\mathbf{t}^{\mathrm{can}};P)
\end{align}
the moduli space defined using canonical tangency data.
\end{definition}

\subsection{Transversality and compactness for spaces of sections} Similarly to~\cite[Definition 12.2]{SeiThesis}, we have a notion of a broken pseudoholomorphic section with sphere bubbles attached. On one hand, we allow ourselves to attach trees of pseudoholomorphic bubbles to disks appearing in Definititon~\ref{tocsectiondefinition}. More precisely, suppose we are given a collection of trees $\lbrace\Gamma_i\rbrace$ of unmarked pseudoholomorphic spheres, where each tree consists of pseudoholomorphic spheres in a \textit{fixed} fiber of $E\to\Sigma^{\circ}$. Let $\mathcal{M}_{\lbrace\Gamma_i\rbrace}(B,\nu,\mathbf{y},\mathbf{t};P)$ denote the moduli space of sections, as in Definition \ref{tocsectiondefinition}, with unmarked bubble trees attached. See~\cite[\S{5.6}]{PerutzSheridan} for a detailed discussion of how attachments are carried out. On the other hand, we allow breaking of sections both along boundary punctures and along the interior input punctures, meaning that a broken pseudoholomorphic disk with unmarked bubble trees can have components which look like elements of~\eqref{cylindersectionmodulispaces}.

We emphasize that such an object will typically \textit{not} be the Gromov limit of a family of pseudoholomorphic sections of Definition~\ref{tocsectiondefinition} (see the remark in~\cite[\S{5.6}]{PerutzSheridan}.) Such a Gromov limit will instead be a broken section with possibly with stabilizing marked points and tangency conditions on the sphere bubbles and cylindrical components. By our constancy and vanishing assumptions on $Y$ and $K$, a broken section in our sense can always be obtained from a Gromov limit by forgetting these extra stabilizing marked points and collapsing any unstable components. With this understood, we can define the \textit{Gromov compactification} $\overline{\mathcal{M}}(B,\nu,\mathbf{y},\mathbf{t};P)$ of $\mathcal{M}(B,\nu,\mathbf{y},\mathbf{t};P)$ to be the space of such broken sections with unmarked bubble trees attached. To describe the boundary strata of this compactification, we are implicitly using the forgetful maps of stabilizing marked points to identify the moduli spaces of Floer cylinders with interior marked points and tangency conditions with their standard unmarked versions.

Similarly, we can generalize the notion of a relative perturbation datum to the domain of such a broken curve in the obvious way, by insisting that $Y$ and $K$ are constant or vanish, respectively, on all sphere components. The disk components of a broken pseudoholomorphic section can be thought of as (normalizations of) fibers of the universal family $\overline{\mathcal{R}}_{k,\ell+m}$, so we can make universal choices of relative perturbation data~\cite[Definition 5.5]{PerutzSheridan} following the proof of~\cite[Lemma 5.6]{PerutzSheridan}. These universal choices of perturbation data should be consistent over the boundary strata of $\overline{\mathcal{R}}_{k,\ell+m}$ (cf. the \textbf{(Consistent on disks)} condition of~\cite[Definition 5.5]{PerutzSheridan}) and they should be equivariant with respect to the action of $S_{\ell}$ given by permuting the stabilizing marked points.

We can now construct relative perturbation data following~\cite{PerutzSheridan}, first for the moduli spaces $\mathcal{M}^*_{\lbrace\Gamma_i\rbrace}(B,\nu,\mathbf{y},\mathbf{t};P)$, by which we mean the subspace of $\mathcal{M}_{\lbrace\Gamma_i\rbrace}(B,\nu,\mathbf{y},\mathbf{t};P)$ with unmarked trees of \textit{simple} sphere bubbles attached.
\begin{lemma}
For fixed $(B,\nu,\mathbf{y},\mathbf{t})$, there is a comeager space of universal relative perturbation data $P_{k,\ell+m,F}$ for which $\mathcal{M}^*_{\lbrace\Gamma_i\rbrace}(B,\nu,\mathbf{y},\mathbf{t};P_{k,\ell+m,F})$ is regular, constructed by induction on $k+2(\ell+m)\geq2$. Moreover we can choose regular universal relative perturbation data from a comeager set.
\end{lemma}
The proof of this lemma closely follows~\cite[Lemma 5.9 and Corollary 5.10]{PerutzSheridan}, which incorporates higher tangency conditions, following~\cite{CieliebakMohnke}, in the standard transversality arguments for disks~\cite[(9k)]{SeiBook}. Working with sections instead of maps to a closed symplectic manifold presents no additional difficulties.

Following ~\cite[\S 5.6, 5.7]{PerutzSheridan}, we can again rule out sphere bubbling for dimension reasons, which shows that the moduli spaces~\eqref{generaltcosections} have all of the properties needed to construct the twisted closed-open map.
\begin{lemma}
If $P$ is a regular relative perturbation datum and $\dim\mathcal{M}(B,\nu,\mathbf{y};P)\leq 1$, then the (broken) curves in $\overline{\mathcal{M}}(B,\nu,\mathbf{y};P)$ have no sphere bubbles.
\end{lemma}
The proof of this lemma is identical to the proof of~\cite[Lemma 5.16]{PerutzSheridan}.

\subsection{Constructing the twisted closed-open map}
We take the fixed point Fleor cochain groups $CF^*(X_t,\phi^d)$ to be defined over the Novikov field $\Lambda$. To simplify some technical matters relating to the existence of bounding cochains, we will assume that:
\begin{assumption}\label{strictactionassumption}
The perturbation and Floer data used to construct the full subcategory of $\Fuk(X_t,\mathbf{E}_t)$ whose objects are supported on $\lbrace\phi^d(L)\rbrace_{d\in\mathbb{Z}}$ are chosen to be equivariant with respect to the $\mathbb{Z}$-action generated by $\phi$. In particular each $\phi^d$ induces a strict $A_{\infty}$-isomorphism $CF^*(L)\to CF^*(\phi^d(L))$ for all $d\in\mathbb{Z}$.
\end{assumption}
We can guarantee this when $L$ is a tropical Lagrangian section, since we can see in a regularized neighborhood that it acts freely on the Lagrangians $\lbrace\phi^d(L)\rbrace_{d\in\mathbb{Z}}$. This assumption implies that if $b\in CF^*(L)$ is a bounding cochain, then its image $\phi^d(b)\in CF^*(\phi^d(L))$ is also a bounding cochain.

As a preliminary step toward defining the twisted closed-open map, we will define so-called \textit{twisted $\mathfrak{q}$-operators} by counting appropriate sections, which consist of map
\begin{align}\label{twistedqoperators}
\mathfrak{q}_k\colon CF^*(X_t,\phi^d)\otimes \bigotimes_{i=1}^k CF^*(L)\to CF^*(L,\phi^d(L))
\end{align}
where $d\geq0$ is fixed. These operations count pseudoholomorphic sections of the fibrations $E_{\phi^d,\Sigma^{\circ}}$, where $\Sigma^{\circ}$ has input interior puncture and $k+1$ boundary punctures, with the Lagrangian boundary condition of Definition~\ref{tcoboundaryconditions}.

Given $\nu\in\mathcal{H}(T_{\phi^d},\Theta_{\phi^d})$ and a sequence of Hamiltonian chords $\mathbf{y}$, let $B$ be class in $\pi_2(\nu,\mathbf{y})$. Let $\mathfrak{t}\colon\lbrace 1,\ldots,\ell\rbrace\to Q$, where $\ell$ depends on $B$, be a function used to define the canonical tangency datum, which we omit from the notations. Also let $\lambda(B)$ denote the energy of this class, which is given by the energy of any section $u\colon\Sigma^{\circ}\to E$ representing $B$. Then we define
\begin{align*}
\mathfrak{q}_k\coloneqq\sum_{u\in\mathcal{M}(\nu,\mathbf{y},A,\mathbf{P})/S_{\mathfrak{q}}}T^{\lambda(B)}\sigma_u
\end{align*}
where the sum is over all elements of all moduli spaces of virtual dimension $0$ with the given boundary conditions and each $u\in \mathcal{M}(\nu,\mathbf{y},B,\mathbf{P})$ again denotes an isomorphism of orientation lines
\begin{align*}
\sigma_u\colon o_{\nu}\otimes o_{y_1}\otimes\cdots\otimes o_{y_k}\to o_{y_0} \,.
\end{align*}

To define the twisted closed-open map, one would na{\"i}vely only consider disks with one boundary marked point and one input interior marked point, but this needs to be deformed by bounding cochains.
\begin{definition}
Let $(L,b)$ be an object of the relative Fukaya category $\Fuk(X_t,\mathbf{E}_t)$ satisfying Assumption~\ref{strictactionassumption}, where $b$ is a bounding cochain. Define the \textit{twisted closed-open map}
\begin{align*}
\mathcal{CO}_{\phi^d}\colon CF^*(X_t,\phi^d)\to CF^*(L,\phi^d(L))
\end{align*}
to be the sum
\begin{align*}
\mathcal{CO}_{\phi^d}(\nu)\coloneqq\sum_{k=0}^{\infty}\mathfrak{q}_k(\nu,b^{\otimes k})
\end{align*}
(where we have abused notation and let $\nu$ denote an element of the orientation line $o_{\nu}$.)
\end{definition}
This sum converges because bounding cochains have positive valuation in the Novikov field.
\begin{lemma}\label{tcochainmaplemma}
The twisted closed-open map $\mathcal{CO}_{\phi^d}$ is a chain map.
\end{lemma}
\begin{proof}
Consider the $1$-dimensional moduli spaces 
\begin{align}\label{1-dimensional-gromov-compactifications-1}
    \overline{\mathcal{M}}(\nu,\mathbf{y},B,\mathbf{P})/S_{\mathbf{q}}
\end{align}
where $\mathbf{y} = (y_0,y_1,\ldots,y_k)$ is chosen so that $y_i$, for $i = 1,\ldots,k$ are asymptotic to be degree $1$ Hamiltonian chords, which account for the bounding cochain insertion $b$. There are two types of boundary strata in these moduli spaces:
\begin{itemize}
\item[(i)] Products 
\[\mathcal{M}(E_{\phi},J;B'')\times {\mathcal{M}}(\nu^{+},\mathbf{y},B',\mathbf{P})/S_{\mathbf{q}'}\]
where $B'\in\pi_2(\nu^{+},\mathbf{y})$ and $B''\in\pi_2(\nu^{-},\nu^{+})$ are any classes such that $B$ is obtained by attaching (representatives of) $B'$ and $B''$ along input punctures, and the tangency data are determined by forgetting marked points on the cylindrical component, and;
\item[(ii)] Products
\[ 
{\mathcal{M}}(\nu,\mathbf{y}',B',\mathbf{P}')/S_{\mathbf{q}'}\times {\mathcal{M}}(\mathbf{y}'',B'',\mathbf{P}'')/S_{\mathbf{q}''}
\]
where $B'$ and $B''$ whose gluing along the appropriate puncture gives $B$ (and similarly for the tangency data, the Hamiltonian chords $\mathbf{y}$, and the perturbation data $\mathbf{P}$.)
\end{itemize}
If we take the inputs of all input boundary punctures to be copies of the bounding cochain $b$, then summing over all such boundary strata gives us a relation between compositions of various operators on Lagrangian and fixed point Floer cochain groups. The first type of boundary components contributes a term of the form $\mathcal{CO}\circ d(o_{\nu^{-}})$, where $d$ is the fixed point Floer differential.

The second type of boundary components can contribute expressions of three different types. Before we enumerate these, note that for topological reasons, in any Gromov limit of any family in~\eqref{1-dimensional-gromov-compactifications-1} will contain two broken sections, one of which contains the input interior puncture. Since the holonomy of the fibration about this punctured is $\phi^d$, the restriction of the fibrations to the other component will be trivial for topological reasons. Therefore on the component over which the fibration is trivial, either (a) all input boundary punctures will be asymptotic to chords in $\mathcal{C}(L,L)$, (b) all boundary punctures will be asymptotic to chords in $\mathcal{C}(\phi^d(L),\phi^d(L))$, or (c) this component will contain the output puncture. The sum all of the contributions of types (a) and (b) vanishes, by the Maurer--Cartan relations for $b$ and $\phi^d(b)$, respectively. On the other hand, the contributions of type (c) contribute a term $\mathfrak{m}_1^{b,\phi^d(b)}\circ\mathcal{CO}(o_{\nu^{-}})$, where $\mathfrak{m}_1^{b,\phi^d(b)}$ denotes the Floer differential on $CF^*(L,\phi^d(L))$ deformed by the bounding cochains $b$ and $\phi^d(b)$. Thus we obtain
\[ \mathcal{CO}\circ d(o_{\nu}) = \mathfrak{m}_1^{b,\phi^d(b)}\circ\mathcal{CO}(o_{\nu}) \,.\]
\end{proof}

For a Lagrangian $L$ satisfying Assumption~\ref{strictactionassumption}, we consider
\begin{align*}
\mathcal{CO}_+^*\coloneqq\bigoplus_{d=0}^{\infty}\mathcal{CO}_{\phi^d}^*\colon\bigoplus_{d=0}^{\infty}CF^*(X_t,\phi^d)\to\bigoplus_{d=0}^{\infty}CF^*(L,\phi^d(L))
\end{align*}
and the induced map $[\mathcal{CO}_+]$ on cohomology. To show that $[\mathcal{CO}_+]$ is a ring homomorphism, we consider moduli spaces with `geodesic constraints,' similarly to~\cite{GanThesis}. We begin by introducing auxiliary moduli spaces of domains $\mathcal{R}_{k,\ell+2}(r)$, where $r\in(0,1)$ consisting of disks $\Sigma$ with $k+1$ boundary punctures $\zeta_0,\zeta_1,\ldots,\zeta_k$ (in order), $\ell$ interior stabilizing marked points, and $2$ interior input marked points $w_1$ and $w_2$, such that up to an automorphism of the domain, the points $\zeta_0$, $w_1$, and $w_2$ lie at $-i$, $-r$, and $r\in\mathbb{C}$, respectively. Let $\Sigma^{\circ}$ denote the complement of the boundary marked points and input interior marked points. Also let $\mathcal{R}_{k,\ell+2;\perp}$ denote the union $\bigcup_{r\in(0,1)}\mathcal{R}_{k,\ell+2}(r)$.

We will construct a symplectic fibration $E_{\phi^{d_1},\phi^{d_2},\Sigma^{\circ}}\to\Sigma^{\circ}$ which has holonomy $\phi^{d_i}$ about the puncture at $w_i$ for $i = 1,2$. To that end, consider the line segments in $\Sigma$, parametrized as above, joining $w_i$ to $\zeta_0$, which are disjoint in the interior of the disk. Then, starting from the trivial fibration with fiber $X_t$ over the complement of these arcs, we can form the desired fibration over $\Sigma^{\circ}$ by gluing along the interiors of the two arcs. This is a mild generalization (to higher dimensions) of the construction given in the proof of Theorem 4.7 of~\cite{JYZ}. We will also write $E = E_{\phi^{d_1},\phi^{d_2},\Sigma^{\circ}}$ when the domain and holonomy transformations $\phi^{d_i}$ are understood from context.

Suppose we are given a sequence of Hamiltonian chords $\mathbf{y} = (y_0,y_1,\ldots,y_k)$ and a pair of horizontal sections $\nu = (\nu^1,\nu^2)$, where $\nu^i\in\mathcal{H}(T_{\phi^{d_i}},\Theta_{\phi^{d_i}})$ for $i = 1,2$. Let $\pi_2(\nu,\mathbf{y})$ denote the set of homotopy classes of sections of a symplectic fiber bundle $E$ over the disk $\Sigma^{\circ}$ with two interior punctures and fiber $X$ as before.

After choosing universal and consistent perturbation data and restricting to the canonical tangency condition as above, we can form the moduli spaces
\begin{align}\label{geodesicmodulispaces}
\mathcal{M}_{\perp}(\nu,\mathbf{y},B,\mathbf{P})
\end{align}
of pseudoholomorphic sections of this fibration, where the domains come from disks in $\mathcal{R}_{k,\ell+2;\perp}$.

\begin{lemma}\label{tcoringhomlemma} The map
\[[\mathcal{CO}_+]\colon \bigoplus_{d=0}^{\infty}HF^*(X_t,\phi^d)\to\bigoplus_{d=0}^{\infty}HF^*(L,\phi^d(L))\]
is a ring homomorphism.
\end{lemma}
\begin{proof}
We observe that the compactifications of the moduli spaces $\mathcal{R}_{k,\ell+2;\perp}$ admit submersions to $[0,1]$. These give rise to boundary strata in the Gromov compactifications of~\eqref{geodesicmodulispaces} (of dimension $1$), which are described as follows.
\begin{itemize}
\item[$\bullet$] In the degeneration $r\to 0$, elements of the corresponding boundary components of~\eqref{geodesicmodulispaces} consist of broken sections with one pair of pants component, with inputs $\nu^{d_1}$ and $\nu^{d_2}$, and one disk component with a single interior puncture. This corresponds to the product on the closed-string side.
\item[$\bullet$] In the degeneration $r\to1$, the elements of the boundary components of~\eqref{geodesicmodulispaces} consist of broken sections with three components. The first of these components is a disk with no interior punctures, and the other two components each contain one interior punctures $w_i$, and the induced fibration over this component has holonomy $\phi^{d_i}$ about $w_i$. This corresponds to the product on the open-string side after applying the twisted closed-open map $\mathcal{CO}_+$.
\end{itemize}
This follows the standard analysis of the boundaries of moduli spaces with a constraint of this form~\cite[Prop. 5.4]{GanThesis}, see in particular~\cite[Figures 15 and 16]{JYZ} for a description of these degenerations accounting for holonomy.

The other boundary components are similar to those already discussed in the proof of Lemma~\ref{tcochainmaplemma}; in particular these will either have cylinder or disk components. The cylinder components correspond to terms involving the differential on $CF^*(X_t,\phi^{d_i})$, and the disk components either correspond to possibly nontrivial terms involving the differential on $\bigoplus_{d=0}^{\infty}CF^*(X_t,\phi^d)\to\bigoplus_{d=0}^{\infty}CF^*(L,\phi^d(L))$ deformed by bounding cochains, or will contribute terms involving disk bubbles, which ultimately cancel each other.
\end{proof}

\begin{proof}[Proof of Proposition~\ref{closedopenisomorphism}]
Using the fact that $\pi$ (which we call denotes $\pi_{reg}$ per Assumption~\ref{completedfamilyassumption}) is compatible with the regularization, we can use the expression of~\eqref{compatiblewithregularizationdefinitingequation} to write $\pi$ in the form~\eqref{semistablelocaldescription} near $\mathbf{D}$. Given a tropical Lagrangian section $L\subset X_t$ in a smooth fiber of $\pi$ sufficiently close to $\mathbf{D}$, we know from Definition~\ref{tropicallagrangiandefinition} that $L$ is locally, in $X_t$, described as a product Lagrangian section of the natural fibration on $(\mathbb{C}^*)^{k-1}\times\mathbb{C}^{n-k+1}$. In the same local charts, the flow of the horizontal lift $\xi^{reg}$ (see Lemma~\ref{orbitsinstandardneighborhoods}) of the rotational vector field on $\mathbb{D}^{\circ}$ will look like a multiple of~\eqref{horzlift} by some nonvanishing real-valued function defined in terms of the Riemannian distance functions implicit in the Hermitian structures of the regularization. This assertion follows from the local expression~\eqref{standardlocalformtubularnbhd} for the regularized symplectic form.

After perturbing the flow of $\xi$ (see~\eqref{explicithamiltonianperturbation}), we see that $CF^0(L,\phi^d(L))$ is supported in degree $0$ when $d$ is nonnegative. It follows from the Morse--Bott analysis of the flow of $\xi^{reg}$ carried out above (see the proof of Lemma~\ref{PSS-domain-codomain-abstract-iso}) that the twisted closed-open map induces the claimed isomorphism.
\end{proof}

\part{Toric degenerations}

\section{Symplectic topology of toric degenerations}\label{toricdegeneration-review-section}
This section summarizes the properties of toric degenerations that we will use to study the Floer theory of their smooth fibers. Most of the material in this section is not new, though an explicit description of the positive real locus in a (positive real) fiber of a real toric degeneration along the lines of Theorem~\ref{lagrangianpositivereallocus} has not, to our knowledge, appeared in the literature.
\subsection{Complex analytic toric degenerations} To eventually study their symplectic topology, we interpret toric degenerations in the complex analytic category following~\cite{RuddatSiebert, Arguzreal}. We begin by describing the central fibers of toric degenerations.
\begin{definition}\label{toriclogcydefn}
A \textit{toric log Calabi--Yau space} is a reduced complex Gorenstein space $X_0$ of pure dimension $n$, whose irreducible components are all normal toric varieties, glued pairwise torically along their toric prime divisors. 
In this situation let $\nu\colon\widetilde{X}_0\to X_0$ denote the normalization.
\end{definition}
We say that a normal complex space $\mathcal{X}$ with a divisor $X_0$ is a \textit{toroidal pair} if $(\mathcal{X},X_0)$ is locally isomorphic to $(Y,D)$, where $Y$ is an affine toric variety and $D$ is the union of prime toric divisors in $Y$.
\begin{definition}\label{toricdegenerationdefn}
Let $\mathbb{D}$ denote a disk centered at the origin in $\mathbb{C}$, and let $\mathcal{X}$ be a proper normal complex space. A proper flat holomorphic map
\begin{equation*}
\pi\colon\mathcal{X}\to\mathbb{D}
\end{equation*}
is called a \textit{toric degeneration} if the central fiber $X_0\coloneqq\pi^{-1}(0)$ is a toric log Calabi--Yau, and if there is a closed analytic subset $\mathcal{Z}\subset\mathcal{X}$ of relative codimension $2$ such that 
\begin{itemize}
\item[(a)] $(\mathcal{X},X_0)$ is a toroidal pair locally near any point in $\mathcal{X}\setminus\mathcal{Z}$;
\item[(b)] $Z\coloneqq\mathcal{Z}\cap X_0$ is contained in the singular locus $(X_0)_{\sing}$ (i.e. the union of toric prime divisors in the irreducible components of $X_0$) and;
\item[(c)] $Z$ does not contain the image under $\nu$ of any toric stratum of $\widetilde{X}_0$.
\end{itemize}
The subset $Z$ is called the \textit{log singular locus} of the toric degeneration (cf. Remark~\ref{loglocalmodels}). As before, we require that the fibers $X_t\coloneqq\pi^{-1}(t)$ for $t\neq0$ are smooth closed Calabi--Yau manifolds in the strict sense.
\end{definition}
Using the polarizations~\eqref{polarizationassumption}, we can make sense of the smooth fibers of toric degenerations as symplectic manifolds. In \S{\ref{reallagrangianconstructionsection}}, we will consider toric degenerations $\mathcal{X}\to\mathbb{D}$ with real involutions, at which point (see Assumption~\ref{symplecticformchoice}) we will choose a fixed symplectic structure of this type for which the real involution becomes anti-symplectic.
\begin{remark}
The general fibers of a toric degeneration $\mathcal{X}$ are not required to be smooth in general, but this is a necessary assumption to define the Fukaya category using the currently available technology.

Condition (c) of Definition~\ref{toricdegenerationdefn} on the log singular locus is adapted from the original definition of toric degeneration~\cite[Definition 4.1]{GS1}, but this restriction can be weakened. The requirement that $Z$ merely avoids the $0$-dimensional toric strata of $\widetilde{X}_0$ is called \textit{faithfulness} in~\cite{Arguzreal}.
\end{remark}

The main examples of toric degenerations that we consider are Batyrev--Borisov complete intersections~\cite{GrossBB} and the canonical degenerations of~\cite{GSrealaffine}. In~\cite{GSrealaffine}, toric degenerations are constructed using polyhedral affine manifolds equipped with certain auxiliary structures.

\begin{definition}
Let $B$ be a topological manifold with $\dim_{\mathbb{R}} B = n$. Let $\mathscr{P}$ be a polyhedral decomposition for $B$, i.e. a decomposition of $B$ into integral polytopes for which the attaching maps of faces preserve the integral affine structure. Furthermore, we assume that no cell of $\mathscr{P}$ self-intersects in $B$. We say that such a pair $(B,\mathscr{P})$ is an integral polyhedral manifold. Let $\mathscr{P}_{\max}$ denote the set of $n$-dimensional polyhedra in $\mathscr{P}$.
\end{definition}
From a projective toric log Calabi--Yau $X_0$, one obtains an integral polyhedral manifold by gluing the moment polytopes of the components of $\widetilde{X}_0$ as follows. For each irreducible component $X_{i}$ of $X_0$, the restriction of the polarization $\mathcal{E}\mid_{X_i}$ is an ample line bundle. Let $\check{\sigma}_i\subset N_{\mathbb{R}}$ denote the Newton polytope of this line bundle. Then there is a bijective correspondence between the faces of $\check{\sigma}_i$ and the toric strata of $X_i$ which preserves inclusions. Thus for two irreducible components $X_i$ and $X_j$, the faces of $\check{\sigma}_i$ and $\check{\sigma}_j$ can be glued in a canonical way to give a polyhedral manifold $(B,\mathscr{P})$. See~\cite[Proposition 4.10]{GS1} for a proof that $B$ is indeed a topological manifold.
\begin{definition}
Given $X_0$, the polyhedral manifold $(B,\mathscr{P})$ obtained as above is called the \textit{intersection complex} of $X_0$.
\end{definition}

Conversely, given a polyhderal manifold $(B,\mathscr{P})$, and some additional \textit{gluing data} $s$, Gross and Siebert construct a projective log Calabi--Yau $X_0(B,\mathscr{P},s)$ with intersection complex $(B,\mathscr{P})$~\cite{GSrealaffine}. The construction of $X_0(B,\mathscr{P},s)$ makes reference to an integral affine structure defined on an open subset of $B$.
\begin{lemma}[{\cite[Proposition 1.27]{GS1}}]\label{singularaffinestructure}
Let $\Delta$ denote the subset of $B$ covered by the $(n-2)$-cells of the barycentric subdivision of $\mathscr{P}$ that do not contain a vertex of $\mathscr{P}$, and do not intersect the interior of any cell in $\mathscr{P}_{\max}$ or any $\mathscr{P}$. Then $B\setminus\Delta$ admits an integral affine structure, which agrees with the natural integral affine structures on the interiors of the cells of $\mathscr{P}_{\max}$.
\end{lemma}

Using so-called \textit{open gluing data} $s$, Gross and Siebert construct $X_0(B,\mathscr{P},s)$ as a scheme by gluing toric varieties along their toric boundaries. Roughly, open gluing data $s = (s_e)$, indexed by inclusions of faces $e\colon\eta\to\tau$ in $\mathscr{P}$ is a collection of piecewise multiplicative functions along $\tau$, valued in $\mathbb{C}^*$, subject to some additional constraints~\cite[Construction 1.17 and Definition 1.18]{GSrealaffine}. This gluing can be thought of as a type of clutching construction. As one might expect, certain (equivalence classes of) open gluing data can be recovered from a \v{C}ech $1$-cocycle, called \textit{lifted open gluing data}~\cite[Theorem 5.2]{GS1}. Here the \v{C}ech complex is defined with respect to a combinatorially defined open cover of $B$. 

The main results of~\cite{GSrealaffine} are proved under the assumption that the affine structure on $B\setminus\Delta$ as above is \textit{positive} and \textit{simple}. In this setting, one has a cohomological classification of the log Calabi--Yau spaces that can arise as the central fiber of a toric degeneration, determined by the open gluing data. Let $i\colon B\setminus\Delta\to B$ denote the inclusion, where $\Delta$ is the discriminant locus of Lemma~\ref{singularaffinestructure}, and let $\check{\Lambda}$ denote the local system of integral covectors on $B\setminus\Delta$.
\begin{theorem}[{\cite[Theorem 5.4]{GS1}}]\label{logCYclassification}
If $(B,\mathscr{P})$ is positive and simple, then there is a bijective correspondence between set of (not necessarily projective\footnote{Projectivity of a log Calabi--Yau space is controlled by an obstruction class associated to $s\in H^1(B,i_*\check{\Lambda}\otimes\mathbb{C}^*)$~\cite[Definition 5.14]{GHS}. Since we will always be working with a given projective toric degeneration, by ~\eqref{polarizationassumption}, we will not discuss these obstructions.}) log Calabi--Yau spaces with intersection complex $(B,\mathscr{P})$ up to isomorphisms preserving $B$ and the \v{C}ech cohomology group $H^1(B,i_*\check{\Lambda}\otimes\mathbb{C}^*)$.
\end{theorem}
When $(B,\mathscr{P})$ is positive and simple, and $\partial B = \emptyset$, Ruddat and Siebert show that any of the toric degenerations constructed in~\cite{GSrealaffine} correspond to flat analytic families over the disk $\mathbb{D}\subset\mathbb{C}$~\cite[Theorem 4.4]{RuddatSiebert}.
\begin{remark}
The intersection complex $(B,\mathscr{P})$ of any toric degeneration $\mathcal{X}\to\mathbb{D}$ in the sense of Definition~\ref{toricdegenerationdefn} is necessarily positive. On the other hand, simplicity implies that the singularities of the total space of $\mathcal{X}$ are locally indecomposable in a suitable sense (cf.~\cite[Definition 2.10]{Arguzreal}). See~\cite[\S{2.2}]{GS1} for a thorough discussion of simple singularities in the algebraic category.

In the main body of this paper, we will always take a toric degeneration with positive simple intersection complex as given, without explicitly referring to the reconstruction algorithm of~\cite{GSrealaffine}. Consequently, we will not need any details about how these conditions are defined. Indeed, most of the properties of $\mathcal{X}\to\mathbb{D}$ that we need are controlled by the \v{C}ech cohomology class of Theorem~\ref{logCYclassification}.
\end{remark}
We conclude this subsection by describing local models for toric degenerations derived from the gluing data by way of a log structure on $X_0(B,\mathscr{P},s)$. We also review some basic definitions about log structures in preparation for our description of real loci in \S{\ref{reallagrangianconstructionsection}}.
\begin{definition}
A \textit{log structure} on a complex analytic space $X$ is a sheaf of monoids $\mathcal{M}_X$ on $X$ together with a homomorphism of sheaves of monoids $\alpha\colon\mathcal{M}_X\to\mathcal{O}_X$ such that
\[ \alpha\colon\alpha^{-1}(\mathcal{O}_X^{\times})\to\mathcal{O}_X^{\times} \]
is an isomorphism. Here $\mathcal{O}_X^{\times}$ denotes the sheaf of invertible elements of $\mathcal{O}_X$. We refer to a complex analytic space equipped with a log structure as a \textit{log analytic space}.
\end{definition}
One can associate the \textit{trivial log structure} to any complex analytic space by setting $\mathcal{M}_X = \mathcal{O}_X^{\times}$. The simplest nontrivial log analytic space is the \textit{standard log point} $\mathbb{C}^{\dagger}$, which has underlying space $\operatorname{Spec}\mathbb{C} = \pt$ and log structure $\alpha\colon\mathbb{C}^*\oplus\mathbb{N}\to\mathbb{C}$ given by
\begin{align*}
\alpha(x,q)\coloneqq\begin{cases}
x & \text{if }q = 0 \\
0 & \text{if }q\neq 0\,.
\end{cases}
\end{align*}

The most important log structures, for our purposes, are the \textit{divisorial log structures} associated to divisors $D\subset X$. Letting $j\colon X\setminus D\to X$ denote the inclusion, the divisorial log structure consists of the sheaf of monoids
\begin{equation}\label{divisoriallogstructure}
\mathcal{M}_{(X,D)}\coloneqq (j_*\mathcal{O}_{X\setminus D}^{\times})\cap\mathcal{O}_X
\end{equation}
and the inclusion $\mathcal{M}_{(X,D)}\hookrightarrow\mathcal{O}_X$.

\begin{remark}\label{loglocalmodels}
A toric degeneration $\pi\colon\mathcal{X}\to\mathbb{D}$, induces a morphism $\pi^{\dagger}\colon\mathcal{M}_{(\mathcal{X},X_0)}\to\mathcal{M}_{(\mathbb{D},\lbrace 0\rbrace)}$ of log analytic spaces. If $j\colon X_0\to\mathcal{X}$ denotes the inclusion, then the restriction \begin{align*}
\mathcal{M}_{X_0}\coloneqq j^{-1}\mathcal{M}_{(\mathcal{X},X_0)}
\end{align*}
determines a log structure on $X_0$. The morphism $\pi^{\dagger}$ restricts to a log morphism $\mathcal{M}_{X_0}\to\mathbb{C}^{\dagger}$, which is \textit{log smooth} away from $Z$ (see~\cite[Definition 3.10]{ModuliHandbook} for a definition of log smoothness.)
\end{remark}

From lifted open gluing data $s\in H^1(B,i_*\check{\Lambda}\otimes\mathbb{C}^*)$, Gross and Siebert define a log structure on the corresponding log Calabi--Yau space $X_0(B,\mathscr{P},s)$~\cite[\S{4}]{GS1}. If $X_0(B,\mathscr{P},s)$ is realized as the central fiber of a toric degeneration, then this log strcuture agrees with the divisorial log structure~\eqref{divisoriallogstructure}. The log structure determines a collection of so-called \textit{slab functions} $f_{\rho,x}\in\mathbb{C}[\Lambda_{\rho}]$ indexed by codimension $1$ faces $\rho$ of $\mathscr{P}$ and connected components $x$ of $\rho\setminus\Delta$~\cite[Definition 2.17]{GSrealaffine}. Near a general point in a codimension $1$ toric stratum of $X_0$, the total space of any simple toric degeneration $\mathcal{X}\to\mathbb{D}$ is locally modeled by the affine variety
\begin{align}\label{gslocalmodels}
\lbrace uv = t^{\kappa}f_{\rho,x}(z_1,\ldots,z_{n-1})\rbrace\subset\mathbb{C}^3_{(u,v,t)}\times(\mathbb{C}^*_z)^{n-1}
\end{align}
where $\kappa\in\mathbb{Z}_{>0}$ and $\rho\in\mathscr{P}$ is the $(n-1)$-cell corresponding to the given stratum~\cite[Theorem 3.22 and Proposition 4.20]{GS1}~,\cite[\S{2}]{GSinvitation}.

\subsection{Real toric degenerations and the positive real locus}\label{reallagrangianconstructionsection} Let $\mathcal{X}\to\mathbb{D}$ be a toric degeneration whose central fiber $X_0$ is a positive simple log Calabi--Yau space. It turns out that the field of definition of such a toric degeneration is controlled by the open gluing data used to construct the central fiber. Given a positive simple toric log Calabi--Yau space $X_0(B,\mathscr{P},s)$ with gluing data $s\in\check{H}^1(B,i_*\check{\Lambda}\otimes\mathbb{C}^*)$ defined over a subring $A\subset\mathbb{C}$, Gross--Siebert's reconstruction algorithm produces a toric degeneration defined over $A$~\cite[Theorem 5.2]{GSrealaffine}. In the special case $A = \mathbb{R}$, we will use a formulation of this result due to Arg\"{u}z.
\begin{lemma}[{\cite[Corollary 7.2 and Remark 7.4]{Arguzreal}}]\label{logrealstructures}
If $(B,\mathscr{P})$ is a positive simple polyhedral manifold defined by lifted open gluing data $s\in\check{H}^1(B,i_*\check{\Lambda}\otimes\mathbb{C}^*)$, then the log Calabi--Yau space $X_0(B,\mathscr{P},s)$ admits a real structure, which is compatible with the natural real structure on the standard log point and which restricts to the standard real structure on each toric irreducible component, if and only if $s$ lies in the image of
\begin{equation}\label{gluingdatainclusion}
\check{H}^1(B,i_*\check{\Lambda}\otimes\mathbb{R}^*) \to \check{H}^1(B,i_*\check{\Lambda}\otimes\mathbb{C}^*) \,.
\end{equation}
This real structure extends to the family $\mathcal{X}\to\mathrm{Spec}\,\mathbb{C}[\![ t ]\!]$ constructed in~\cite{GSrealaffine}.
\end{lemma}
We emphasize that such a real structure induces the trivial involution on $B$. Consequently, the real structure on $\mathcal{X}\to\mathrm{Spec}\,\mathbb{C}[\![ t ]\!]$ corresponds to a real structure on the associated analytic degeneration.\footnote{It is explained in~\cite[Remark 7.4]{Arguzreal} that the smoothing algorithm of~\cite{GSrealaffine} is compatible with real structures. One can then extend the real structure to the corresponding analytic degeneration following~\cite[Theorem 4.4]{RuddatSiebert}.} Conversely, given a complex analytic toric degeneration defined over $\mathbb{R}$, it is not difficult to see that its central fiber admits a real structure of the type specified in~\cite[Corollary 7.2]{Arguzreal}, which, by \textit{loc. cit.}, implies that its central fiber has lifted open gluing data in the image of~\eqref{gluingdatainclusion}.

There is an involution on the total space of any analytic toric degeneration $\mathcal{X}\to\mathbb{D}$ defined over $\mathbb{R}$ given by complex conjugation. This involution fixes the fiber $X_t$ of $\mathcal{X}$ over any real point $t\in\mathbb{D}$, and hence determines an anti-holomorphic involution
\begin{align}\label{realinvolution}
    \iota\colon X_t\to X_t \,.
\end{align}
We will fix a K{\"a}hler form on $X_t$ with respect to which this involution is anti-symplectic.
\begin{lemma}\label{realembedding}
Suppose that $\mathcal{X}\to\mathbb{D}$ is a toric degeneration. Let $t$ be a positive real number, and suppose that the fiber $X_t$ over $t$ is equipped with an anti-holomorphic involution $\iota$. Consider the line bundle $\mathcal{E}_t$ on $X_t$ obtained by restricting the polarization $\mathcal{E}$ of~\eqref{polarizationassumption}. Then we can assume, possibly after replacing $\mathcal{E}$ with a positive power, that $\mathcal{E}$ determines an embedding of $X_t$ into some projective space under which $\iota$ coincides with the restriction of complex conjugation. In particular $X_t$ is defined over $\mathbb{R}$.
\end{lemma}
\begin{proof}
Choose a generic section $s\in\Gamma(\mathcal{E}_t)$. If $s^{-1}(0)$ is not $\iota$-invariant, then we can consider the linear system $|s^{-1}(0)+\iota(s^{-1}(0))|$, which corresponds to $\mathcal{E}_t^{\otimes 2}$. In either case, $\mathcal{E}_t^{\otimes 2}$ determines a holomorphic embedding of $X_t$ into $\mathbb{CP}^N$ for some $N\in\mathbb{Z}$, as well as an involution on $\mathbb{CP}^N$ that extends $\iota$. If $N$ is even, this involution is necessarily conjugate to complex conjugation. Otherwise, if $N$ is odd and $N+1 = 2k$, this involution could be conjugate to
\begin{align*}
[z_1:\ldots:z_k:z_{k+1}:\ldots:z_{2k}]\mapsto[-\overline{z}_{k+1}:\ldots:-\overline{z}_{2k}:-\overline{z}_1:\ldots:-\overline{z}_k] \,.
\end{align*}
In this case, we can choose a Segre embedding, corresponding to a line bundle on $\mathbb{CP}^N$, of $\mathbb{CP}^N$ into an odd-dimensional projective space. Under this embedding, we can assume that $\iota$ extends to complex conjugation. This embedding corresponds to some power $\mathcal{E}_t$. 
\end{proof}

With this, we make a specific choice of K{\"a}hler forms associated to~\eqref{polarizationassumption}.
\begin{assumption}\label{symplecticformchoice}
Let $\mathcal{X}\to\mathbb{D}$ be a real toric degeneration. After replacing $\mathcal{E}$ with a positive power if necessary, we assume that the restriction $\mathcal{E}_t = \mathcal{E}\mid_{X_t}$ to a positive real fiber $X_t$ determines an embedding of $X_t$ into some projective space for which $\iota$ is the restriction of complex conjugation, per Lemma~\ref{realembedding}. We equip $X_t$ with the symplectic form given by pulling back the Fubini--Study form under this embedding.
\end{assumption}
When $X_t$ is equipped with such a symplectic form, the set of fixed points of $\iota$ is a -- possibly disconnected -- Lagrangian submanifold of $X_t$. We will mainly be interested in situations where this real locus has a distinguished connected component that is canonically homeomorphic to $B$, i.e. a \textit{positive real locus}. When the positive real locus is well-defined, the following lemma implies that it is a rational homology sphere.
\begin{lemma}
If $\mathcal{X}\to\mathbb{D}$ is a toric degeneration in the sense of Definition~\ref{toricdegenerationdefn}, then its intersection complex $B$ is a rational homology $n$-sphere.
\end{lemma}
\begin{proof}
By~\cite[Proposition 2.37]{GS1} and our restriction on the Hodge numbers of smooth fibers in Definition~\ref{toricdegenerationdefn}, the intersection complex $B$ has the homology of $S^n$ over $\mathbb{Q}$.
\end{proof}
The positive real locus should be thought of, morally, as the base of an SYZ fibration on $X_t$, and we can make this intuition more precise at the topological level following~\cite{Arguzreal}.

To do this, we begin by describing a topological model for $X_t$, constructed from the log structure $\mathcal{M}_{X_0}$ on $X_0$. More precisely, this is a fiber of the \textit{Kato--Nakayama space}~\cite{KatoNakayama} of $(X_0,\mathcal{M}_{X_0})$.

Let $(X,\mathcal{M}_X)$ be a complex analytic space equipped with a log structure. The \textit{polar log point} is the space $\Pi^{\dagger} \coloneqq (\mathrm{Spec}\,\mathbb{C},\mathcal{M}_{\Pi})$, where the log structure $\mathcal{M}_{\Pi}$ is given by
\begin{align*}
\alpha_{\Pi}\colon\mathcal{M}_{\Pi,0} = \mathbb{R}_{\geq0}\times U(1) &\to \mathbb{C} \\
(r,e^{i\theta}) &\mapsto re^{i\theta} \,.
\end{align*}
We can describe the Kato--Nakayama space $X^{\log}\coloneqq(X,\mathcal{M}_{X})^{\log}$ as a set in terms of log morphisms from $\Pi^{\dagger}$ to $(X,\mathcal{M}_{X_0})$. For a definition of the topology on $X^{\log}$, see~\cite[\S{1.1}]{Arguzreal}. A log morphism $f\colon\Pi^{\dagger}\to(X,\mathcal{M}_{X})$ with $f(0) = x$ is specified by monoid homomorphism $f^{\flat} = (\rho,\theta)\colon\mathcal{M}_{X,x}\to\mathbb{R}_{\geq0}\times U(1)$ such that
\begin{align*}
\alpha_{\Pi}\circ f^{\flat} = \operatorname{ev}_x\circ\alpha_{X,x}
\end{align*}
where $\alpha_{X,x}$ is the stalk of the structure homomorphism $\mathcal{M}_{X}\to\mathcal{O}_{X}$ and $\operatorname{ev}_x\colon\mathcal{O}_{X,x} \to \mathbb{C}$ is given by evaluation of analytic functions at the point $x\in X$. This condition determines $\rho$ uniquely by $\rho(s)  = |(\alpha_{X,x}(s))(x)|$ for any $s\in\mathcal{M}_{X,x}$, so the choice of $f^{\flat}$ is controlled by a monoid homomorphism $\theta\colon\mathcal{M}_{X,x}\to U(1)$ such that
\begin{align*}
(\alpha_{X,x}(s))(x) = |(\alpha_{X,x}(s))(x)|\cdot\theta(s)
\end{align*}
or equivalently by a group homomorphism $\theta\colon\mathcal{M}_{X,x}^{\mathrm{gp}}\to U(1)$ with this property, where $\mathcal{M}_{X,x}^{\mathrm{gp}}$ is the Grothendieck group of $\mathcal{M}_{X,x}$. As a set, we define
\begin{align}\label{katonakayamaset}
X^{\log}\coloneqq\left\lbrace(x,\theta) \bigg\vert x\in X\,,\theta\in\coprod_{x\in X}\Hom(\mathcal{M}_{X,x}^{\mathrm{gp}},U(1)) \text{ and }\theta(h) = \frac{h(x)}{|h(x)|}\text{ for all }h\in\mathcal{O}^{\times}_{X,x}\right\rbrace \,.
\end{align}
From this description, it follows that any point $f\in X^{\log}$ has a phase valued in $U(1)$, and that there is a natural projection map $X^{\log}\to X$. For the trivial log structure on $X$, this map is easily seen to be the identity. The topology on $X^{\log}$ is defined so that this projection is a (continuous) proper map~\cite[p. 36]{ModuliHandbook}.
\begin{remark}
In the Kato--Nakayama space of a divisorial log structure is closely related to the real blowup. If $X$ is smooth and $D\subset X$ is a simple normal crossings divisor, then the Kato--Nakayama space of $\mathcal{M}_{(X,D)}$ is indeed a globalized version of the real blowup~\cite[\S{8.2}]{ModuliHandbook}.
\end{remark}
As before, let $X_0$ denote the central fiber of a toric degeneration. There is a map of Kato--Nakayama spaces
\begin{align*}
\pi^{\log}\colon X_0^{\log}\to(\mathbb{C}^{\dagger})^{\log} = U(1)
\end{align*}
induced by the log morphism $\mathcal{M}_{X_0}\to \mathbb{C}^{\dagger}$. For any $\xi\in U(1)$, define
\begin{align}
X_0^{\log}(\xi)\coloneqq(\pi^{\log})^{-1}(\xi) \,.
\end{align}
The relevance of Kato--Nakayama spaces to our discussion is explained by the following theorem.
\begin{theorem}[Arg\"{u}z {\cite[Theorem 3.1]{Arguzreal}}, Gross]\label{toricdegenerationtopology}
Consider a simple toric degeneration with strongly semi-simple singularities $\mathcal{X}\to\mathbb{D}$ over a disk of radius $r>0$. Then $\mathcal{X}\setminus X_0\to\mathbb{D}^{\circ}$ is isomorphic, as a topological fiber bundle over the punctured disk, to
\begin{align*}
X_0^{\log}\times(0,1)\to(\mathbb{C}^{\dagger})^{\log}\times(0,1) = U(1)\times(0,1) \,. 
\end{align*}
This isomorphism induces a homeomorphism of fibers
\begin{align}
X_t\cong X_0^{\log}(e^{i\theta})
\end{align}
where $t\in\mathbb{D}^{\circ}$ has phase $\theta$.
\end{theorem}
\begin{remark}\label{stronglysemisimpleremark}
    Strongly semi-simple singularities are introduced in~\cite[Definition 2.10]{Arguzreal} as a strengthening of the notion of simple singularities. It is expected that this condition can be removed (see the discussion after the statement of Theorem 3.1 in \textit{op. cit.}.) Note also that, per this discussion, simple but not strongly (semi-)simple singularities only appear in relative dimension $>3$.
\end{remark}
Roughly, the strategy for proving this theorem is to compare $X_0^{\log}$ to the Kato--Nakayama space $(\mathcal{X},\mathcal{M}_{(\mathcal{X},X_0)})^{\log}$ of the divisorial log structure. The toric degeneration induces a continuous map
\begin{align*}
(\mathcal{X},\mathcal{M}_{(\mathcal{X},X_0)})^{\log}\to(\mathbb{D},\mathcal{M}_{(\mathbb{D},\lbrace 0\rbrace)})^{\log}
\end{align*}
and there is a homeomorphism 
\[(\mathbb{D},\mathcal{M}_{(\mathbb{D},\lbrace 0\rbrace)})^{\log}\cong\mathbb{R}_{\geq0}\times S^1\]
coming from a general relation between Kato--Nakayama spaces of divisorial log spaces and the real blowup along the divisor~\cite[\S{8.2}]{ModuliHandbook}.\footnote{In this context, the Kato--Nakayama space can be thought of as a topological model for the infinite root stack.} This gives us the following commutative diagram.
\begin{equation}\label{katonakayamadiagram}
  \begin{tikzcd}
    X_0^{\log} \arrow{d} \arrow{r} & (\mathcal{X},\mathcal{M}_{(\mathcal{X},X_0)})^{\log} \arrow{d}\arrow{r} & \mathcal{X}\arrow{d}{\pi} \\
    S^1 \arrow{r} & \mathbb{R}_{\geq0}\times S^1\arrow{r} & \mathbb{D}
  \end{tikzcd}
\end{equation}
The homeomorphism of Theorem~\ref{toricdegenerationtopology} is the composition of maps in the top row of~\eqref{katonakayamadiagram}, and one checks that this is a homeomorphism using the fact that the log structure on $\mathcal{X}$ is trivial away from $X_0$. Working with the Kato--Nakayama space, as opposed to $X_t$ directly, is helpful because it admits a topological torus fibration defined away from the log singular locus $Z$. Assuming projectivity, Ruddat and Siebert construct a so-called \textit{degenerate moment map}
\begin{align}
\mu\colon X_0\to B
\end{align}
in~\cite[Proposition 2.1]{RuddatSiebert}. The restriction of $\mu$ to each irreducible component of $X_0$ is a moment map with respect to \textit{some} $U(1)^n$-invariant K{\"a}hler form.\footnote{If the gluing data used to define $X_0$ is trivial, then the moment maps on each irreducible component are just the ones coming from the restriction of the polarization on $X_0$.} Projectivitiy of $X_0$ is used to construct a basis of theta sections for an ample line bundle on $X_0$ as in~\cite[\S{5.2}]{GHS}, which are used to define moment maps on each component of $X_0$, from which $\mu$ is constructed.

\begin{theorem}[{\cite[Theorem 4.7]{Arguzreal}}]
The composition
\begin{align}
\mu\circ\pi^{\log}\colon X_0^{\log}\to B
\end{align}
is a $T^{n+1}$-bundle over $B\setminus\mu(Z)$, where $Z\subset X_0$ denotes the log singular locus as before. This restricts to a $T^n$-bundle over $B\setminus\mu(Z)$ on the fiber $X_0^{\log}(\xi)$.
\end{theorem}
Since $(B,\mathscr{P})$ is simple, $B\setminus\mu(Z)$ deformation retracts onto
\begin{align*}
B'\coloneqq\bigcup_{\sigma\in\mathscr{P}_{\max}}\operatorname{Int}\sigma\cup\lbrace v\in\mathscr{P}\mid\dim(v) = 0\rbrace \,.
\end{align*}
By~\cite[Proposition 4.15]{Arguzreal}, the restriction of this $T^{n+1}$-bundle to $B'$ is classified by a class
\begin{align}\label{bundleextgroup}
\Ext^1(\Lambda,\underline{\mathbb{Z}}\oplus\underline{U(1)}) = H^1(B\setminus\mu(Z),\check{\Lambda})\oplus H^1(B\setminus\mu(Z),\check{\Lambda}\otimes U(1)) \,.
\end{align}
where $\underline{\mathbb{Z}}$ and $\underline{U(1)}$ denote the constant sheaves on $B\setminus\mu(Z)$. The second component of this extension class is
\begin{align*}
    \Arg(s)\in H^1(B\setminus\mu(Z),\check{\Lambda}\otimes U(1))
\end{align*}
the argument of the gluing data $s$ from Theorem~\ref{logCYclassification}. Since we have assumed that this gluing data is contained in the image of~\eqref{gluingdatainclusion}, it follows that $\Arg(s)$ is valued in $\lbrace\pm1\rbrace\subset U(1)$. For any point $x\in B'$, these classes give rise to a pair of group cohomology classes lying in
\begin{align}\label{bundleextgroupcohomology}
H^1(\pi_1(B',x),\Hom(\Lambda_x,\mathbb{Z}))\oplus H^1(\pi_1(B',x),\Hom(\Lambda_x,U(1))) \,.
\end{align}
By a standard interpretation of group cohomology, such a class induces a pair of crossed homomorphisms
\begin{align}\label{extensioncrossedhoms}
\lambda\colon\pi_1(B',x)&\to \Hom(\Lambda_x,\mathbb{Z}) \\
\theta\colon\pi_1(B',x)&\to \Hom(\Lambda_x,U(1)) \,. \nonumber
\end{align}

The next observation we will need is that a real involution on $(X_0,\mathcal{M}_{X_0})$ lifts to a real involution on the Kato--Nakayama space, which we can see from~\eqref{katonakayamaset}. In more detail, let $\iota_{X_0}$ denote the real structure on $(X_0,\mathcal{M}_{X_0})$ and consider the real involution $\iota_{\Pi^{\dagger}}$ on the polar log point given, as a map of spaces, by $(r,e^{i\theta})\mapsto (r,e^{-i\theta})$. The lift of $\iota_{X_0}$ to $X_0^{\log}$ is
\begin{align*}
\iota_{X_0}\colon\Hom(\Pi^{\dagger},\mathcal{M}_{X_0}) &\to \Hom(\Pi^{\dagger},\mathcal{M}_{X_0}) \\
f &\mapsto \iota_{X_0}\circ f\circ\iota_{\Pi^{\dagger}} \,.
\end{align*}
This involution is continuous with respect to the topology on $X_0^{\log}$. The Kato--Nakayama space thus has a real locus $(X^{\log}_0)_{\mathbb{R}}$, which Arg\"{u}z shows is a $2^{n+1}$-fold branched cover of $B$.
\begin{theorem}[{\cite[Theorem 7.5]{Arguzreal}}]\label{reallocuscovertheorem}
The restriction of $\mu\circ\pi^{\log}$ to $(X^{\log}_0)_{\mathbb{R}}$ is a surjection with finite fibers. Over $B\setminus\mu(Z)$, it is a covering map of degree $2^{n+1}$.
\end{theorem}
The monodromy of this cover of $B\setminus\mu(Z)$ is computed in~\cite[Remark 7.8]{Arguzreal}. By~\cite[Theorem 7.7]{Arguzreal}, the fiber over a point $x\in B'$ can be canonically identified with
\begin{align}\label{reallocusfiber}
\Hom(\Lambda_x,\mathbb{Z}/2\mathbb{Z})\oplus\Hom(\mathbb{Z},\mathbb{Z}/2\mathbb{Z})
\end{align}
where $\mathbb{Z}/2\mathbb{Z}$ is thought of as a the subgroup  $\lbrace\pm 1\rbrace\subset U(1)$. The monodromy operator of $\gamma\in\pi_1(B',x)$ acting on~\eqref{reallocusfiber} is
\begin{align}\label{reallocusmonodromy}
(\phi,\xi)\mapsto (\phi\circ T_{\gamma}+\xi\circ\lambda_{\gamma}+\theta_{\gamma})
\end{align}
where $T_{\gamma}$ denotes the monodromy action of $\gamma$ on $\Lambda_x$ and $\lambda_{\gamma}$ and $\theta_{\gamma}$ are the values of the extension classes~\eqref{extensioncrossedhoms} at $\gamma$.

The value of $\xi$ in~\eqref{reallocusfiber} determines whether the fiber of the real locus is contained in the fiber $X_0^{\log}(\pm 1)$ of the Kato--Nakayama space over $1$ or $-1$. These correspond to fibers of the toric degeneration over points on the positive or negative real axis, respectively. Examining~\eqref{reallocusmonodromy} shows that $\lambda_{\gamma}$ acts trivially on fibers of the real locus contained in $X_0^{\log}(1)$, corresponding to $\xi = 0\in\mathbb{Z}/2\mathbb{Z}$. On the other hand, if we assume that 
\begin{align*}
\Arg(s)\in\check{H}^1(B\setminus\mu(Z),\check{\Lambda}\otimes\mathbb{Z}/2\mathbb{Z})
\end{align*}
vanishes, then the contribution of $\theta_{\gamma}$ to~\eqref{reallocusmonodromy} vanishes. If $\phi\in\Hom(\Lambda_x,\mathbb{Z}/2\mathbb{Z})$ is the constant homomorphism with value $0$, it follows that the corresponding fiber of the real locus is fixed by the monodromy. This observation lets us describe a connected component of the real locus in $X_0^{\log}(1)$.
\begin{lemma}\label{topologicalpositivereallocus}
Suppose that $\Arg(s) = 0\in \check{H}^1(B,i_*\check{\Lambda}\otimes\mathbb{Z}/2\mathbb{Z})$, where $s\in\check{H}^1(B,i_*\check{\Lambda}\otimes\mathbb{R}^*)$ is the lifted gluing data associated to $X_0$. Then the real locus of $X_0^{\log}(1)$ contains a connected component homeomorphic to $B$.
\end{lemma}
\begin{proof}
Under the assumption on $\Arg(s)$, we can construct a $0$-section $B\setminus\mu(Z)\to X_{0,\mathbb{R}}^{\log}$ of the covering map of Theorem~\ref{reallocuscovertheorem}. This extends continuously over $B$.
\end{proof}
Finally, this component of the real locus corresponds to a distinguished Lagrangian submanfiold of a fiber $X_t$ of $\mathcal{X}\to\mathbb{D}$, over a positive real point.
\begin{theorem}\label{lagrangianpositivereallocus}
Let $\mathcal{X}\to\mathbb{D}$ be a strongly semi-simple toric degeneration defined over $\mathbb{R}$ with positive simple intersection complex $(B,\mathscr{P})$, and suppose that its central fiber $X_0$ corresponds, under Theorem~\ref{logCYclassification}, to a \v{C}ech cohomology class $s$ (which is necessarily contained in the image of~\eqref{gluingdatainclusion}) for which $\Arg(s) = 0$. Then $\mathcal{X}$ is defined over $\mathbb{R}$, and if $t\in\mathbb{D}$ is a positive real number, then the fixed-point set of $\iota\colon X_t\to X_t$, as defined in~\eqref{realinvolution}, contains a distinguished connected component homeomorphic to $B$. We call this Lagrangian submanifold $L_0\subset X_t$ the \textbf{positive real locus}. 
\end{theorem}
\begin{proof}
We need to show that the real component described in Lemma~\ref{topologicalpositivereallocus} corresponds to a component of the real locus in $X_t$. This follows by observing that the maps in~\eqref{katonakayamadiagram} respect the real involutions on $\mathcal{X}$ and $X_0$, as well as their lifts to Kato--Nakayama spaces.
\end{proof}
\begin{definition}
    We say that a toric degeneration satisfying the hypotheses of Theorem~\ref{lagrangianpositivereallocus} has \textit{positive real gluing data}.
\end{definition}

\subsection{Fiberwise divisors} To construct the relative Fukaya category following~\cite{PerutzSheridan}, we need to specify a suitable system of divisors. We have already discussed how to identify a system of divisors in $X_t$ that is disjoint from the positive real locus by Lemma~\ref{divisorconstruction1}. Additionally, we need such divisor to be fixed by the real involution $\iota$ on $X_t$, since we will need to consider its action on the Fukaya category. This requires a more specialized construction of a divisor than the one given in Lemma~\ref{divisorconstruction1}.
\begin{lemma}\label{divisorconstruction2}
Let $X_t$ be a fiber of a toric degeneration, with positive real gluing data, over a positive real point $t\in\mathbb{D}$. Then there is a system of divisors $\mathbf{E}_t$ in $X_t$, which is contained in an open subset $U\subset X_t$ that is disjoint from the positive real locus $L_0$. Moreover, each $E_{q,t}$ in $\mathbf{E}_t$ is fixed by both $\phi$ and the real involution.
\end{lemma}
\begin{proof}
By Assumption~\ref{symplecticformchoice}, we can assume that $\mathcal{E}_t$ determines a holomorphic embedding of $X_t$ into $\mathbb{CP}^N$ for which $\iota$ is obtained by restricting complex conjugation on $\mathbb{CP}^N$. In particular $L_0$ is mapped to the real locus of $\mathbb{CP}^N$ under this embedding, and since $L_0$ is compact we can choose a real hyperplane in $\mathbb{CP}^N$ which avoids it. This corresponds to a section $s_t\in\Gamma(\mathcal{E}_t)$, and by Bertini's theorem we can assume that this hyperplane section $E_t\coloneqq s_t^{-1}(0)$ of $X_t$ is smooth. By taking $s_t^{-1}(0)$ to be a real hyperplane section, we can guarantee that it is preserved by $\iota$. We can then imitate the proof of Lemma~\ref{divisorconstruction1}, defining all components $E_{q,t}$ of the system of divisors to be real hyperplane sections, from which it follows that $E_{q,t}$ is fixed by both $\phi$ and $\iota$.
\end{proof}
\begin{corollary}\label{BBcorollary}
    The smooth fibers over positive real points of Batyrev--Borisov degenerations associated to a nef partition (see~\cite[Definition 2.1]{GrossBB}) have positive real gluing data.
\end{corollary}
\begin{proof}
    In fact such degenerations have trivial gluing data, so \textit{a fortiori} have positive real gluing data.
\end{proof}

\section{Homological mirror symmetry for toric degenerations}\label{toric-degeneration-hms-section} In this section, we check that the reasoning of \S{\ref{mainsection}} applies to certain toric degenerations with positive real gluing data by verifying Assumptions~\ref{commutativeassumption} and~\ref{lagrangianassumption}. In these special cases, we can use the real involution to show that Assumption~\ref{commutativeassumption} follows from Assumption~\ref{lagrangianassumption}. In other words, we are reduced to showing that positive real loci are tropical Lagrangian sections. The approach that we take to achieving this in the present work uses the existence of suitable resolutions of a toric degeneration $\mathcal{X}\to\mathbb{D}$.
\begin{assumption}\label{toricdegenerationresolutionassumption}
    In this section, we assume that all toric degenerations $\mathcal{X}\to\mathbb{D}$ of projective Calabi--Yau manifolds under consideration, in addition being smooth away from $0\in\mathbb{D}$ and having positive real gluing data, admit resolutions $\mathcal{X}'\to\mathcal{X}$ to log smooth and minimal log Calabi--Yau degenerations $\mathcal{X}'\to\mathbb{D}$ for which the dual intersection complex of $X_0'$ is a subdivision of the intersection complex of $\mathcal{X}$ (possibly with different integral affine structures.)
\end{assumption}
Under this assumption, it follows that the essential skeleton of $\mathcal{X}^{\circ}$ is a subdivision of the intersection complex. Some cases where such resolutions have already been shown to exist are as follows.
\begin{proposition}\label{resolutionassumption--verified}
    Suppose that $\mathcal{X}\to\mathbb{D}$ is a toric degeneration of projective Calabi--Yau varieties which is smooth away from the origin and has positive real gluing data. If either
    \begin{itemize}
        \item[(i)] $\mathcal{X}\to\mathbb{D}$ is a MPCP Batyrev--Borisov degeneration constructed in~\cite{GrossBB}, or;
        \item[(ii)] $\mathcal{X}\to\mathbb{D}$ is a toric degeneration of relative dimension $3$ satisfying Assumption 6.1 of~\cite{Goncharov}
    \end{itemize}
    then $\mathcal{X}\to\mathbb{D}$ satisfies Assumption~\ref{toricdegenerationresolutionassumption}.
\end{proposition}
To clarify, this result applies to \textit{any} MPCP Batyrev--Borisov degeneration with strongly semi-simple sigularities, by Corollary~\ref{BBcorollary}, and to all toric degenerations of Calabi--Yau $3$-folds satisfying both Goncharov's assumptions and ours.
\begin{proof}
    In case (i), this follows from the proof of Theorem 1.1 in~\cite{Yamamoto}. In fact that the statement of \textit{loc. cit.} mentions the identification of the essential skeleton with the Gross--Siebert dual intersection which is the relevant fact for our purposes. In case (ii), the resolutions we need are the \textit{strongly admissible resolutions} of~\cite[Definition 6.11 and Proposition 6.15(3)]{Goncharov}.
\end{proof}
\begin{remark}
    It is expected that Assumption~\ref{toricdegenerationresolutionassumption} should hold much more generally than in the cases described above. Goncharov~\cite[\S{6}]{Goncharov} sketches an argument that would prove the existence of strongly admissible resolutions of certain toric degenerations of higher relative dimension. Since strongly admissible resolutions satisfy many more properties than one needs for Assumption~\ref{toricdegenerationresolutionassumption}, the techniques developed in \textit{op. cit.} may allow one to verify this assumption more generally. On the other hand, Siebert has informed us that one should be able to construct suitable resolutions of toric degenerations of relative dimension $3$ using a procedure similar to the proof of~\cite[Theorem 1.1]{Yamamoto}.
\end{remark}
With these minimal models in hand, we can apply Proposition~\ref{gsresolutionmodel}, so as to have a semistable degeneration of the form considered elsewhere in this paper.
\begin{lemma}\label{lagrangianassumption--verified}
    Suppose that $\mathcal{X}\to\mathbb{D}$ is a toric degeneration subject to Assumption~\ref{toricdegenerationresolutionassumption}. Then the positive real locus of Theorem~\ref{lagrangianpositivereallocus} is a tropical Lagrangian section, and so Assumption~\ref{lagrangianassumption} holds.
\end{lemma}
\begin{proof}
    Note that we can resolve $\mathcal{X}\to\mathbb{D}$ in such a way that the resolution is trivial near the $0$-strata of $X_0$, by Definition~\ref{toricdegenerationdefn}(c). Thus in the semistable charts near any $0$-stratum of the central fiber $\mathbf{D}$ of the resolution $\mathscr{X}$, the positive real locus $L_0$ coincides with a portion of the (ordinary) positive real locus in $(\mathbb{C}^*)^n$. In terms the toric degeneration, these regions correspond to regions near the big tori in the central fiber of $\mathcal{X}$.
    
    With this understood, it suffices to show that $L_0$ restricts to a Lagrangian section near the central strata of lower codimension. To that end, consider $L_0\cap U_I$ where $I\subset S_0$ corresponds to a stratum of positive codimension. Set $U_{I,t} \coloneqq U_I\cap X_t$. It is easy to see that the projection from $L_0\cap U_{I,t}$ onto the positive real locus in $U_{I,t}$ (by which we mean the product of the positive real locus in the $(\mathbb{C}^*)^{k-1}$-factor with the real locus in the $\mathbb{C}^{n-k+1}$-factor) is a surjection. So for any point $p\in L_0\cap U_{I,t}$, we move $p$ by a translation (which is also a Hamiltonian isotopy) so that it lies in the positive real locus of $U_{I,t}$. An elementary argument shows that by a biholomorphic map between neighborhoods of $p$ contained in $\mathbb{C}^n$, we can identify a neighborhood of $L_0$ around $p$ with the real locus. Thus after a Hamiltonian isotopy we can arrange that $L_0$ satisfies Definition~\ref{tropicallagrangiandefinition}(i). The second item, Definition~\ref{tropicallagrangiandefinition}(ii), follows immediately since we have identified the essential skeleton with $L_0$ topologically and described its projection onto the intersection complex through Assumption~\ref{toricdegenerationresolutionassumption}.
\end{proof}
To verify Assumption~\ref{commutativeassumption}, we use Proposition~\ref{closedopenisomorphism} to work with the Lagrangian Floer-theoretic model of $R_{\phi}$.
\begin{lemma}\label{commutativeassumption--verified}
    If $\mathcal{X}\to\mathbb{D}$ is a toric degeneration with positive real gluing data subject to Assumption~\ref{toricdegenerationresolutionassumption}, then the smooth part $\mathcal{X}^{\circ}\to\mathbb{D}^{\circ}$ satisfies Assumption~\ref{commutativeassumption}.
\end{lemma}
\begin{proof}
    Let $L = L_0$. Consider the subcategory of $\Fuk(X_t,\mathbf{E}_t)$ whose objects are $\lbrace\phi^d(L)\rbrace_{d\in\mathbb{Z}}$. We claim that the anti-symplectic involution induces an involution
    \begin{align*}
        H^0(\Fuk(X_t,\mathbf{E}_t))\to H^0(\Fuk(X_t,\mathbf{E}_t))^{opp}
    \end{align*}
    which maps the cohomological (ordinary) category $H^0(\Fuk(X_t,\mathbf{E}_t))$ to its opposite category, and which acts via
    \begin{align*}
        \phi^d(L)\mapsto\phi^{-d}(L)
    \end{align*}
    on (isomorphism classes of) objects. The main point of this construction, taken from~\cite[(10b)]{SeiBook}, is that we need to augment $\Fuk(X_t,\mathbf{E}_t)$ by adjoining objects of the form $(\phi^d(L),\iota)$, for all $d$, and we declare that $\iota$ acts on this extended Fukaya category via
    \begin{align*}
        \iota\colon\phi^d(L)\mapsto(\iota\circ\phi^d(L),\iota)
    \end{align*}
    these additional objects allow us to choose perturbation data which is $\iota$-equivariant. We also use the fact that $\iota\circ\phi^d(L)$ and $\phi^{-d}(L)$ are always Hamiltonian isotopic, which is easy to see in the semistable charts $U_I\cap X_t$. Thus we obtain the desired action on objects at the cohomology level.

    Using this involution, we will show that
    \begin{align*}
        R_{\phi}\cong\bigoplus_{d=0}^{\infty}HF^0(L,\phi^d(L))
    \end{align*}
    is commutative. Given a pair of elements $a_i\in HF^0(L,\phi^{d_i}(L))$, any pseudoholomorphic disk contributing to the product $a_1\star a_2$ will have three boundary punctures and Lagrangian labeling $(L,\phi^{d_1}(L),\phi^{d_1+d_2}(L))$. Applying $\iota$ to this labeling gives the new Lagrangian labeling
    \[ ((\phi^{-d_1-d_2}(L),\iota),(\phi^{-d_1}(L),\iota),(L,\iota))\]
    where the order is determined by the fact that $\iota$ is anti-symplectic, and hence reverse the orientation of the disk. Applying $\phi^{d_1+d_2}$ to this boundary condition gives
    \[ ((L,\iota),(\phi^{d_2}(L),\iota),(\phi^{d_1+d_2}(L),\iota))\]
    which is isomorphic to
    \[ (L,\phi^{d_2}(L),\phi^{d_1+d_2}(L))\]
    in $H^0(\Fuk(X_t,\mathbf{E}_t))$. This gives a bijection of perturbed pseudoholomorphic disks contributing to the two products, which establishes commutativity of $R_{\phi}\otimes\Lambda_{\Bbbk}$, and it follows from standard facts in commutative algebra about localizations that $R_{\phi}$ is also commutative.
    \end{proof}

Putting all of this together, we conclude the proof of Theorem~\ref{toric-degeneration-hms-thm} from the introduction.
\begin{proof}[Proof of Theorem~\ref{toric-degeneration-hms-thm}]
    In this setting we have verified Assumption~\ref{commutativeassumption} (see Lemma~\ref{commutativeassumption--verified}) and Assumption~\ref{lagrangianassumption} (see Lemma~\ref{lagrangianassumption--verified}.) Thus the result follows from Theorem~\ref{mirrorinjectionthm}.
\end{proof}
The final details showing that the resolutions of Corollary~\ref{bb3corollary} satisfy Assumption~\ref{toricdegenerationresolutionassumption} are as follows.
\begin{proof}[Proof of Corollary~\ref{bb3corollary}]
In item (i), the technical MPCP condition shows that the Batyrev--Borisov degenerations constructed by Gross have simple intersection complexes~\cite{GrossBB}. In item (ii), simplicity is part of Assumption 6.1 in~\cite{Goncharov}. That both types of degeneration have resolutions of the type we need is the content of Proposition~\ref{resolutionassumption--verified}.
\end{proof}

The proof of Theorem~\ref{comparingmirrors} is a relatively easy application of the tailoring trick developed in~\cite{AbouzaidCoordinateRing}, given our description of the positive real locus above.
\begin{proof}[Proof of Theorem~\ref{comparingmirrors}]
    Let $\mathcal{X}\to\mathbb{D}$ be a Batyrev degeneration satisfying the hypotheses of the theorem. If we take a smooth fiber $X_t$ of $\mathcal{X}$ over a positive real point of $\mathbb{D}$ sufficiently close to its origin, then the positive real locus of Theorem~\ref{lagrangianpositivereallocus} is the positive real locus of Theorem~\ref{lagrangianpositivereallocus}. In particular, the positive real locus is one of the Lagrangian spheres described in~\cite[\S{4.3}]{GHHPS}.

    We will now use the tailoring method of~\cite{AbouzaidCoordinateRing} to compare the other Lagrangian spheres considered in~\cite{GHHPS} to the iterates of the positive real locus under the regularized monodromy $\phi\colon X_t\to X_t$. The compact Lagrangian submanifolds of an $(n+1)$-dimensional toric Fano variety considered in~\cite[\S{5}]{AbouzaidCoordinateRing} can be thought of as fiberwise translations of the positive real locus with respect to the (singular) moment fibration with iamge $\Delta$ cf.~\cite[(3-2)]{AbouzaidCoordinateRing}. We can deform the hypersurface $X_t$ as in~\cite[Proposition 4.2]{AbouzaidCoordinateRing} through real hypersurfaces. Let $M_{s,t}$ denote such a hypersurface. This uses the smoothness of $\Sigma^*$ as given in the theorem statement. 
    
    Such deformations can be chosen so that the resulting real hypersurface projects to an arbitrarily small neighborhood of a tropical hypersurface in $\Delta$. If we apply the ambient fiberwise translation symplectomorphism to such a hypersurface, it will be carried to another hypersurface in the same cohomology class. Thus we can deform this new hypersurface to the original one by a $C^1$-small isotopy with trivial flux, which we can then replace with a Hamiltonian isotopy. This determines a symplectic automorphism $\psi$ of $M_{s,t}$. Now by the description in~\cite{AbouzaidCoordinateRing}, $M_{s,t}$ will have the structure of a Lagrangian torus fibration over a(n arbitrarily large) subset of the smooth part of the tropical hypersurface. In this region, it is clear that $\phi$ and $\psi$ will agree with each other, and thus $\phi^d(L)$ and $\psi^d(L)$ will be $C^1$-close Lagrangian rational homology spheres. We can then use the same isotopy argument to show that these are equivalent objects of the Fukaya category.

    The proof of homological mirror symmetry given in~\cite{GHHPS} implies that
    \[\underline{\Proj}_{\Bbbk[\![T]\!]}HF^0(L,\psi^d(L))\]
    gives the Batyrev mirror space, which completes the proof of the theorem.
\end{proof} 

\part{Appendix}
\appendix
\section{A one-sided inverse functor}\label{onesidedappendix}
This appendix continues in the setting of \S{\ref{mainsection}}. The functor~\eqref{mirrorinjection} determines a dg-functor
\begin{align}\label{adjunctionfunctor}
\Mod\Fuk(X_t,\mathbf{E}_t)\to\Mod\Perf_{dg}(\check{\mathcal{X}})
\end{align}
of dg-categories of $A_{\infty}$-modules. Explicitly, one can restrict any module over $\Fuk(X_t,\mathbf{E}_t)$ to a module over the full subcategory with objects $\lbrace\mathcal{L}^{\otimes d}\rbrace_{d\in\mathbb{Z}}$, which is quasi-equivalent to $\Perf_{dg}(\check{\mathcal{X}})$. We will show that this gives rise to a quasi-inverse of~\eqref{mirrorinjection}. 

As recalled above, it is shown in~\cite[Theorem 4]{Orlov} that the line bundles $\mathcal{L},\ldots,\mathcal{L}^{\otimes(n+1)}$ split-generate $\Perf_{dg}(\check{\mathcal{X}})$. Define the object $\mathcal{G}\in\Perf_{dg}(\check{\mathcal{X}})$ as the direct sum $\mathcal{G}\coloneqq\mathcal{L}\oplus\cdots\oplus\mathcal{L}^{\otimes(n+1)}$ of these generators, and let $\mathcal{P}$ denote the endomorphism algebra $\mathcal{P}\coloneqq\Ext^*(\mathcal{G},\mathcal{G})$. From this the functor~\eqref{adjunctionfunctor} reduces to a functor
\begin{align}\label{adjunectionfunctor1}
\Mod\Fuk(X_t)\to\Mod\End\mathcal{P} \,.
\end{align}
Note that this functor needs not be homologically faithful. We will characterize the perfect modules of $\Fuk(X_t)$ inside $\Mod\End\mathcal{P}$.

Our starting point will be the following result of Schwede--Shipley.
\begin{theorem}[{\cite[Theorem 3.1.1]{SS03}}]\label{schwedeshiplythm}
Let $\mathcal{C}$ be a simplicial, cofibrantly generated, proper stable model category with a compact generator $\mathcal{P}$. Then there exists a chain of simplicial Quillen equivalences:
\begin{align}
\mathcal{C}\simeq_Q\Mod\End\mathcal{P} \,.
\end{align}
\end{theorem}
Derived categories of quasi-coherent sheaves are discussed in Example 2.4(iv) of \emph{op. cit.} Note also that every cofibrantly generated, proper, stable model category is Quillen equivalent to a simplicial model category~\cite[p. 112]{SS03}. We will briefly review how Schwede--Shipley's theorem applies in our situation, following~\cite[\S{8.3}]{Toe07}.

Recalling that $\check{\mathcal{X}}$ is a quasi-compact separated scheme over $\Spec\Lambda_{\Bbbk}$, consider $\mathrm{QCoh}(\mathcal{\check{X}})$, the category of $\mathbb{U}$-small quasi-coherent sheaves on $\mathcal{\check{X}}$, where $\mathbb{U}$ is some universe. There is a cofibrantly generated model category $C(\mathrm{QCoh}(\mathcal{\check{X}}))$ of unbounded complexes of quasi-coherent sheaves on $\mathcal{\check{X}}$. Here the cofibrations are precisely the monomorphisms and the equivalences precisely are the quasi-isomorphisms. There is a natural $C(\Lambda_{\Bbbk})_{\mathbb{U}}$-enrichment of $C(\mathrm{QCoh}(X))$ which is a $C(\Lambda_{\Bbbk})_{\mathbb{U}}$-model category. For some universe $\mathbb{V}$ such that $\mathbb{U}\in\mathbb{V}$, one can construct (see~\cite[p. 629]{Toe07}) a $\mathbb{V}$-small dg-category $\mathrm{Int}(C(\mathrm{QCoh}(\check{\mathcal{X}}))$ denoted $L_{qcoh}(\check{\mathcal{X}})$ in~\cite{Toe07}. Finally, the homotopy category of $L_{qcoh}(\mathcal{\check{X}})$ is naturally equivalent to the unbounded derived category $D(\mathrm{QCoh}(\mathcal{\check{X}}))$ of quasi-coherent sheaves.\footnote{Since $\check{\mathcal{X}}$ is separated, not merely quasi-separated, there is a natural equivalence between $D(\mathrm{QCoh}(\mathcal{\check{X}}))$ and $D_{\mathrm{QCoh}}(\mathcal{\check{X}})$, the category of complexes with cohomology sheaves in $\mathrm{QCoh}(\check{\mathcal{X}})$.}

Theorem~\ref{schwedeshiplythm} implies that we have a Quillen equivalence
\begin{align*}
C(\mathrm{QCoh}(\mathcal{\check{X}}))\simeq\Mod\End\mathcal{P}.
\end{align*}
As explained in \textit{loc. cit.}, an object of $L_{qcoh}(\mathcal{\check{X}})$ is perfect (i.e. homotopically finitely-presented) in the sense of~\cite[\S{7}]{Toe07} if and only if it is a compact object of $D(\mathrm{QCoh}(\mathcal{\check{X}}))$, if and only if it is a perfect complex on $\mathcal{\check{X}}$. Also, by~\cite[Lemma 8.10]{Toe07}, we do not need to distinguish between left and right $A_{\infty}$-modules.

Using the Calabi--Yau structure on $\Fuk(X_t)$, we will show that
\begin{theorem}
Any element in the image of $\Perf\Fuk(X_t)$, under~\eqref{adjunectionfunctor1}, is homotopically finitely-presented and is, consequently, lies in $\Perf_{dg}(\check{\mathcal{X}})$.
\end{theorem}
\begin{proof}
Recall that the homotopy category of $L_{qcoh}(\check{X}_0)$ is the derived category of graded $R_\phi$-modules, denoted $D(R_\phi)$. All modules arising in the course of the proof are assumed to be graded modules. By the above discussion and standard facts about perfect complexes~\cite[Lemma 15.76.3]{StacksProject}, it suffices to check that the image of any such object in $D(R_\phi)$ is (i) finitely-generated and (ii) has finite projective dimension. This uses the fact that $\Lambda_{\Bbbk}$ is a Noetherian ring (it is a field.) Let $K$ be an object of $\Perf\Fuk(X_t)$. We will abuse notation and also use $K$ to refer to the image of this object in $\Mod\End\mathcal{P}$ and in the derived category $D(R_{\phi})$ interchangeably.

For (i), note that the $\mathcal{P}$-module morphisms from $K$ to $\mathcal{P}$ are naturally identified with Floer cohomology groups $HF^*(K,L_i)$, for $i = 0,\ldots,d+1$. Since the objects of $\Fuk(X_t,D_t)$ are \textit{compact} Lagrangians, the space of such morphisms are finitely-generated as $\Bbbk$-vector spaces. From this it is easy to see that there is a morphism $\mathcal{P}^{\oplus N}\to K$ descends to a surjection in $\mathrm{QCoh}(\check{X}_0)$, for some $N>0$. This implies that the image of $K$ in $D(R_0)$ is finitely-generated. In particular, $K$ is a \textit{coherent} $R_{\phi}$-module.

For (ii), we consider the Yoneda product
\begin{align*}
\Ext_{R_{\phi}}^*(K,M)\otimes\Ext_{R_{\phi}}^{n-*}(M,K)\to\Ext_{R_{\phi}}^n(K,K)
\end{align*}
where $M$ is an arbitrary finitely-generated $R_{\phi}$-module (i.e. a coherent sheaf). Since $K$ is a compact Lagrangian, we can identify the codomain of this pairing with $HF^n(K,K) = \Lambda_{\Bbbk}$. Since $\Ext_{R_{\phi}}^*(K,-)$ and $\Ext_{R_{\phi}}^{n-*}(-,K)$ are both $\delta$-functors, the Yoneda product induces a natural transformation
\begin{align}\label{functoriso}
\Ext^*(K,-)\to\Ext^{n-*}(-,K)
\end{align}
which is in fact an isomorphism when restricted to finite rank free $R_{\phi}$-modules, which we see by interpreting these Ext groups as Floer cohomology groups, and using the weak Calabi--Yau structure. Since we can resolve a finitely-generated module $M$ by a possibly infinite complex of finite-rank free $R_{\phi}$-modules, we see that this natural transformation is in fact a natural isomorphism of functors on $D(R_{\phi})$. Using this isomorphism, it follows from~\cite[Definition 15.70.1 Lemma 15.70.2]{StacksProject} that $K$ has finite projective dimension, since our homological computations have shown that it has finite projective amplitude. In more detail, ~\eqref{functoriso} should be thought of as a Serre duality statement for $K$, which holds even though we have not shown that $\check{\mathcal{X}}$ is smooth. Because the right-hand side vanishes when $n-*<0$, which follows because $\Hom(M,M'[k]) = 0$ for any $k<0$ when $M$ and $M'$ are both coherent sheaves on $\check{\mathcal{X}}$, so the duality statement allows us to apply the results of~\cite[Definition 15.70.1 Lemma 15.70.2]{StacksProject}.
\end{proof}
\begin{corollary}
    The functor~\eqref{adjunectionfunctor1} determines a one-sided inverse of the functor described in Theorem~\ref{mirrorinjectionthm}.
\end{corollary}

\begin{proof}[Proof of Proposition~\ref{finitelygeneratedprop}]
    This follows from a classical result due to Serre, which identifies the category of perfect complexes on $\check{\mathcal{X}}$ with a quotient of the category of \textit{finitely-generated} $R_{\phi}$-modules.
\end{proof}
\begin{remark}
This argument uses the fact that Floer cohomology groups are finite-dimensional in an essential way, and so does not admit a readily obvious generalization to the wrapped Fukaya category.
\end{remark}

\section{Singularities of the central fiber}\label{slcappendix}
This appendix contains the proof of Lemma~\ref{slclemma}. We will closely follow~\cite{Oldfield}, which was written for (Spec of) the Stanley--Reisner ring of a log Calabi--Yau pair $(X,\Delta)$.

Continuing with the notation of \S{\ref{semistablereductionsection}}, let $\mathcal{X}^{\circ}$ denote a maximally unipotent family of smooth closed Calabi--Yau $n$-folds (in the strict sense) over $\mathbb{D}^{\circ}$. Choose an snc model $\widetilde{\mathcal{X}}\to\mathbb{D}$ as in~\ref{gsresolutionmodel} with central fiber $\widetilde{X}_0$. Recall that we require $\widetilde{X}_0$ to be a reduced snc divisor. Furthermore, after possibly applying additional blowups so that all intersections between components of $\widetilde{X}_0$ are irreducible, we can assume that the skeleton $\Sk(\mathcal{X}^{\circ})$ is a simplicial complex.

Recall from~\cite[4.18]{KollarBook} that for any dlt pair $(X,\Delta)$ and any stratum $W$ of $(X,\Delta)$, there is a canonically defined \textit{different} of $\Delta$ on $W$, which is a rational divisor $\Diff_W^*$ on $X$ such that
\begin{align*}
(K_X+\Delta)\mid_W\equiv K_W+\Diff_W^*
\end{align*}
which allows us to consider the dlt log CY pair $(W,\Diff_W^*)$~\cite[4.19]{KollarBook}.

Recall that if $\Sigma$ is a simplicial complex (thought of as its set of faces) and $\sigma$ is a face of $\Sigma$, then the \textit{link} of $\sigma$ is
\begin{align*}
\mathrm{link}_{\Sigma}(\sigma)\coloneqq\lbrace\tau\in\Sigma\colon\sigma\cup\tau\in\Sigma\text{ but }\sigma\cap\tau = \emptyset\rbrace \,.
\end{align*}
\begin{lemma}
Let $W$ be a (good) stratum of $\widetilde{X}_0$, and denote by $\sigma_W$ the corresponding face in $\Sk(\mathcal{X}^{\circ})$. Then there is an identification of the dual complex of $\Diff_W^*$ with the link of $\sigma_W$ in $\Sk(\mathcal{X}^{\circ})$.
\end{lemma}
\begin{proof}
We can interpret $\Diff^*_W$ as the union of irreducible components of $\widetilde{X}_0$ that intersect $W$, but which do not contain $W$. Hence a $k$-cell in the dual complex of $\Diff_W^*$ corresponds to a $k+1$ good components of $\widetilde{X}_0$ not containing $W$ but that intersect inside $W$. This determines a correspondence between the cells in the dual complex of $\Diff_W^*$ and the cells in $\Sk(\mathcal{X}^{\circ})$ which are disjoint from $\sigma_W$, but whose union with $\sigma_W$ is also a cell. The latter collection of cells is $\mathrm{link}_{\Sk(\mathcal{X}^{\circ})}(\sigma_W)$.
\end{proof}

\begin{proof}[Proof of Lemma~\ref{slclemma}]
First we will show that $\check{X}_0\coloneqq\Proj\SR(\mathcal{X}^{\circ})$ is Cohen--Macaulay. The previous lemma allows us to identify the link $\mathrm{link}_{\Sk(\mathcal{X}^{\circ})}(\sigma_W)$ of any face $\sigma_W$ of $\Sk(\mathcal{X}^{\circ})$ with the dual complex of some log CY pair. By~\cite[Prop. 31]{KollarXu}, this has the rational homology of a $n$-sphere. Therefore the Cohen--Macaulay condition for $\check{X}_0$ follows from Reisner's criterion~\cite[Theorem 1]{ReisnersRef}.

To show that $\check{X}_0$ is slc, it suffices to show that its normalization with a suitable choice of boundary divisor is log-canonical. Since the normalization is just a disjoint union of copies of $\mathbb{P}^n$, and the conductor subscheme of the normalization is just the union of toric boundary divisors in each component, this follows immediately from the fact that toric pairs are log-canonical.

It remains to show that the canonical bundle on $\check{X}_0$ is trivial. Since we have already shown that $\check{X}_0$ is Cohen--Macaulay, it follows that its dualizing complex $\omega_{\check{X}_0}$ is a sheaf on $\check{X}_0$. Since $\Sk(\mathcal{X}^{\circ})$ is a pseudo-manifold, it follows that every $(n-1)$-dimensional toric boundary component of each component of the normalization of $\check{X}_0$ is contained in precisely two copies of $\mathbb{P}^n$. It suffices to construct a section of $\omega_{\check{X}_0}$ away from a codimension $2$ subset: the complement $U$ of the set of points contained in more than two irreducible components of $\check{X}_0$. Since $\check{X}_0$ is Cohen--Macaulay, it satisfies Serre's $S_2$ criterion, which allows us to extend such a section over this codimension $2$ subset.

A section of $\omega_{\check{X}_0}\mid_U$ corresponds to a collection of $n$-forms on each component of the normalization of $\check{X}_0$ with at worse poles along each coordinate hyperplane, so that when two components $U_1$ and $U_2$ intersect, the residues along $U_1\cap U_2$ sum to $0$. If we start with the standard toric volume form on each component, we need to show that a sign for the form on each component can be chosen so that the signs of any two intersecting components are opposite. This amounts to choosing an orientation for $\Sk(\mathcal{X}^{\circ})$, which is possible since it has the rational homology of $S^n$ by~\cite{KollarXu}.
\end{proof}

\bibliography{references}
\bibliographystyle{alpha}

\end{document}